\documentclass[12pt]{article}

\usepackage{amsmath}
\usepackage{amsfonts}
\usepackage{amssymb}
\usepackage{wasysym}
\usepackage{amsthm}
\usepackage[margin=1in]{geometry}

\title{Non-negative integral level affine Lie algebra tensor categories and their associativity isomorphisms}
\author{Robert McRae\\
\small \it 
Beijing International Center for Mathematical Research\\
\small \it Peking University, Beijing, China 100084\\
\small \textit{E--mail address:} \texttt{robertmacrae@math.pku.edu.cn}}
\date{}

    \theoremstyle{definition}\newtheorem{rema}{Remark}[section]
    \theoremstyle{plain}\newtheorem{propo}[rema]{Proposition}
    \newtheorem{theo}[rema]{Theorem}
    \theoremstyle{definition}\newtheorem{defi}[rema]{Definition}
    \theoremstyle{plain}\newtheorem{lemma}[rema]{Lemma}
    \newtheorem{corol}[rema]{Corollary}
    \theoremstyle{definition}\newtheorem{exam}[rema]{Example}
    \theoremstyle{definition}

\begin{document}

\newcommand{\N}{\mathbb{N}}
\newcommand{\Z}{\mathbb{Z}}
\newcommand{\Q}{\mathbb{Q}}
\newcommand{\R}{\mathbb{R}}
\newcommand{\C}{\mathbb{C}}
\newcommand{\gvmu}{V_{\widehat{\mathfrak{g}}}(\ell,U)}
\newcommand{\imu}{L_{\widehat{\mathfrak{g}}}(\ell,U)}
\newcommand{\gvmzero}{V_{\widehat{\mathfrak{g}}}(\ell,0)}
\newcommand{\imzero}{L_{\widehat{\mathfrak{g}}}(\ell,0)}
\newcommand{\gvmlambda}{V_{\widehat{\mathfrak{g}}}(\ell,L_\lambda)}
\newcommand{\imlambda}{L_{\widehat{\mathfrak{g}}}(\ell,L_\lambda)}
\newcommand{\Wo}{W^{(1)}}
\newcommand{\Wt}{W^{(2)}}
\newcommand{\Wth}{W^{(3)}}
\newcommand{\Wi}{W^{(i)}}
\newcommand{\wo}{w_{(1)}}
\newcommand{\wt}{w_{(2)}}
\newcommand{\wth}{w_{(3)}}
\newcommand{\wi}{w_{(i)}}
\newcommand{\g}{\mathfrak{g}}
\newcommand{\ghat}{\widehat{\mathfrak{g}}}
\newcommand{\ilambdaone}{L_{\widehat{\mathfrak{g}}}(\ell,\lambda_1)}
\newcommand{\ilambdatwo}{L_{\widehat{\mathfrak{g}}}(\ell,\lambda_2)}
\newcommand{\ilambdathree}{L_{\widehat{\mathfrak{g}}}(\ell,\lambda_3)}
\newcommand{\imuone}{L_{\widehat{\mathfrak{g}}}(\ell,U_1)}
\newcommand{\imutwo}{L_{\widehat{\mathfrak{g}}}(\ell,U_2)}
\newcommand{\imuthree}{L_{\widehat{\mathfrak{g}}}(\ell,U_3)}
\newcommand{\T}{T^{(\ell)}_{\lambda_1, \lambda_2}}
\newcommand{\ittensprod}{(W_1\boxtimes_{P(z_1-z_2)} W_2)\boxtimes_{P(z_2)} W_3}
\newcommand{\prodtensprod}{W_1\boxtimes_{P(z_1)} (W_2\boxtimes_{P(z_2)} W_3)}

\maketitle

\newcommand{\nordcirc}{\mbox{\small $\genfrac{}{}{0pt}{}{\circ}{\circ}$}}
\numberwithin{equation}{section}

\begin{abstract}
\noindent For a finite-dimensional simple Lie algebra $\g$, we use the vertex 
tensor category theory of Huang and Lepowsky to identify the category of 
standard modules for the affine Lie algebra $\ghat$ at a fixed level $\ell\in\N$ 
with a certain tensor category of finite-dimensional $\g$-modules. More 
precisely, the category of level $\ell$ standard $\ghat$-modules is the module 
category for the simple vertex operator algebra $\imzero$, and as is well known, 
this category is equivalent as an abelian category to $\mathbf{D}(\g,\ell)$, the 
category of finite-dimensional modules for the Zhu's algebra $A(\imzero)$, which 
is a quotient of $U(\g)$. Our main result is a direct construction using 
Knizhnik-Zamolodchikov equations of the associativity isomorphisms in 
$\mathbf{D}(\g,\ell)$ induced from the associativity isomorphisms constructed by 
Huang and Lepowsky in $\imzero-\mathbf{mod}$. This construction shows that 
$\mathbf{D}(\g,\ell)$ is closely related to the Drinfeld category of 
$U(\g)[[h]]$-modules used by Kazhdan and Lusztig to identify categories of 
$\ghat$-modules at irrational and most negative rational levels with categories 
of quantum group modules.
\end{abstract}

\tableofcontents

\section{Introduction}

Suppose $\g$ is a finite-dimensional simple Lie algebra over $\C$; then the 
affine Lie algebra $\ghat$ is a central extension of the loop algebra 
$\g\otimes\C[t, t^{-1}]$ by a one-dimensional space $\C\mathbf{k}$. If 
$\mathbf{k}$ acts on a $\ghat$-module by a scalar $\ell\in\C$, we say that the 
module has level $\ell$. Categories of $\ghat$-modules at fixed non-negative 
integral levels are particularly important in physics, since they correspond to 
WZNW models, important examples of rational conformal field theories. The study 
of conformal field theory by physicists, especially by Moore and Seiberg, 
predicted that the category of standard (that is, integrable highest weight) 
$\ghat$-modules at a fixed level $\ell\in\N$ should have the structure of a 
rigid braided tensor category (see for instance \cite{BPZ}, \cite{KZ}, and 
\cite{MS}). Indeed, in \cite{KL1} and \cite{KL2}, Kazhdan and Lusztig showed 
that when $\ell\notin\Q$, or when $\ell\in\Q$ and $\ell<-h^\vee$ where $h^\vee$ 
is the dual Coxeter number of $\g$, a certain category of $\ghat$-modules of 
level $\ell$ has a natural braided tensor category structure, and they proved 
rigidity for most of these tensor categories in \cite{KL4}. However, their 
constructions do not apply to the case $\ell\in\N$.

There are several approaches to obtaining tensor category structure when 
$\ell\in\N$. First, motivated by Kazhdan and Lusztig's constructions, Huang and 
Lepowsky developed a general tensor product theory for the category of modules 
for a vertex operator algebra in \cite{HL1}-\cite{HL3} and \cite{H1}. Since the 
category of standard $\ghat$-modules at a fixed level $\ell\in\N$ is the module 
category for a simple vertex operator algebra $\imzero$ (\cite{FZ}), they were 
able to use this theory in \cite{HL4} to prove that this category has natural 
braided tensor category structure; rigidity, and indeed modularity, of this 
braided tensor category was proved in \cite{H4}. More recently, in 
\cite{HLZ1}-\cite{HLZ8}, Huang, Lepowsky, and Zhang have developed a more 
general logarithmic tensor category theory for so-called generalized modules for 
a vertex operator algebra. Using this theory, Zhang showed in \cite{Zha} that 
the braided tensor categories of Kazhdan and Lusztig when $\ell\notin\Q$ or 
$\ell\in\Q_{<-h^\vee}$ agree with the vertex algebraic braided tensor categories 
of certain generalized $\imzero$-modules.

Another approach to the $\ell\in\N$ case using ideas of Beilinson, Feigin, and 
Mazur yields braided tensor category structure (see for instance Chapter 7 of 
\cite{BK}), but not rigidity. A third approach by Finkelberg in \cite{Fi1}, 
\cite{Fi2} involves transferring Kazhdan and Lusztig's constructions at negative 
level to the positive level categories. This work also requires the Verlinde 
formula for multiplicities of irreducible modules in tensor products, which was 
proved independently by Faltings \cite{F} and Teleman \cite{T} for 
$\ghat$-modules and was proved by Huang \cite{H3} in a general vertex algebraic 
context. Finkelberg's work does not apply to the cases $E_6$ level $1$, $E_7$ 
level $1$, and $E_8$ levels $1$ and $2$ because Kazhdan and Lusztig did not 
prove rigidity for the corresponding negative level categories.

All the constructions of tensor category structure on $\imzero-\mathbf{mod}$, 
that is, the category of standard $\ghat$ modules at level $\ell\in\N$, are 
complicated by the fact that the usual vector space tensor product of two 
modules does not have a natural module structure. Note, for instance, that the 
usual Lie algebra tensor product of $\ghat$-modules does not preserve level. 
This in turn means the associativity isomorphisms in this tensor category are 
highly non-trivial (see for instance \cite{H1} and \cite{HLZ6}). As a result, 
useful, explicit descriptions of the tensor category structure on 
$\imzero-\mathbf{mod}$ are missing from the literature. In this paper, we use 
the vertex tensor category theory of Huang and Lepowsky to give an explicit 
description of the tensor category $\imzero-\mathbf{mod}$ when $\ell\in\N$, in 
particular a description of the associativity isomorphisms. More precisely, we 
show that there is an equivalence between $\imzero-\mathbf{mod}$ and an explicit 
tensor category $\mathbf{D}(\g,\ell)$ of finite-dimensional $\g$-modules. We 
expect that this description will be useful for obtaining a uniform proof of the 
braided tensor equivalence between $\imzero-\mathbf{mod}$ and a category of 
quantum group modules, which we recall now.

In addition to constructing braided tensor categories of $\ghat$-modules, 
Kazhdan and Lusztig showed in \cite{KL3} and \cite{KL4} that their category at 
level $\ell$ is equivalent to the braided tensor category of finite-dimensional 
modules for the quantum group $U_q(\g)$, where $q=e^{\pi i/m(\ell+h^\vee)}$; 
here $m$ is the ratio of the squared length of the long roots of $\g$ to the 
squared length of the short roots. Then Finkelberg's work in \cite{Fi1}, 
\cite{Fi2} showed that, with the possible exceptions of $E_6$ level $1$, $E_7$ 
level $1$, and $E_8$ levels $1$ and $2$, the category of standard 
$\ghat$-modules of level $\ell\in\N$ is equivalent to a certain semisimple 
subquotient of the corresponding category of finite-dimensional quantum group 
modules. As mentioned above, the exceptions exist because of the use of Kazhdan 
and Lusztig's constructions at negative levels. Thus to prove the equivalence 
between $\imzero-\mathbf{mod}$ when $\ell\in\N$ and the quantum group category 
with no exceptions, we need an explicit description of the (rigid) tensor 
category structure at non-negative level.

We now discuss the results of this paper in more detail. Since 
$\imzero-\mathbf{mod}$ when $\ell\in\N$ is semisimple, it follows from \cite{Z} 
that there is an equivalence of abelian categories between 
$\imzero-\mathbf{mod}$ and the category of finite dimensional modules for the 
Zhu's algebra $A(\imzero)$ which takes an irreducible $\imzero$-module to its 
lowest conformal weight space. Since $A(\imzero)\cong U(\g)/\langle 
x_\theta^{\ell+1}\rangle$ where $x_\theta$ is a root vector corresponding the 
longest root $\theta$ of $\g$ by \cite{FZ}, $\imzero-\mathbf{mod}$ is equivalent 
as an abelian category to the category $\mathbf{D}(\g,\ell)$ whose objects are 
finite-dimensional $\g$-modules on which $x_\theta^{\ell+1}$ acts trivially. 
Then we can use this equivalence to transfer the tensor category structure on 
$\imzero-\mathbf{mod}$ constructed by Huang and Lepowsky to 
$\mathbf{D}(\g,\ell)$, and it remains to give an explicit description of the 
tensor products, unit isomorphisms, and associativity isomorphisms thus induced 
in $\mathbf{D}(\g,\ell)$.

The tensor products and unit isomorphisms in $\mathbf{D}(\g,\ell)$ are 
straightforward to describe. Since tensor products of vertex operator algebra 
modules are defined using a universal property involving intertwining operators 
(in analogy with the definition of a tensor product of vector spaces in terms of 
a universal property involving bilinear maps; see for instance the introduction 
to \cite{HLZ1}), the description of the space of intertwining operators among 
irreducible $\imzero$-modules given in \cite{FZ} allows us to identify the 
tensor product of two modules $U_1$ and $U_2$ in $\mathbf{D}(\g,\ell)$ as a 
certain quotient $U_1\boxtimes U_2$ of the usual tensor product $\g$-module. 
Such a quotient is necessary because $\mathbf{D}(\g,\ell)$ is not closed under 
the usual tensor product of $\g$-modules. We remark that we could obtain the 
tensor product $U_1\boxtimes U_2$ in $\mathbf{D}(\g,\ell)$ simply by taking a 
direct sum of irreducible $\g$-modules with multiplicities calculated using the 
Verlinde formula (\cite{F}, \cite{T}, \cite{H3}) or using a result such as 
Theorem 6.2 in \cite{FF}. However, this would be less natural than our approach 
because it would leave the relation of $U_1\boxtimes U_2$ to $U_1\otimes U_2$ 
unclear, and it would make it more difficult to understand the associativity 
isomorphisms in $\mathbf{D}(\g,\ell)$. The unit object of 
$\mathbf{D}(\g,\ell)$ is the trivial one-dimensional $\g$-module 
$\C\mathbf{1}$, 
and the unit isomorphisms are obvious.

Most of the work in this paper focuses on describing the associativity 
isomorphisms in $\mathbf{D}(\g,\ell)$. They cannot be trivial because if $U_1$, 
$U_2$, and $U_3$ are modules in $\mathbf{D}(\g,\ell)$, $U_1\boxtimes 
(U_2\boxtimes U_3)$ and $(U_1\boxtimes U_2)\boxtimes U_3$ are typically 
different (albeit isomorphic) quotients of $U_1\otimes U_2\otimes U_3$. Our main 
result is that the associativity isomorphisms in $\mathbf{D}(\g,\ell)$ come from 
solutions of Knizhnik-Zamolodchikov (KZ) equations (\cite{KZ}), as in a category 
of modules for the trivial deformation $U(\g)[[h]]$ of the universal 
enveloping algebra of $\g$, where $h$ is a formal variable, constructed by 
Drinfeld (\cite{D1}, \cite{D2}, \cite{D3}; see also \cite{BK}, \cite{Ka}). 
(Drinfeld's category is equivalent to a category of modules for the formal 
quantum group $U_{h}(\g)$, and Kazhdan and Lusztig used an explicit 
equivalence between these categories in \cite{KL3}, \cite{KL4} to show the 
equivalence between their category of $\ghat$-modules and the category of 
$U_q(\g)$-modules.)

To describe the associativity isomorphisms in $\mathbf{D}(\g,\ell)$, consider 
objects $U_1$, $U_2$, and $U_3$ in $\mathbf{D}(\g,\ell)$ and the one-variable KZ 
equation
\begin{equation}\label{introKZ} 
(\ell+h^\vee)\dfrac{d}{dz}\varphi(z)=\left(\dfrac{\Omega_{1,2}}{z}-\dfrac{
\Omega_{2,3}}{1-z}\right)\varphi(z)
\end{equation}
where $\varphi(z)$ is a $(U_1\otimes U_2\otimes U_3)^*$-valued analytic 
function, and $\Omega_{1,2}$, $\Omega_{2,3}$ are certain (non-commuting) 
operators on $(U_1\otimes U_2\otimes U_3)^*$. In the case $\ell\notin\Q$, any 
solution $\varphi(z)$ of \eqref{introKZ} can be expressed as
\begin{equation*}
 \varphi(z)=z^{\Omega_{1,2}/(\ell+h^\vee)}\cdot\varphi_0(z)
\end{equation*}
around $z=0$, where $\varphi_0(z)$ is analytic in a neighborhood of $0$, and as
\begin{equation*}
 \varphi(z)=(1-z)^{\Omega_{2,3}/(\ell+h^\vee)}\cdot\varphi_1(z)
\end{equation*}
around $z=1$, where $\varphi_1(z)$ is analytic in a neighborhood of $1$. The 
solution $\varphi(z)$ is completely determined by the initial value 
$\varphi_0(0)$, and also by the initial value $\varphi_1(1)$. Then there is a 
unique automorphism $\Phi_{KZ}$ of $(U_1\otimes U_2\otimes U_3)^*$, called the 
Drinfeld associator, that maps the initial value $\varphi_0(0)$ to 
$\varphi_1(1)$ for any solution $\varphi(z)$ of \eqref{introKZ}. In the case 
$\ell\in\Q$, the situation is not so simple because expansions of solutions to 
\eqref{introKZ} about the singularities $0$ and $1$ may contain logarithms. 
However, we show that series solutions around the singularities remain 
determined by initial data from $(U_1\otimes U_2\otimes U_3)^*$, and we 
construct a Drinfeld associator $\Phi_{KZ}$ that maps the initial datum for any 
solution at the singularity $0$ to the initial datum at $1$.

Our main theorem is that when $\ell\in\N$, the associativity isomorphisms in 
$\mathbf{D}(\g,\ell)$ are induced by adjoints of Drinfeld associators. In 
particular, the adjoint $\Phi_{KZ}^*$ of $\Phi_{KZ}$ induces a well-defined 
isomorphism between the two quotients $U_1\boxtimes(U_2\boxtimes U_3)$ and 
$(U_1\boxtimes U_2)\boxtimes U_3$ of $U_1\otimes U_2\otimes U_3$. This assertion 
is not obvious from the construction of the Drinfeld associator. Rather, it 
follows from the existence of associativity isomorphisms in 
$\imzero-\mathbf{mod}$ proven in \cite{HL4}, which is equivalent to the 
convergence and associativity of intertwining operators among $\imzero$-modules 
(see \cite{H1} or \cite{HLZ6}). The associativity of intertwining operators 
follows from the fact (first shown in \cite{KZ}) that a product of intertwining 
operators
\begin{equation}\label{product}
 u_{(1)}\otimes u_{(2)}\otimes u_{(3)}\mapsto\langle 
u_{(4)}',\mathcal{Y}_1(u_{(1)},1)\mathcal{Y}_2(u_{(2)},1-z)u_{(3)}\rangle,
\end{equation}
when $u_{(1)}$, $u_{(2)}$, $u_{(3)}$, and $u_{(4)}'$ are lowest-conformal-weight 
vectors of their respective modules, defines a solution of \eqref{introKZ} 
expanded about the singularity $z=1$, while an iterate of intertwining operators
\begin{equation}\label{iterate}
 u_{(1)}\otimes u_{(2)}\otimes u_{(3)}\mapsto\langle 
u_{(4)}',\mathcal{Y}^1(\mathcal{Y}^2(u_{(1)},z)u_{(2)},1-z)u_{(3)}\rangle
\end{equation}
corresponds to a solution of \eqref{introKZ} expanded about the singularity 
$z=0$. From this, the definitions imply that the Drinfeld associator maps the 
initial datum of an iterate functional to the initial datum of a corresponding 
product functional. Our main theorem then follows from the (non-trivial) fact 
that the initial data determining the series expansions \eqref{product} and 
\eqref{iterate} are given by replacing all intertwining operators with their 
projections to the lowest conformal weight spaces of their target modules.

Our description of the associativity isomorphisms in $\mathbf{D}(\g,\ell)$ might 
be expected since a similar result holds for the Kazhdan-Lusztig category of 
$\ghat$-modules when $\ell\notin\Q$ (see for instance the discussion in Section 
1.4 of \cite{BK}). However, there are several complications that make the case 
$\ell\in\N$ more diffficult. First, it is somewhat more complicated to construct 
Drinfeld associators when $\ell\in\Q$ because expansions of solutions to the KZ 
equation \eqref{introKZ} about its singularities typically contain logarithms. 
Because of this, it is more difficult to identify the initial data at the 
singularities determining a KZ solution corresponding to a product or iterate of 
intertwining operators. Identifying these initial data requires a theorem 
restricting the weights of irreducible standard $\ghat$-modules. Moreover, we 
use the vertex tensor category structure on $\imzero-\mathbf{mod}$ from 
\cite{HL4} to show that adjoints of Drinfield associators are isomorphisms 
between the correct quotients of triple tensor products of $\g$-modules in 
$\mathbf{D}(\g,\ell)$. It seems to be difficult to prove directly that 
$\mathbf{D}(\g,\ell)$, with its correct tensor product, is a tensor category, 
without using the tensor category structure on $\imzero-\mathbf{mod}$.

Although we focus in this paper on tensor categories of non-negative 
integral level affine Lie algebra modules, it is interesting to consider whether 
the methods and results here extend to $\ell\notin\N$. First, we remark 
that our methods and results easily extend to the case of generic level, 
$\ell\notin\Q$. Here, we consider the semisimple category generated by irreducible $\imzero$-modules, called $\mathcal{O}_{\ell+h^\vee}$ in \cite{KL1}-\cite{KL4}. It is easy to use results in \cite{FZ} and \cite{Li2} and the methods of this paper to show that the vertex algebraic tensor category structure on $\mathcal{O}_{\ell+h^\vee}$ constructed in \cite{Zha} based on \cite{HLZ1}-\cite{HLZ8} is equivalent to a tensor category whose objects consist of all finite dimensional $\g$-modules, whose tensor product is the usual tensor product of $\g$-modules, and whose associativity isomorphisms are obtained from Drinfeld associators constructed from solutions of KZ equations. These results may be compared with the tensor category structure on $\mathcal{O}_{\ell+h^\vee}$ constructed in \cite{KL1}-\cite{KL4} (see also Section 1.4 of \cite{BK}).

For the case $\ell\in\Q_{<-h^\vee}$, our methods and results do not immediately 
generalize to yield a description of the tensor category structure on the 
category $\mathcal{O}_{\ell+h^\vee}$ of $\ghat$-modules considered in 
\cite{KL1}-\cite{KL4} and \cite{Zha}. The main problem is that in this case, 
modules in $\mathcal{O}_{\ell+h^\vee}$ are not always generated by their lowest 
conformal weight spaces, so $\mathcal{O}_{\ell+h^\vee}$ is not equivalent to the 
category of finite-dimensional $A(\imzero)$-modules. Moreover, KZ equations 
are no longer sufficient to describe the associativity isomorphisms. It may be 
possible to obtain an equivalence between $\mathcal{O}_{\ell+h^\vee}$ and a 
category of finite-dimensional modules for one of the algebras $A_N(\imzero)$, $N\in\Z_+$, constructed in \cite{DLM} generalizing Zhu's algebra, and it may be 
possible to obtain a description of the tensor product in 
$\mathcal{O}_{\ell+h^\vee}$ using the description of the space of (logarithmic) 
intertwining operators among a triple of generalized modules for a vertex 
operator algebra from \cite{HY}. However, the algebras $A_N(V)$ for a vertex 
operator algebra $V$ are typically very difficult to calculate explicitly. 
Alternatively, it may be possible to use the methods of this paper to describe a 
tensor category structure on the semisimple subcategory of 
$\mathcal{O}_{\ell+h^\vee}$ generated by irreducible $\imzero$-modules. This 
could possibly be interesting for comparison with the semisimple subquotient of the 
category of finite-dimensional $U_q(\g)$-modules, where $q$, as before, is the 
root of unity corresponding to $\ell$. We remark that for the case 
$\ell\in\Q_{>-h^\vee}\setminus\N$, we expect that the category of 
finitely-generated generalized modules for $\imzero$ should be a braided tensor category, 
but this does not seem to have been proven yet. In fact, in many cases these categories are not well understood; for instance, although irreducible $\imzero$-modules have been classified for certain rational levels in the case $\g=\mathfrak{sl}_2(\C)$ (\cite{AM}, \cite{DLM2}, \cite{RW}), such a classificiation does not seem to exist in general.

It is also interesting to consider how the methods and results in this paper may 
generalize to vertex operator algebras associated to other rational conformal 
field theories. Our analysis of $\imzero$-modules when $\ell\in\N$ is 
aided by the particularly simple nature of the KZ equations satisfied 
by correlation functions corresponding to compositions of intertwining operators.
 For a general vertex 
operator algebra $V$, Huang showed in \cite{H5} that compositions of intertwining operators among $V$-modules satisfying the 
$C_1$-cofiniteness condition satisfy systems of regular singular point 
differential equations. However, these differential equations are not generally 
explicit and may be of limited use for obtaining an explicit description of the 
associativity isomorphisms in the tensor category of $V$-modules. In situations 
where we have an explicit description of the abelian category of $V$-modules and 
explicit differential equations for correlation functions, the methods and results of this paper 
may be generalizable. For example, for Virasoro vertex operator algebras 
associated to minimal models in rational conformal field theory, the correlation 
functions satisfy Belavin-Polyakov-Zamolodchikov equations (\cite{BPZ}).

Now we discuss the outline of this paper. Section 2  recalls the affine Lie 
algebra $\ghat$ and its modules, the vertex operator algebra structure on 
certain $\ghat$-modules, such as $\imzero$, and the associated Kac-Moody Lie 
algebra $\widetilde{\g}$. Moreover, we prove a theorem on the weights of 
irreducible standard $\widetilde{\mathfrak{g}}$-modules. In Section 3, we recall 
the notions of intertwining operator and $P(z)$-intertwining map, 
$z\in\C^\times$, among modules for a vertex operator algebra, as well as the 
notion of $P(z)$-tensor product of modules for a vertex operator algebra. In 
Section 4, we recall the Zhu's algebra of a vertex operator algebra and the 
(tensor) equivalence between modules for a suitable vertex operator algebra, 
such as $\imzero$ when $\ell\in\N$, and finite-dimensional modules for its Zhu's 
algebra. We also recall the classification theorem of \cite{FZ} and \cite{Li2} 
for intertwining operators among certain modules for a vertex operator algebra 
and use it to describe $P(z)$-tensor products among $\imzero$-modules when 
$\ell\in\N$. In Section 5, we recall how products and iterates of intertwining 
operators among modules for $\imzero$ lead to solutions of the KZ equations, and 
we construct Drinfeld associators from solutions to a general one-variable KZ 
equation at arbitrary level. Finally, in Section 6, we use the associativity of 
intertwining operators among $\imzero$-modules to show that under the 
equivalence of categories given by the Zhu's algebra of $\imzero$, the 
associativity isomorphisms in $\imzero-\mathbf{mod}$ correspond to adjoints of 
Drinfeld associators.

\section{Affine Lie algebras and their modules}

In this section we recall the vertex operator algebras and their modules which 
come from representations of affine Lie algebras. We also recall some facts 
about affine Kac-Moody Lie algebras and prove a theorem on the weights of their 
irreducible standard modules.

\subsection{Vertex operator algebras from affine Lie algebras}

We fix a finite-dimensional simple complex Lie algebra $\g$, with Cartan 
subalgebra $\mathfrak{h}$. The Lie algebra $\g$ has a unique up to scale 
nondegenerate invariant bilinear form $\langle\cdot,\cdot\rangle$ which remains 
nondegenerate when restricted to $\mathfrak{h}$. Thus 
$\langle\cdot,\cdot\rangle$ induces a nondegenerate bilinear form on 
$\mathfrak{h}^*$, and we scale $\langle\cdot,\cdot\rangle$ so that
\begin{equation*}
 \langle\alpha,\alpha\rangle=2
\end{equation*}
when $\alpha\in\mathfrak{h}^*$ is a long root of $\g$. 

We recall some basic facts and notation regarding $\g$ (see for example 
\cite{Hu} for more details). We recall that the \textit{root lattice} $Q$ is the 
$\Z$-span of the simple roots $\lbrace\alpha_i\rbrace_{i=1}^{\mathrm{rank}\,\g}$ 
of $\g$, and we use $Q_+$ to denote the set of non-negative integral linear 
combinations of the simple roots. Then $\mathfrak{h}^*$ has a partial order 
$\prec$ given by $\alpha\prec\beta$ if and only if $\beta-\alpha\in Q_+$. The 
\textit{weight lattice} $P$ is the set of all $\lambda\in\mathfrak{h}^*$ such 
that
\begin{equation*}
 \dfrac{2\langle\lambda,\alpha_i\rangle}{\langle\alpha_i,\alpha_i\rangle}\in\Z
\end{equation*}
for each simple root $\alpha_i$. The set of \textit{dominant integral weights} 
$P_+\subseteq P$ is the set of weights $\lambda$ such that
\begin{equation*}
 \dfrac{2\langle\lambda,\alpha_i\rangle}{\langle\alpha_i,\alpha_i\rangle}\geq 0
\end{equation*}
for each simple root $\alpha_i$. An example of a dominant integral weight is the 
weight $\rho$ defined to be half the sum of the positive roots of $\g$. The 
weight $\rho$ satisfies 
$\frac{2\langle\rho,\alpha_i\rangle}{\langle\alpha_i,\alpha_i\rangle}=1$ for each 
simple root $\alpha_i$.

Every finite-dimensional 
$\g$-module $U$ is completely reducible, and the irreducible finite-dimensional 
$\g$-modules are given by the irreducible highest-weight modules $L_\lambda$ 
where $\lambda\in\mathfrak{h}^*$ is a dominant integral weight. Every $\g$- 
module $U$ has a contragredient: the dual space $U^*$ is a $\g$-module with the 
action of $\g$ given by
\begin{equation*}
 \langle a\cdot u',u\rangle=-\langle u', a\cdot u\rangle
\end{equation*}
for $a\in\g$, $u'\in U^*$, and $u\in U$. Moreover, the tensor product of two 
$\g$-modules $U$ and $V$ is a $\g$-module with the module structure given by
\begin{equation*}
 a\cdot(u\otimes v)=(a\cdot u)\otimes v+u\otimes (a\cdot v)
\end{equation*}
for $a\in\g$, $u\in U$, and $v\in V$.

We also recall the Casimir element $C=\sum_{i=1}^{\mathrm{dim}\,\g} 
\gamma_i^2\in U(\g)$, where $\lbrace \gamma_i\rbrace_{i=1}^{\mathrm{dim}\,\g}$ 
is an orthonormal basis with respect to the form $\langle\cdot,\cdot\rangle$ on 
$\g$. Since $C$ is central in $U(\g)$, $C$ acts on finite-dimensional 
irreducible modules by a scalar. In particular, $C$ acts on the adjoint 
representation $\g$ by a scalar $2 h^{\vee}$, where $h^{\vee}$ is the 
\textit{dual Coxeter number} of $\g$. More generally, $C$ acts on the 
irreducible $\g$-module $L_\lambda$ with highest weight $\lambda$ by the scalar
$\langle\lambda,\lambda+2\rho\rangle$. Since the highest weight of the adjoint 
representation is the highest root $\theta$ of $\g$, and since 
$\langle\theta,\theta\rangle=2$, it follows that
\begin{equation}\label{thetadotrho}
 \langle\rho,\theta\rangle=h^{\vee}-1.
\end{equation}
Because $\rho\in P_+$, it follows that $h^{\vee}$ is a positive integer.

The affine Lie algebra $\ghat$ is given by
\begin{equation*}
 \ghat=\g\otimes\mathbb{C}[t,t^{-1}]\oplus\mathbb{C}\mathbf{k}
\end{equation*}
where $\mathbf{k}$ is central and the remaining brackets are determined by
\begin{equation*}
 [g\otimes t^m, h\otimes t^n]=[a,b]\otimes t^{m+n}+m\langle 
g,h\rangle\delta_{m+n,0}\mathbf{k}
\end{equation*}
for $g,h\in\g$, $m,n\in\Z$. We also have the decomposition
\begin{equation*}
 \ghat=\ghat_-\oplus\ghat_0\oplus\ghat_+
\end{equation*}
where
\begin{equation*}
 \ghat_\pm=\coprod_{n\in\pm\Z_+} \g\otimes t^n,\;\;\;\ghat_0=\g\otimes 
t^0\oplus\mathbb{C}\mathbf{k}.
\end{equation*}

We construct modules for $\ghat$ as follows. Any finite-dimensional $\g$-module 
$U$ becomes a $\ghat_0\oplus\ghat_+$-module on which $\mathfrak{g}\otimes t^n$ 
acts trivially for $n>0$, $\g\otimes t^0$ acts as $\g$, and $\mathbf{k}$ as 
some \textit{level} $\ell\in\mathbb{C}$. Then we have the generalized Verma 
module
\begin{equation*}
 \gvmu=U(\ghat)\otimes_{U(\ghat_0\oplus\ghat_+)} U.
\end{equation*}
The generalized Verma module $\gvmu$ is the linear span of vectors of the form
\begin{equation*}
 g_1(-n_1)\cdots g_k(-n_k)u
\end{equation*}
where $g_i\in\g$, $n_i>0$, $u\in U$, and we use the notation $g(n)$ to denote 
the action of $g\otimes t^n$ on a $\ghat$-module. When $U$ is an irreducible 
$\g$-module, $\gvmu$ has a unique maximal irreducible quotient $\imu$. If $U$ is 
not 
irreducible, it is still a direct sum of irreducible $\g$-modules, and in this 
case we use $\imu$ to denote the corresponding direct sum of irreducible 
$\ghat$-modules. Note that in any case, $\imu$ is the quotient of $\gvmu$ by 
the maximal submodule which does not intersect $U$.
\begin{rema}
 If $L_\lambda$ is the finite-dimensional irreducible $\g$-module with highest 
weight $\lambda\in P_+$, we will typically use $V_{\ghat}(\ell,\lambda)$ and 
$L_{\ghat}(\ell,\lambda)$ to denote the $\ghat$-modules $\gvmlambda$ and 
$\imlambda$, respectively. In particular, $\gvmzero$ is the generalized Verma 
module induced from the trivial one-dimensional $\g$-module $\C\mathbf{1}$, and 
$\imzero$ is its irreducible quotient.
\end{rema}

When $\ell\neq-h^{\vee}$, the generalized Verma module $V_{\ghat}(\ell,0)$ 
induced from the one-dimensional $\mathfrak{g}$-module $\mathbb{C}\mathbf{1}$ 
has the structure of a vertex operator algebra (\cite{FZ}; see for example 
\cite{FLM}, \cite{FHL}, and \cite{LL} for the definitions of vertex operator 
algebra and module, and for notation). The vacuum vector of $\gvmzero$ is 
$\mathbf{1}$ and the vertex operator map determined by
\begin{equation}\label{vrtxop}
 Y(g(-1)\mathbf{1},x)=g(x)=\sum_{n\in\mathbb{Z}} g(n) x^{-n-1}
\end{equation}
for $g\in\mathfrak{g}$. The conformal vector $\omega$ of $\gvmzero$ is given by
\begin{equation*}
 \omega=\dfrac{1}{2(\ell+h^\vee)}\sum_{i=1}^{\mathrm{dim}\,\mathfrak{g}} 
\gamma_i(-1)^2\mathbf{1},
\end{equation*}
where as before $\lbrace \gamma_i\rbrace_{i=1}^{\mathrm{dim}\,\g}$ is an 
orthonormal basis of $\mathfrak{g}$. 

It follows from \eqref{vrtxop} and the Jacobi identity for vertex operator 
algebras that the Virasoro operators $L(n)$ on $\gvmzero$ for $n\in\Z$ are given 
by
\begin{equation}\label{L(n)}
 L(n)=\dfrac{1}{2(\ell+h^\vee)}\sum_{i=1}^{\mathrm{dim}\,\g}\sum_{m\in\Z} 
\nordcirc \gamma_i(m)\gamma_i(n-m)\nordcirc;
\end{equation}
here the normal ordering notation means
\begin{equation*}
 \nordcirc g(k)h(l)\nordcirc =\left\lbrace\begin{array}{ccc}
                                      g(k)h(l) & \mathrm{if} & k\leq 0\\
                                      h(l)g(k) & \mathrm{if} & k>0
                                     \end{array}\right. 
\end{equation*}
for $g,h\in\g$ and $k,l\in\Z$. In particular, we have
\begin{equation}\label{L(0)} 
L(0)=\dfrac{1}{2(\ell+h^\vee)}\sum_{i=1}^{\mathrm{dim}\,\g}\gamma_i(0)^2+\dfrac{
1}{\ell+h^\vee}\sum_{i=1}^{\mathrm{dim}\,\g}\sum_{n>0}\gamma_i(-n)\gamma_i(n)
\end{equation}
and
\begin{equation}\label{L(-1)}
 L(-1)=\dfrac{1}{\ell+h^\vee}\sum_{i=1}^{\mathrm{dim}\,\g}\sum_{n\geq 
0}\gamma_i(-n-1)\gamma_i(n).
\end{equation}

For any weight $\lambda\in P_+$ of $\mathfrak{g}$, the generalized Verma module 
$V_{\ghat}(\ell,\lambda)$ is a $\gvmzero$-module with vertex operator map also 
determined by (\ref{vrtxop}). Thus the Virasoro operators $L(n)$ for $n\in\Z$ on 
$V_{\ghat}(\ell,\lambda)$ are also given by \eqref{L(n)}. In particular, since 
$g(n)$ for $g\in\g$ and $n>0$ annihilates $L_\lambda\subseteq 
V_{\ghat}(\ell,\lambda)$, we have from \eqref{L(0)} that $L(0)$ acts on 
$L_\lambda$ as the scalar
\begin{equation}\label{hlambdal}
 h_{\lambda,\ell}=\dfrac{1}{2(\ell+h^\vee)}\langle\lambda,\lambda+2\rho\rangle.
\end{equation}
Moreover, one can check that
\begin{equation}\label{L0comm}
 [L(0), g(n)]=-n\,g(n)
\end{equation}
for any $g\in\g$ and $n\in\Z$; this means that the conformal weight gradation of 
$V_{\ghat}(\ell,\lambda)$ is given by
\begin{equation*}
 \mathrm{wt}\,g_1(-n_1)\cdots g_k(-n_k)u=n_1+\ldots
+n_k+h_{\lambda,\ell}
\end{equation*}
for $g_i\in\g$, $n_i>0$, and $u\in L_\lambda$. The irreducible $\ghat$-module 
$L_{\ghat}(\ell,\lambda)$ is also a $\gvmzero$-zero module, and in fact every 
irreducible $V_{\ghat}(\ell,0)$-module is isomorphic to 
$L_{\ghat}(\ell,\lambda)$ for some $\lambda\in P_+$.

The irreducible $\ghat$-module $\imzero$ is also a vertex operator algebra with 
vertex operator algebra structure induced from $\gvmzero$. We shall be 
particularly interested in the case $\ell\in\mathbb{N}$. Then the category of 
$\imzero$-modules consists of all standard level $\ell$ $\ghat$-modules 
(\cite{FZ}). In particular, all $\imzero$-modules are completely reducible, and 
the irreducible $\imzero$-modules are given by $L_{\ghat}(\ell,\lambda)$ where 
$\lambda\in P_+$ satisfies $\langle\lambda,\theta\rangle\leq\ell$.

\subsection{The Kac-Moody Lie algebra $\widetilde{\g}$}\label{sub:gtilde}

Here we recall the affine Kac-Moody Lie algebra $\widetilde{\g}$ and some of its 
properties that we shall need later; for more details and proofs, see references 
such as \cite{L}, \cite{K}, \cite{MP}, and \cite{C}. Let us extend the algebra 
$\ghat$ to the Kac-Moody Lie algebra
\begin{equation*}
 \widetilde{\g}=\ghat\rtimes\mathbb{C}\mathbf{d}
\end{equation*}
where $[\mathbf{d},\mathbf{k}]=0$ and
\begin{equation}\label{dcomm}
 [\mathbf{d},g\otimes t^n]=n\,g\otimes t^n
\end{equation}
for $g\in\g$ and $n\in\Z$. Recalling the Cartan subalgebra $\mathfrak{h}$ of 
$\g$,
\begin{equation*}
 \mathfrak{H}=\mathfrak{h}\oplus\C\mathbf{k}\oplus\C\mathbf{d}
\end{equation*}
is a Cartan subalgebra for $\widetilde{\g}$. Then $\mathfrak{H}^*$ has a basis 
$\lbrace\alpha_i\rbrace_{i=1}^{\mathrm{rank}\,\g}\cup\lbrace\mathbf{k}',\mathbf{
d}'\rbrace$ where $\mathbf{k}'$ and $\mathbf{d}'$ are dual to $\mathbf{k}$ and 
$\mathbf{d}$, respectively. The remaining simple root of $\widetilde{\g}$ is 
given by
\begin{equation*}
 \alpha_0=-\theta+\mathbf{d}',
\end{equation*}
where as previously $\theta$ is the longest root of $\g$. We use $\Delta$ to 
denote the set of roots of $\widetilde{g}$ and $\Delta_{\pm}$ to denote the sets 
of positive and negative roots of $\widetilde{\g}$, respectively. We use 
$\widehat{Q}$ to denote the root lattice of $\widetilde{\g}$, the integer span 
of the roots of $\widetilde{\g}$. Note that
\begin{equation*}
 \widehat{Q}=Q\oplus\Z\mathbf{d'}.
\end{equation*}
We use $\widehat{Q}_+$ to denote the non-negative integral span of the simple 
roots of $\widetilde{\g}$. Then the partial order on $\mathfrak{h}^*$ extends to 
a partial order $\prec$ on $\mathfrak{H}^*$ given by $\alpha\prec\beta$ if and 
only if $\beta-\alpha\in\widehat{Q}_+$.

We extend the bilinear form $\langle\cdot,\cdot\rangle$ on $\mathfrak{h}^*$ to a 
nondegenerate symmetric form on $\mathfrak{H}^*$ by setting
\begin{equation}\label{kdotd}
 \langle\mathbf{k}',\mathbf{d}'\rangle=1
\end{equation}
and
\begin{equation*}
 \langle\mathbf{k}',\mathbf{k}'\rangle=\langle\mathbf{d}',\mathbf{d}'\rangle 
=\langle\mathbf{k}',\alpha_i\rangle=\langle\mathbf{d}',\alpha_i\rangle=0
\end{equation*}
for $i\geq 1$. Now we use $\widehat{P}$ to denote the \textit{weights} of 
$\widetilde{\g}$, the elements $\Lambda$ of $\mathfrak{H}^*$ such that
\begin{equation*}
 \dfrac{2\langle\Lambda,\alpha_i\rangle}{\langle\alpha_i,\alpha_i\rangle}\in\Z
\end{equation*}
for all $0\leq i\leq\mathrm{rank}\,\g$. Note that
\begin{equation*}
 \widehat{P}=P\oplus\Z\mathbf{k}'\oplus\C\mathbf{d'}.
\end{equation*}
We say that a weight $\Lambda\in\widehat{P}$ is \textit{dominant integral} if 
\begin{equation*} 
\dfrac{2\langle\Lambda,\alpha_i\rangle}{\langle\alpha_i,\alpha_i\rangle}
\in\mathbb{N}
\end{equation*}
for all $0\leq i\leq\mathrm{rank}\,\g$; we use $\widehat{P}_+$ to denote the 
dominant integral weights of $\widetilde{\g}$. Note that 
$\Lambda=\lambda+\ell\mathbf{k'}+h\mathbf{d'}$, where $\lambda\in P$ and 
$\ell\in\Z$, is dominant integral if and only if $\ell\in\mathbb{N}$ and 
$\lambda$ is a dominant integral weight of $\g$ such that 
$\langle\lambda,\theta\rangle\leq\ell$.

Recalling the dominant integral weight $\rho$ of $\g$ and the dual Coxeter 
number $h^\vee$, we have:
\begin{propo}
 If we set $\widehat{\rho}=\rho+h^{\vee}\mathbf{k}'$, then 
$\frac{2\langle\widehat{\rho},\alpha_i\rangle}{\langle\alpha_i,\alpha_i\rangle}
=1$ for all simple roots $\alpha_i$ of $\widetilde{\g}$.
\end{propo}
\begin{proof}
 It is clear that 
$\frac{2\langle\widehat{\rho},\alpha_i\rangle}{\langle\alpha_i,\alpha_i\rangle}
=1$ for $i\geq 1$ since $\rho$ has this property. Since 
$\langle\alpha_0,\alpha_0\rangle=2$, we just need to check 
$\langle\widehat{\rho},\alpha_0\rangle=1$. In fact,
\begin{equation*} 
\langle\widehat{\rho},\alpha_0\rangle=\langle\rho+h^{\vee}\mathbf{k}',
-\theta+\mathbf{d}'\rangle=-(h^{\vee}-1)+h^{\vee}=1
\end{equation*}
by \eqref{thetadotrho} and \eqref{kdotd}.
\end{proof}

We now recall the Weyl group of $\widetilde{\g}$. For each $0\leq 
i\leq\mathrm{rank}\,\g$, we define the simple reflection $r_i$ on 
$\mathfrak{H}^*$ by
\begin{equation*} 
r_i(\alpha)=\alpha-\dfrac{2\langle\alpha,\alpha_i\rangle}{\langle\alpha_i,
\alpha_i\rangle}\alpha_i
\end{equation*}
for $\alpha\in\mathfrak{H}^*$. Then the Weyl group $W$ is the group of 
isometries of $\mathfrak{H}^*$ generated by the reflections $r_i$. For an 
element $w\in W$, the \textit{length} $l(w)$ is the minimal number of generators 
in any expression of $w$ as a product of the generators $r_i$. Here we recall 
some standard facts about the Weyl group of $\widetilde{\g}$ that we shall need:
\begin{propo}\label{weylgrouppropo1}
For $w\in W$, let $\Phi_w\subseteq\Delta_+$ denote the positive roots sent by 
$w$ to negative roots. Then for any $w\in W$, $l(w)=\vert\Phi_w\vert$ and 
$\widehat{\rho}-w(\widehat{\rho})=\sum_{\alpha\in\Phi_w} \alpha$.
\end{propo}
 
\begin{propo}\label{weylgrouppropo2}
The Weyl group $W$ preserves the set of weights $\widehat{P}$, and for any 
$\Lambda\in\widehat{P}$, there is some $w\in W$ such that 
$w(\Lambda)\in\widehat{P}_+$.
\end{propo}

Now we consider $\widetilde{\g}$-modules. For each $\Lambda\in\widehat{P}_+$, 
there is a unique irreducible standard $\widetilde{\g}$-module $L(\Lambda)$ with 
highest weight $\Lambda$. Note that by \eqref{L0comm} and \eqref{dcomm}, for any 
$\lambda\in P_+$ and $\ell\in\C$, the $\ghat$-modules $V_{\ghat}(\ell,\lambda)$ 
and $L_{\ghat}(\ell,\lambda)$ become $\widetilde{\g}$-modules on which 
$\mathbf{d}$ acts as $-L(0)$. In particular, when $\ell\in\mathbb{N}$ and 
$\lambda\in P_+$ satisfies $\langle\lambda,\theta\rangle\leq\ell$, 
$L_{\ghat}(\ell,\lambda)\cong L(\Lambda)$ where
\begin{equation*}
 \Lambda=\lambda+\ell\mathbf{k}'-h_{\lambda,\ell}\mathbf{d}'\in\widehat{P}_+.
\end{equation*}
For a dominant integral weight $\Lambda\in\widehat{P}_+$, the weights of 
$L(\Lambda)$ lie in $\Lambda-\widehat{Q}_+$. Moreover, the Weyl group $W$ 
preserves the weights of any standard $\widetilde{\g}$-module. Thus we obtain:
\begin{propo}\label{weylstanmod}
 If $\Lambda'$ is a weight of $L(\Lambda)$ and $w\in W$, then 
$w(\Lambda')\prec\Lambda$.
\end{propo}

We will need the following theorem on the weights of of irreducible 
$\widetilde{\g}$-modules. Suppose that 
$\Lambda=\lambda+\ell\mathbf{k}'-h_{\lambda,\ell}\mathbf{d}'$ is a dominant 
integral weight of $\widetilde{\g}$, where $\ell\in\N$ and $\lambda\in P_+$ 
satisfies $\langle\lambda,\theta\rangle\leq\ell$.
\begin{theo}\label{weightstheo}
If $\Lambda'=\lambda'+\ell\mathbf{k}'-(h_{\lambda',\ell}-m)\mathbf{d'}$ is a 
weight of $L(\Lambda)$ such that $\lambda'\in P_+$ and $m\in\mathbb{N}$, then 
$\Lambda'=\Lambda$.
\end{theo}
\begin{proof}
By Proposition \ref{weylgrouppropo2} there is a weight 
$\widetilde{\Lambda}\in\widehat{P}$ and an element $w\in W$ such that
 \begin{equation}\label{lamtdef}  
w^{-1}(\Lambda'+\widehat{\rho})=\widetilde{\Lambda}+\widehat{\rho}\in\widehat{P}
_+.
 \end{equation}
\begin{lemma}\label{lhsame}
 We have 
$\widetilde{\Lambda}=\widetilde{\lambda}+\ell\mathbf{k'}-(h_{\widetilde{\lambda}
,\ell}-m)\mathbf{d'}$ for some $\widetilde{\lambda}\in P$ which satisfies 
$\widetilde{\lambda}+\rho\in P_+$ and 
$\langle\widetilde{\lambda},\theta\rangle\leq\ell+1$.
\end{lemma}
\begin{proof}
Certainly we can write 
$\widetilde{\Lambda}=\widetilde{\lambda}+\widetilde{\ell}\mathbf{k}'-h\mathbf{d'
}$ for some $\widetilde{\lambda}\in P$, $\widetilde{\ell}\in\Z$, and $h\in\C$. 
First we observe that for any weight $\Lambda'\in\widehat{P}$ and any simple 
generator $r_i$ of $W$, the coefficient of $\mathbf{k}'$ in
\begin{equation*} 
r_i(\Lambda')=\Lambda'-\dfrac{2\langle\Lambda',\alpha_i\rangle}{\langle\alpha_i,
\alpha_i\rangle}\alpha_i
\end{equation*}
is the same as the coefficient of $\mathbf{k}'$ in $\Lambda'$ because 
$\alpha_i\in Q+\Z\mathbf{d'}$. Thus $\widetilde{\ell}=\ell$. Also, the fact that 
$\widetilde{\Lambda}+\widehat{\rho}\in\widehat{P}_+$ implies that 
$\widetilde{\lambda}+\rho\in P_+$ and
\begin{equation*}
 \langle\lambda+\rho,\theta\rangle\leq\ell+h^{\vee}.
\end{equation*}
By \eqref{thetadotrho}, this implies 
$\langle\widetilde{\lambda},\theta\rangle\leq\ell+1$.

It remains to show that $h=h_{\widetilde{\lambda},\ell}-m$. For this, we use the 
fact that
\begin{align*}
 \langle\widetilde{\Lambda},\widetilde{\Lambda}+2\widehat{\rho}\rangle & 
=\langle w^{-1}(\Lambda'+\widehat{\rho})-\widehat{\rho}, 
w^{-1}(\Lambda'+\widehat{\rho})+\widehat{\rho}\rangle\nonumber\\
 & =\langle 
w^{-1}(\Lambda'+\widehat{\rho}),w^{-1}(\Lambda'+\widehat{\rho}
)\rangle-\langle\widehat{\rho},\widehat{\rho}\rangle\nonumber\\
 & =\langle\Lambda',\Lambda'+2\widehat{\rho}\rangle
\end{align*}
because $w^{-1}$ is an isometry. Now,
\begin{align*}
 \langle\Lambda',\Lambda'+2\widehat{\rho}\rangle 
&=\langle\lambda'+\ell\mathbf{k}'-(h_{\lambda',\ell}-m)\mathbf{d}',
\lambda+2\rho+(\ell+2h^{\vee})\mathbf{k}'-(h_{\lambda',\ell}-m)\mathbf{d}
'\rangle\nonumber\\
 & 
=\langle\lambda',\lambda'+2\rho\rangle-2(\ell+h^{\vee})(h_{\lambda',\ell}
-m)=2m(\ell+h^\vee)
\end{align*}
by the definition of $h_{\lambda',\ell}$. Thus
\begin{equation*}
 \langle\widetilde{\Lambda},\widetilde{\Lambda}+2\widehat{\rho}\rangle 
=\langle\widetilde{\lambda},\widetilde{\lambda}+2\rho\rangle-2(\ell+h^{\vee}
)h=2m(\ell+h^\vee)
\end{equation*}
likewise, so
\begin{equation*} 
h=\dfrac{1}{2(\ell+h^{\vee})}(\langle\widetilde{\lambda},\widetilde{\lambda}
+2\rho\rangle-2m(\ell+h^\vee))=h_{\widetilde{\lambda},\ell}-m,
\end{equation*}
as desired.
\end{proof}
Continuing with the proof of the theorem, by Propositions \ref{weylgrouppropo1} 
and \ref{weylstanmod} we have
\begin{equation*} 
\widetilde{\Lambda}=w^{-1}(\Lambda')+w^{-1}(\widehat{\rho})-\widehat{\rho}
\prec\Lambda.
\end{equation*}
Set $N=h_{\widetilde{\lambda},\ell}-h_{\lambda,\ell}$. Since we know
\begin{equation}\label{lamminlamt} 
\Lambda-\widetilde{\Lambda}=\lambda-\widetilde{\lambda}+(h_{\widetilde{\lambda},
\ell}-h_{\lambda,\ell}-m)\mathbf{d'}=\lambda-\widetilde{\lambda}
+(N-m)\theta+(N-m)\alpha_0\in\widehat{Q}_+,
\end{equation}
it follows that $N-m\in\mathbb{N}$ (so that $N\in\mathbb{N}$ as well) and 
$\widetilde{\lambda}\prec\lambda+(N-m)\theta\prec\lambda+N\theta$. The following 
lemma then implies that $N=m=0$:
\begin{lemma}
 Suppose $\lambda,\widetilde{\lambda}$ are weights of $\g$ such that 
$\lambda+\widetilde{\lambda}+\rho\in P_+$, 
$\langle\lambda+\widetilde{\lambda},\theta\rangle\leq 2\ell+1$, and 
$\widetilde{\lambda}\prec\lambda+N\theta$ for some $N\in\mathbb{N}$. Then 
$h_{\widetilde{\lambda},\ell}-h_{\lambda,\ell}<N$ if $N>0$.
\end{lemma}
\begin{proof}
 Because $\widetilde{\lambda}-\lambda\prec N\theta$, we have 
$\langle\widetilde{\lambda}-\lambda,\mu\rangle\leq N\langle\theta,\mu\rangle$ 
whenever $\mu\in P_+$. In particular,
 \begin{equation*}
  \langle\widetilde{\lambda}-\lambda,\rho\rangle\leq 
N\langle\rho,\theta\rangle=N(h^{\vee}-1)
  \end{equation*}
 by \eqref{thetadotrho} and
 \begin{equation*}
  \langle\widetilde{\lambda}-\lambda,\lambda+\widetilde{\lambda}+\rho\rangle\leq 
N\langle\lambda+\widetilde{\lambda}+\rho,\theta\rangle\leq N(2\ell+h^{\vee}).
 \end{equation*}
Then we have
\begin{align*}
 h_{\widetilde{\lambda},\ell}-h_{\lambda,\ell} & 
=\dfrac{1}{2(\ell+h^{\vee})}\left(\langle\widetilde{\lambda},\widetilde{\lambda}
+2\rho\rangle-\langle\lambda,\lambda+2\rho\rangle\right)\nonumber\\ 
&=\dfrac{1}{2(\ell+h^{\vee})}\langle\widetilde{\lambda}-\lambda,
\lambda+\widetilde{\lambda}+2\rho\rangle\nonumber\\
 & \leq\dfrac{2(\ell+h^{\vee})-1}{2(\ell+h^{\vee})} N.
\end{align*}
Thus $h_{\widetilde{\lambda},\ell}-h_{\lambda,\ell}<N$ when $N>0$. 
\end{proof}

We have now proved that $\widetilde{\lambda}\prec\lambda$ and 
$h_{\widetilde{\lambda},\ell}=h_{\lambda,\ell}$. Thus 
$\lambda=\widetilde{\lambda}+\alpha$ for some $\alpha\in Q_+$ and
\begin{equation*} 
\langle\widetilde{\lambda},\widetilde{\lambda}+2\rho\rangle=\langle\widetilde{
\lambda}+\alpha,\widetilde{\lambda}+\alpha+2\rho\rangle,
\end{equation*}
or
\begin{equation*}
 \langle\alpha,\alpha\rangle+2\langle\widetilde{\lambda}+\rho,\alpha\rangle=0.
\end{equation*}
Since $\widetilde{\lambda}+\rho\in P_+$ by Lemma \ref{lhsame}, this means 
$\langle\alpha,\alpha\rangle=\langle\widetilde{\lambda}+\rho,\alpha\rangle =0$. 
Thus we must have $\alpha=0$ and $\widetilde{\lambda}=\lambda$. Then by 
\eqref{lamminlamt} we have $\widetilde{\Lambda}=\Lambda$.

It now follows from \eqref{lamtdef} that 
$w^{-1}(\Lambda'+\widehat{\rho})=\Lambda+\widehat{\rho}$, or
\begin{equation*}
 \Lambda'=w(\Lambda+\widehat{\rho})-\widehat{\rho}.
\end{equation*}
To conclude the proof of the theorem, we must show that $w= 1$. In fact, when 
$w\neq 1$, Proposition \ref{weylgrouppropo1} implies that
\begin{equation*} 
w^{-1}(\Lambda')=\Lambda+\widehat{\rho}-w^{-1}(\widehat{\rho})=\Lambda+\sum_{
\alpha\in\Phi_{w^{-1}}}\alpha\nprec\Lambda.
\end{equation*}
Since by assumption $\Lambda'$ is a weight of $L(\Lambda)$, this contradicts 
Proposition \ref{weylstanmod}; thus $w=1$ and $\Lambda'=\Lambda$. This completes 
the proof of the theorem.
\end{proof}

\section{Intertwining operators and tensor products of modules for a vertex 
operator algebra}

In this section we recall the notion of tensor product of modules for a vertex 
operator algebra from \cite{HL1}, \cite{HLZ3}. We first recall the notions of 
intertwining operator and $P(z)$-intertwining map, for $z\in\C^\times$, among a 
triple of modules for a vertex operator algebra.

\subsection{Intertwining operators and $P(z)$-intertwining maps}

In general for a vector space $V$, we use $V\lbrace x\rbrace$ to denote the 
space of formal series
\begin{equation*}
 V\lbrace x\rbrace =\left\lbrace\sum_{n\in\C} v_n x^n, \,\,v_n\in 
V\right\rbrace.
\end{equation*}
Now suppose that $V$ is a vertex operator algebra; we recall the notion of 
intertwining operator among a triple of $V$-modules (see \cite{FHL} or 
\cite{HL1}):
\begin{defi}
 Suppose $W_1$,  $W_2$ and $W_3$ are $V$-modules. An \textit{intertwining 
operator} of type $\binom{W_3}{W_1\,W_2}$ is a linear map
 \begin{eqnarray*}
\mathcal{Y}: W_1\otimes W_2&\to& W_3\{x\},
\\w_{(1)}\otimes w_{(2)}&\mapsto &\mathcal{Y}(w_{(1)},x)w_{(2)}=\sum_{n\in
{\mathbb C}}(w_{(1)})_n
w_{(2)}x^{-n-1}\in W_3\{x\}
\end{eqnarray*}
satisfying the following conditions:
\begin{enumerate}

\item  {\it Lower truncation}: For any $w_{(1)}\in W_1$, $w_{(2)}\in W_2$ and 
$n\in\mathbb{C}$,
\begin{equation}\label{log:ltc}
(w_{(1)})_{(n+m)}w_{(2)}=0\;\;\mbox{ for }\;m\in {\mathbb
N} \;\mbox{ sufficiently large.}
\end{equation}

\item The {\it Jacobi identity}:
\begin{align}\label{intwopjac}
 x^{-1}_0\delta \left(\frac{x_1-x_2}{x_0}\right) 
Y_{W_3}(v,x_1)\mathcal{Y}(w_{(1)},x_2) & - 
x^{-1}_0\delta\left(\frac{x_2-x_1}{-x_0}\right)\mathcal{Y}(w_{(1)},x_2)Y_{W_2}(v
,x_1)\nonumber\\
& =  x^{-1}_2\delta \left(\frac{x_1-x_0}{x_2}\right) 
\mathcal{Y}(Y_{W_1}(v,x_0)w_{(1)},x_2)
\end{align}
for $v\in V$ and $w_{(1)}\in W_1$.

\item The {\it $L(-1)$-derivative property:} for any
$w_{(1)}\in W_1$,
\begin{equation}\label{intwopderiv}
\mathcal{Y}(L(-1)w_{(1)},x)=\frac{d}{dx}\mathcal{Y}(w_{(1)},x).
\end{equation}
\end{enumerate}
\end{defi}
\begin{rema}
 We use $\mathcal{V}^{W_3}_{W_1 W_2}$ to denote the vector space of intertwining 
operators of type $\binom{W_3}{W_1\,W_2}$; the dimension of 
$\mathcal{V}^{W_3}_{W_1 W_2}$ is the corresponding \textit{fusion rule} 
$\mathcal{N}^{W_3}_{W_1 W_2}$.
\end{rema}
\begin{rema}
 The vertex operator $Y$ of $V$ is an intertwining operator of type 
$\binom{V}{V\,V}$, and for a $V$-module $W$, the vertex operator $Y_W$ is an 
intertwining operator of type $\binom{W}{V\,W}$.
\end{rema}

 Taking the coefficient of $x_0^{-1}$ in the Jacobi identity \eqref{intwopjac} 
yields the \textit{commutator formula}
 \begin{align}\label{commform}  
Y_{W_3}(v,x_1)\mathcal{Y}(w_{(1)},x_2)-\mathcal{Y}(w_{(1)},x_2)Y_{W_2}(v,
x_1)=\mathrm{Res}_{x_0} 
x_2^{-1}\delta\left(\dfrac{x_1-x_0}{x_2}\right)\mathcal{Y}(Y_{W_1}(v,x_0)w_{(1)}
,x_2),
 \end{align}
where the formal residue notation $\mathrm{Res}_{x_0}$ denotes the coefficient 
of $x_0^{-1}$. Similarly, taking the coefficient of $x_1^{-1}$ in 
\eqref{intwopjac} and then the coefficient of $x_0^{-n-1}$ for any $n\in\Z$ 
yields the \textit{iterate formula}
\begin{align}\label{itform}
 \mathcal{Y}(v_n w_{(1)},x_2)=\mathrm{Res}_{x_1} \left((x_1-x_2)^n 
Y_{W_3}(v,x_1)\mathcal{Y}(w_{(1)},x_2)-(-x_2+x_1)^n 
\mathcal{Y}(w_{(1)},x_2)Y_{W_2}(v,x_1)\right).
\end{align}
Together, the commutator and iterate formulas are equivalent to the Jacobi 
identity (see \cite{LL} for the special case that $\mathcal{Y}=Y_W$ for a 
$V$-module $W$).

We now recall from \cite{HL1}, \cite{HLZ3} the notion of $P(z)$-intertwining map 
for $z\in\C^\times$. First, if
\begin{equation*}
 W=\coprod_{n\in\mathbb{C}} W_{(n)}
\end{equation*}
is a $\mathbb{C}$-graded vector space (such as a $V$-module), then the 
\textit{algebraic completion} of $W$ is the vector space
\begin{equation*}
 \overline{W}=\prod_{n\in\mathbb{C}} W_{(n)}.
\end{equation*}
For any $n\in\mathbb{C}$, we use $\pi_n$ to denote the canonical projection 
$\overline{W}\rightarrow W_{(n)}$. Recall that the graded dual of a graded 
vector space $W$ is given by
\begin{equation*}
 W'=\coprod_{n\in\mathbb{C}} W_{(n)}^*;
\end{equation*}
it is easy to see that $\overline{W}=(W')^*$.
\begin{defi}
 Suppose $W_1$, $W_2$, and $W_3$ are $V$-modules and $z\in\mathbb{C}^\times$. A 
\textit{$P(z)$-intertwining map} of type $\binom{W_3}{W_1\,W_2}$ is a linear map
 \begin{equation*}
  I: W_1\otimes W_2\rightarrow \overline{W_3}
 \end{equation*}
satisfying the following conditions:
\begin{enumerate}
 \item \textit{Lower truncation}: For any $w_{(1)}\in W_1$, $w_{(2)}\in W_2$ 
and $n\in\mathbb{C}$,
\begin{equation}\label{map:ltc}
\pi_{n-m}(I(w_{(1)}\otimes w_{(2)})=0\;\;\mbox{ for }\;m\in {\mathbb
N} \;\mbox{ sufficiently large.}
\end{equation}
\item The {\it Jacobi identity}:
\begin{align}\label{intwmapjac}
 x^{-1}_0\delta \left(\frac{x_1-z}{x_0}\right) 
Y_{W_3}(v,x_1) I(w_{(1)}\otimes w_{(2)}) & - 
x^{-1}_0\delta\left(\frac{z-x_1}{-x_0}\right)I(w_{(1)}\otimes 
Y_{W_2}(v,x_1)w_{(2)})\nonumber\\
& =  z^{-1}\delta \left(\frac{x_1-x_0}{z}\right) 
I(Y_{W_1}(v,x_0)w_{(1)}\otimes w_{(2)})
\end{align}
for $v\in V$, $w_{(1)}\in W_1$, and $w_{(2)}\in W_2$.
\end{enumerate}
\end{defi}
\begin{rema}
We use $\mathcal{M}[P(z)]^{W_3}_{W_1 W_2}$, or simply $\mathcal{M}^{W_3}_{W_1 
W_2}$ if $z$ is clear, to denote the space of $P(z)$-intertwining maps of type 
$\binom{W_3}{W_1\,W_2}$.
\end{rema}

The definitions suggest that a $P(z)$-intertwining map is essentially an 
intertwining operator with the formal variable $x$ specialized to the complex 
number $z$. To make such a substitution precise, however, we must fix a 
branch of logarithm. We use log $z$ to denote the following branch of logarithm 
with a branch cut along the positive real axis:
\begin{equation*}
 \textrm{log}\,z=\textrm{log}\,\vert z\vert +i\,\textrm{arg}\,z,
\end{equation*}
where $0\leq\mathrm{arg}\,z<2\pi$. Then we define
\begin{equation*}
 \ell_p(z)=\mathrm{log}\,z+2\pi i p
\end{equation*}
for any $p\in\Z$. Then from \cite{HLZ3}, we have
\begin{propo}\label{opmapiso}
 For any $p\in\Z$, there is a linear isomorphism $\mathcal{V}^{W_3}_{W_1 
W_2}\rightarrow\mathcal{M}^{W_3}_{W_1 W_2}$ given by
 \begin{equation*}
  \mathcal{Y}\mapsto I_{\mathcal{Y},p},
 \end{equation*}
where
\begin{equation*}
 I_{\mathcal{Y},p}(w_{(1)}\otimes w_{(2)})=\mathcal{Y}(w_{(1)}, e^{\ell_p(z)})
\end{equation*}
for $w_{(1)}\in W_1$ and $w_{(2)}\in W_2$. The inverse is given by
\begin{equation*}
 I\mapsto\mathcal{Y}_{I,p}
\end{equation*}
where 
\begin{equation*}
 \mathcal{Y}_{I,p}(w_{(1)},x)w_{(2)}= 
\left(\dfrac{e^{\ell_p(z)}}{x}\right)^{-L(0)} 
I\left(\left(\dfrac{e^{\ell_p(z)}}{x}\right)^{L(0)} w_{(1)}\otimes 
\left(\dfrac{e^{\ell_p(z)}}{x}\right)^{L(0)} w_{(2)}\right)
\end{equation*}
for $w_{(1)}\in W_1$ and $w_{(2)}\in W_2$.
\end{propo}
\begin{rema}
 From Proposition \ref{opmapiso}, the dimension of $\mathcal{M}[P(z)]^{W_3}_{W_1 
W_2}$ for any 
$z\in\mathbb{C}^\times$ is also the fusion rule $\mathcal{N}^{W_3}_{W_1 W_2}$
\end{rema}

\subsection{$P(z)$-tensor products}

We now recall from \cite{HL1} and \cite{HLZ3} the definition of a $P(z)$-tensor 
product of 
$V$-modules $W_1$ and $W_2$ using a universal property:
\begin{defi}
 For $z\in\C^\times$, a \textit{$P(z)$-tensor product} of $V$-modules $W_1$ and 
$W_2$ is a 
$V$-module $W_1\boxtimes_{P(z)} W_2$ equipped with a $P(z)$-intertwining map
\begin{equation*}
 \boxtimes_{P(z)}: W_1\otimes W_2\rightarrow \overline{W_1\boxtimes_{P(z)} W_2}
\end{equation*}
such that if $W_3$ is any $V$-module and $I$ is any $P(z)$-intertwining map of 
type $\binom{W_3}{W_1\,W_2}$, then there is a unique $V$-module homomorphism 
\begin{equation*}
 \eta: W_1\boxtimes_{P(z)} W_2\rightarrow W_3
\end{equation*}
satisfying
\begin{equation*}
\overline{\eta}\circ\boxtimes_{P(z)}=I.
\end{equation*}
\end{defi}
\begin{rema}
 We typically use the notation
 \begin{equation*}
  w_{(1)}\boxtimes_{P(z)} w_{(2)}=\boxtimes_{P(z)}(w_{(1)}\otimes w_{(2)})
 \end{equation*}
for $w_{(1)}\in W_1$, $w_{(2)}\in W_2$.
\end{rema}

If $V$ is a suitable vertex operator algebra, then $P(z)$-tensor products of 
$V$-modules always exist. In particular, $P(z)$-tensor products exist when $V$ 
is finitely reductive in the sense of \cite{HLZ3}:
\begin{defi}
 A vertex operator algebra $V$ is \textit{finitely reductive} if:
 \begin{enumerate}
  \item Every $V$-module is completely reducible.
  \item There are finitely many equivalence classes of irreducible $V$-modules.
  \item All fusion rules for triples of $V$-modules are finite.
 \end{enumerate}
\end{defi}
\begin{exam} 
When $\ell\in\mathbb{N}$, the vertex operator algebra $\imzero$ is finitely 
reductive (\cite{FZ}).
\end{exam}
Suppose tensor products of $V$-modules exist. In order to give the category 
$V-\mathbf{mod}$ a tensor category structure, we must choose a specific tensor 
product bifunctor, say $\boxtimes_{P(1)}$. Under suitable additional conditions, 
$V-\mathbf{mod}$ with the tensor product $\boxtimes_{P(1)}$ becomes a braided 
tensor category (see \cite{HL1}-\cite{HL3}, \cite{H1} or 
\cite{HLZ1}-\cite{HLZ8}). We will discuss the associativity isomorphisms for 
this tensor category below.

\section{The Zhu's algebra of a vertex operator algebra and its applications}

In this section we recall from \cite{Z} that the category of modules for a 
suitable vertex operator algebra $V$ is equivalent to the category of 
finite-dimensional modules for the Zhu's algebra $A(V)$ of $V$. We also recall 
the connection given in \cite{FZ}, \cite{Li2} between $A(V)$ and 
intertwining operators among $V$-modules, and apply these results to the vertex 
operator algebra $\imzero$ when $\ell\in\mathbb{N}$.

\subsection{Zhu's algebra and an equivalence of categories}

Here we recall the Zhu's algebra of a vertex operator algebra $V$ from \cite{Z}; 
see also \cite{FZ}. The vertex operator algebra $V$ has a product $*$ given by
\begin{equation*}
 u*v=\mathrm{Res}_x\, x^{-1} Y((1+x)^{L(0)}u,x)v
\end{equation*}
for $u,v\in V$. We also define the subspace $O(V)\subseteq V$ as the linear 
span of elements
\begin{equation*}
 \mathrm{Res}_x\, x^{-2} Y((1+x)^{L(0)}u,x)v
\end{equation*}
for $u,v\in V$. Then $*$ is a well-defined product on the quotient 
$A(V)=V/O(V)$, and in fact $(A(V),*)$ is an associative algebra with unit 
$\mathbf{1}+O(V)$ called the \textit{Zhu's algebra} of $V$.

If $W$ is a $V$-module, then there is a left action of $V$ on $W$ defined by
\begin{equation*}
 v*w=\mathrm{Res}_x\,x^{-1} Y((1+x)^{L(0)}v,x)w
\end{equation*}
for $v\in V$, $w\in W$, and a right action defined by
\begin{equation*}
 w*v=\mathrm{Res}_x\,x^{-1} Y((1+x)^{L(0)-1}v,x)w
\end{equation*}
for $v\in V$, $w\in W$. If we define the subspace $O(W)\subseteq W$ as the 
linear span of elements of the form
\begin{equation*}
 \mathrm{Res}_x\,x^{-2} Y((1+x)^{L(0)}v,x)w
\end{equation*}
for $v\in V$, $w\in W$, then these left and right actions define an 
$A(V)$-bimodule structure on the quotient $A(W)=W/O(W)$.

Now, following \cite{H2}, for a $V$-module $W$ we define the \textit{top level} 
$T(W)$ to be the subspace
\begin{equation*}
 T(W)=\lbrace w\in W\,\vert\,v_n w=0, v\in 
V\;\mathrm{homogeneous},\;\mathrm{wt}\,v-n-1<0\rbrace.
\end{equation*}
Then $T(W)$ is a (left) $A(V)$-module with action defined by
\begin{equation*}
 (v+O(V))\cdot w=o(v)w
\end{equation*}
for $v\in V$, $w\in T(W)$; here $o(v)=v_{\mathrm{wt}\,v-1}$ if $v$ is 
homogeneous and we extend linearly to define $o(v)$ for non-homogeneous $v$. 
The corresponce $W\rightarrow T(W)$ defines a functor
\begin{equation*}
 T: V-\mathbf{mod}\rightarrow A(V)-\mathbf{mod},
\end{equation*}
in which a $V$-module homomorphism $f: W_1\rightarrow W_2$ corresponds to 
$T(f)=f\vert_{T(W_1)}: T(W_1)\rightarrow T(W_2)$. Note that the image of the 
restriction of a $V$-homomorphism to the top level of $W_1$ is indeed contained 
in the top level of $W_2$.
\begin{rema}
 If $W$ is an irreducible $V$-module, $T(W)$ is simply the lowest conformal 
weight space of $W$, and if $W$ is semisimple, $T(W)$ is the direct sum of the 
lowest weight spaces of the irreducible components of $W$. In general, however, 
determining $T(W)$ is a subtle problem.
\end{rema}

When $V$ is a suitable vertex operator algebra, we want $T$ to give an 
equivalence of categories between the category of $V$-modules and the category 
of finite-dimensional (left) $A(V)$-modules. Thus we need to construct a 
functor $S$ from the category $\mathbf{C}^{fin}(A(V))$ of finite-dimensional 
(left) $A(V)$-modules to the category of $V$-modules. Starting from a 
finite-dimensional $A(V)$-module $U$, there are several constructions of a 
suitable $V$-module $S(U)$ (see \cite{Z}, 
\cite{Li2}, and \cite{H2}), all of which 
involve a kind of induced module construction. Here we sketch the construction 
from \cite{H2}.

We consider the affinization $V[t, t^{-1}]=V\otimes\mathbb{C}[t,t^{-1}]$ and 
its tensor algebra $\mathcal{T}(V[t,t^{-1}])$. Then given a finite-dimensional 
$A(V)$-module $U$, the vector space $\mathcal{T}(V[t,t^{-1}])\otimes U$ is a 
left $\mathcal{T}(V[t,t^{-1}])$-module in the obvious way. For $v\in V$ and 
$n\in\Z$, we use $v(n)$ to denote the action of $v\otimes t^n$ on 
$\mathcal{T}(V[t,t^{-1}])\otimes U$, and we set
\begin{equation*}
 Y_t(v,x)=\sum_{n\in\Z} v(n) x^{-n-1}.
\end{equation*}

Now we let $\mathcal{I}$ denote the submodule of 
$\mathcal{T}(V[t,t^{-1}])\otimes U$ generated by suitable elements so that on 
the quotient $S_1(U)=(\mathcal{T}(V[t,t^{-1}])\otimes U)/\mathcal{I}$, $v(n)$ 
for $v\in V$ homogeneous and $\mathrm{wt}\,v-n-1<0$ acts trivially on $U$, 
$v(\mathrm{wt}\,v-1)$ for $v\in V$ homogeneous acts on $U$ as $v+O(V)$ acts, 
and a commutator formula for the operators $Y_t(v,x)$ holds. Next, we define 
$\mathcal{J}$ to be the submodule of $S_1(U)$ generated by suitable elements so 
that on the quotient $S(U)=S_1(U)/\mathcal{J}$, an $L(-1)$-derivative property 
and an iterate formula hold for the operators $Y_t(v,x)$. Then $S(U)$ equipped 
with the vertex operator $Y(v,x)=Y_t(v,x)$ for $v\in V$ satisfies all the axioms 
for a $V$-module except that it does not necessarily have a conformal weight 
grading by $L(0)$-eigenvalues satisfying the usual grading restriction 
conditions (recall the definition of module for a vertex operator algebra from 
\cite{FLM}, \cite{FHL}, or \cite{LL}).

However, $S(U)$ does admit an $\N$-grading $S(U)=\coprod_{n\geq 0} S(U)(n)$ 
determined by the properties $S(U)(0)=U$ and for any homogeneous $v\in V$ and 
$n\in\Z$, $v_n=v(n)$ is an operator of degree $\mathrm{wt}\,v-n-1$. This means 
$S(U)$ is an $\N$-gradable weak $V$-module in the sense of \cite{H2} (such 
modules are called simply $V$-modules in \cite{Z}, but we reserve the term 
$V$-module for modules with a conformal weight grading satisfying the grading 
restriction conditions). Note that $U$ is contained in the top level of $S(U)$.
\begin{rema}
 Note that $U$ generates $S(U)$, so when $S(U)$ is semisimple, $T(S(U))=U$. 
However, $S(U)$ is not generally semisimple, and $T(S(U))$ does not generally 
equal $U$.
\end{rema}
The correspondence $U\rightarrow S(U)$ defines a functor from 
$\mathbf{C}^{fin}(A(V))$ to the category of $\N$-gradable weak $V$-modules. The 
fact that an $A(V)$-module 
homomorphism $f: U_1\rightarrow U_2$ extends to a unique $V$-homomorphism $S(f): 
S(U_1)\rightarrow S(U_2)$ amounts to an appropriate 
universal property satisfied by the induced module $S(U_1)$. Now, if every 
$\N$-gradable weak $V$-module is a direct sum of irreducible $V$-modules, then 
for any finite-dimensional $A(V)$-module $U$, $S(U)$ is a direct sum of finitely 
many irreducible $V$-modules since it is generated by the finite-dimensional 
space $U$. In this case, $S$ is a functor from $\mathbf{C}^{fin}(A(V))$ to 
$V-\mathbf{mod}$. Then from \cite{Z} 
and \cite{H2}, we have:
\begin{theo}\label{catequiv}
 If every $\N$-gradable weak $V$-module is a direct sum of irreducible $V$-modules, then the functors
 \begin{equation*}
  T: V-\mathbf{mod}\rightarrow\mathbf{C}^{fin}(A(V))
 \end{equation*}
and
\begin{equation*}
 S: \mathbf{C}^{fin}(A(V))\rightarrow V-\mathbf{mod}
\end{equation*}
are equivalences of categories. More specifically, $T\circ 
S=1_{\mathbf{C}^{fin}(A(V))}$ and $S\circ T$ is naturally isomorphic to 
$1_{V-\mathbf{mod}}$.
\end{theo}

\subsection{Zhu's algebra and intertwining operators}\label{sub:Zhuintwop}

Now we consider intertwining operators among $V$-modules. If a $V$-module $W$ 
is indecomposable, its conformal weights lie in $h+\mathbb{N}$ for some 
$h\in\mathbb{C}$. 
Thus suppose $W_i$ for $i=1,2,3$ are indecomposable $V$-modules whose lowest 
weights are $h_i\in\mathbb{C}$. Then for any intertwining operator 
$\mathcal{Y}\in \mathcal{V}^{W_3}_{W_1 W_2}$, we can 
write (\cite{FHL})
\begin{equation*}
 \mathcal{Y}(w_{(1)},x)w_{(2)}=\sum_{n\in\Z} o^{\mathcal{Y}}_n (w_{(1)}\otimes 
w_{(2)}) x^{h_3-h_1-h_2-n-1}
\end{equation*}
for any $w_{(1)}\in W_1$ and $w_{(2)}\in W_2$; here for any $n\in\Z$,
\begin{equation*}
 o^{\mathcal{Y}}_n(w_{(1)}\otimes w_{(2)})=(w_{(1)})_{n-h_3+h_1+h_2} w_{(2)}
\end{equation*}
for $w_{(1)}\in W_1$ and $w_{(2)}\in W_2$. If $w_{(1)}$ and $w_{(2)}$ are 
homogenous, then
\begin{equation*}
 \mathrm{wt}\,o^{\mathcal{Y}}_n(w_{(1)}\otimes 
w_{(2)})=\mathrm{wt}\,w_{(1)}+\mathrm{wt}\,w_{(2)}+h_3-h_1-h_2-n-1
\end{equation*}
for any $n\in\Z$. In particular, for any $w_{(1)}\in W_1$, we have a linear map
\begin{equation*}
 o_\mathcal{Y}(w_{(1)})\in\mathrm{Hom}((W_2)_{(h_2)}, (W_3)_{(h_3)})
\end{equation*}
defined by
\begin{equation*}
 o_\mathcal{Y}(w_{(1)}): u_{(2)}\mapsto 
o^{\mathcal{Y}}_{\mathrm{wt}\,w_{(1)}-h_1-1}(w_{(1)}\otimes u_{(2)})
\end{equation*}
for homogeneous $w_{(1)}\in W_1$ and $u_{(2)}\in W_{(h_2)}$.

From \cite{FZ}, the linear map $o_\mathcal{Y}(w_{(1)})=0$ when $w_{(1)}\in 
O(W_1)$, so for any 
$\mathcal{Y}\in\mathcal{V}^{W_3}_{W_1 W_2}$, we have an $A(V)$-homomorphism
\begin{equation*}
 \pi(\mathcal{Y}): A(W_1)\otimes_{A(V)} (W_2)_{(h_2)}\rightarrow (W_3)_{(h_3)}
\end{equation*}
defined by
\begin{equation*}
 \pi(\mathcal{Y})((w_{(1)}+O(W_1))\otimes u_{(2)})=o_\mathcal{Y}(w_{(1)})\cdot 
u_{(2)}
\end{equation*}
for $w_{(1)}\in W_1$ and $u_{(2)}\in (W_2)_{(h_2)}$. In certain cases, the map 
$\mathcal{Y}\mapsto\pi(\mathcal{Y})$ is an isomorphism. In fact, from Theorem 
2.1 in \cite{Li2} (which is a correction and generalization of Theorem 1.5.3 in \cite{FZ}) we have
\begin{theo}
 Suppose $W_1\cong S(M_1)$ and $W_2\cong S(M_2)$ where $M_1$ and $M_2$ are 
finite-dimensional irreducible $A(V)$-modules; suppose also that $W_3\cong 
S(M_3)'$ where $M_3$ is a finite-dimensional irreducible $A(V)$-module. Then 
$\pi$ is a linear isomorphism, so as vector spaces,
 \begin{equation*}
  \mathcal{V}^{W_3}_{W_1\,W_2}\cong\mathrm{Hom}_{A(V)}(A(W_1)\otimes_{A(V)} 
(W_2)_{(h_2)}, 
(W_3)_{(h_3)}).
 \end{equation*}
\end{theo}

When every $\N$-gradable weak $V$-module is a direct sum of irreducible $V$-modules, both irreducible modules and 
their (irreducible) contragredients are isomorphic to modules $S(M)$ where $M$ 
is a finite-dimensional irreducible $A(V)$-module. Moreover, the lowest weight 
space of an irreducible module $W$ is $T(W)$. Thus we have:
\begin{corol}\label{intwopcorol}
 Assume that every $\N$-gradable weak $V$-module is a direct sum of irreducible $V$-modules and that $W_1$, $W_2$, and $W_3$ are 
$V$-modules. Then as vector spaces,
 \begin{equation*}
  \mathcal{V}^{W_3}_{W_1 W_2}\cong\mathrm{Hom}_{A(V)}(A(W_1)\otimes_{A(V)} 
T(W_2), 
T(W_3)).
 \end{equation*}
\end{corol}

We can use Corollary \ref{intwopcorol} to identify the $P(z)$-tensor product of 
$W_1$ and $W_2$. In particular, suppose that the $A(V)$-module 
$A(W_1)\otimes_{A(V)} T(W_2)$ is finite-dimensional. Then the identity on 
$A(W_1)\otimes_{A(V)} T(W_2)$ induces an intertwining operator 
$\mathcal{Y}_\boxtimes$ of type $\binom{S(A(W_1)\otimes_{A(V)} 
T(W_2))}{W_1\,\,\,\,\,W_2}$. Moreover, recalling Proposition \ref{opmapiso}, we 
have for any $z\in\C^\times$ a $P(z)$-intertwining map 
$I_{\mathcal{Y}_\boxtimes,0}$. Then we have (see also Theorem 6.3.3 in 
\cite{Li1}):
\begin{propo}\label{tensprodpropo}
 In the setting of Corollary \ref{intwopcorol}, assume that the $A(V)$-module 
$A(W_1)\otimes_{A(V)} T(W_2)$ is finite dimensional. Then for any 
$z\in\mathbb{C}^\times$, $(S(A(W_1)\otimes_{A(V)} T(W_2)), 
I_{\mathcal{Y}_\boxtimes, 0})$ is a $P(z)$-tensor product of $W_1$ and $W_2$.
\end{propo}
\begin{proof}
 Let us use $M_{1,2}$ to denote the finite-dimensional $A(V)$-module 
$A(W_1)\otimes_{A(V)} T(W_2)$. We need to check that $(S(M_{1,2}), 
I_{\mathcal{Y}_\boxtimes,0})$ satisfies the universal property of a 
$P(z)$-tensor product. Thus suppose $W_3$ is a $V$-module and $I$ is an 
intertwining map of type $\binom{W_3}{W_1\,W_2}$. We need to show that there is 
a unique $V$-module homomorphism $\eta: S(M_{1,2})\rightarrow W_3$ such that 
$\overline{\eta}\circ I_{\mathcal{Y}_\boxtimes,0}=I$.
 
 Recalling Proposition \ref{opmapiso}, we have the intertwining operator 
$\mathcal{Y}_{I,0}$ of type $\binom{W_3}{W_1\,W_2}$, and we have the 
$A(V)$-homomorphism
 \begin{equation*}
  \pi(\mathcal{Y}_{I,0}): M_{1,2}\rightarrow T(W_3).
 \end{equation*}
Suppose also that $\tau: S(T(W_3))\rightarrow W_3$ is the unique $V$-module 
isomorphism that equals the identity on $T(W_3)$. We set $\eta=\tau\circ 
S(\pi(\mathcal{Y}_{I,0}))$. 

Now, $\eta\circ\mathcal{Y}_\boxtimes$ and $\mathcal{Y}_{I,0}$ are two 
intertwining operators of type $\binom{W_3}{W_1\,W_2}$. Moreover, it is clear 
from the definitions that
\begin{equation*}
 \pi(\eta\circ\mathcal{Y}_\boxtimes)=\pi(\mathcal{Y}_{I,0})
\end{equation*}
as $A(V)$-module homomorphisms from $M_{1,2}$ to $T(W_3)$. Since $\pi$ is an 
isomorphism by Corollary \ref{intwopcorol}, it follows that 
$\eta\circ\mathcal{Y}_{\boxtimes}=\mathcal{Y}_{I,0}$. Then
\begin{equation*}
 \overline{\eta}\circ 
I_{\mathcal{Y}_\boxtimes,0}=\overline{\eta}\circ\mathcal{Y}_\boxtimes (\cdot, 
e^{\mathrm{log}\,z})\cdot =\mathcal{Y}_{I,0}(\cdot,e^{\mathrm{log}\,z})\cdot =I
\end{equation*}
as well. To show the uniqueness of $\eta$, note that if $\overline{\eta}\circ 
I_{\mathcal{Y}_\boxtimes,0}=I$, then Proposition \ref{opmapiso} implies that 
$\eta\circ\mathcal{Y}_\boxtimes=\mathcal{Y}_{I,0}$ as well. Thus 
$\pi(\eta\circ\mathcal{Y}_\boxtimes)=\pi(\mathcal{Y}_{I,0})$, which forces 
$T(\eta)=\pi(\mathcal{Y}_{I,0})$. Therefore we have
\begin{equation*}
 \eta=\tau\circ S(T(\eta))=\tau\circ S(\pi(\mathcal{Y}_{I,0})),
\end{equation*}
as desired.
\end{proof} 

\subsection{An equivalence of tensor categories}\label{sub:tensequiv}

Here we work in the setting of Proposition \ref{tensprodpropo}, that is, we 
assume that every $\N$-gradable weak $V$-module is a direct sum of irreducible $V$-modules and that $A(W_1)\otimes_{A(V)} T(W_2)$ is 
finite dimensional for any $V$-modules $W_1$ and $W_2$. We also assume that 
$V-\mathbf{mod}$ equipped with the tensor product $\boxtimes_{P(1)}$ is a tensor 
category with associativity isomorphisms $\mathcal{A}$, unit object $V$, and 
left and right unit isomorphisms $l$ and $r$ (see Section 12.2 in \cite{HLZ8} 
for a detailed description of this tensor category structure). We discuss 
precisely how the functor $T$ induces tensor category structure on 
$\mathbf{C}^{fin}(A(V))$, omitting proofs because they are straightforward.

As in Proposition \ref{tensprodpropo}, for $V$-modules $W_1$ and $W_2$, we take
\begin{equation*}
 W_1\boxtimes_{P(1)} W_2=S(A(W_1)\otimes_{A(V)} T(W_2)),
\end{equation*}
and we use $I_{\mathcal{Y}_\boxtimes,0}$ as the tensor product 
$P(1)$-intertwining map. Also, for morphisms $f_1: 
W_1\rightarrow\widetilde{W}_1$ and $f_2: W_2\rightarrow\widetilde{W}_2$ in 
$V-\mathbf{mod}$, the tensor product morphism is induced by the universal 
property of the tensor product: $f_1\boxtimes_{P(1)} f_2$ is the unique morphism 
such that
\begin{equation*}
 \overline{f_1\boxtimes_{P(1)} f_2}(\mathcal{Y}_{\boxtimes}(\cdot, 
1)\cdot)=\widetilde{\mathcal{Y}}_{\boxtimes}(f_1(\cdot),1)f_2(\cdot),
\end{equation*}
where $\widetilde{\mathcal{Y}}_\boxtimes$ denotes the tensor product 
intertwining operator of type 
$\binom{\widetilde{W}_1\boxtimes_{P(1)}\widetilde{W}_2}{\widetilde{W}_1\,\,\,
\widetilde{W}_2}$. By the definition of $\mathcal{Y}_\boxtimes$ and 
$\widetilde{\mathcal{Y}}_\boxtimes$, this means $f_1\boxtimes_{P(1)} f_2$ is the 
morphism
\begin{equation*}
 f_1\boxtimes_{P(1)} f_2=S(A(f_1)\otimes T(f_2)): S(A(W_1)\otimes_{A(V)} 
T(W_2))\rightarrow S(A(\widetilde{W}_1)\otimes_{A(V)} T(\widetilde{W}_2)),
\end{equation*}
where $A(f_1): A(W_1)\rightarrow A(\widetilde{W}_1)$ is the $A(V)$-bimodule 
homomorphism induced by $f_1$. 

Now we can define a tensor product bifunctor $\boxtimes$ on 
$\mathbf{C}^{fin}(A(V))$ as follows. For finite-dimensional $A(V)$-modules $U_1$ 
and $U_2$, define
\begin{equation*}
 U_1\boxtimes U_2=A(S(U_1))\otimes_{A(V)} U_2,
\end{equation*}
a finite-dimensional $A(V)$-module by our assumptions. Also, for morphisms $f_1: 
U_1\rightarrow\widetilde{U}_1$ and $f_2: U_2\rightarrow\widetilde{U}_2$ in 
$\mathbf{C}^{fin}(A(V))$, we define 
\begin{equation}\label{tensmorphA}
 f_1\boxtimes f_2=A(S(f_1))\otimes f_2: A(S(U_1))\otimes_{A(V)} U_2\rightarrow 
A(S(\widetilde{U}_1))\otimes_{A(V)}\widetilde{U}_2.
\end{equation}
Note that with these definitions,
\begin{equation}\label{Sboxtimes}
S(U_1\boxtimes U_2)= S(U_1)\boxtimes_{P(1)} S(U_2),
\end{equation}
so because $T\circ S=1_{\mathbf{C}^{fin}(A(V))}$,
\begin{equation}\label{boxtimes}
 U_1\boxtimes U_2=T(S(U_1)\boxtimes_{P(1)} S(U_2))
\end{equation}
for any finite-dimensional $A(V)$-modules $U_1$ and $U_2$.

The relations \eqref{Sboxtimes} and \eqref{boxtimes} imply that for objects 
$U_1$, $U_2$, and $U_3$ in $\mathbf{C}^{fin}(A(V))$,
\begin{equation*}
 U_1\boxtimes(U_2\boxtimes 
U_3)=T(S(U_1)\boxtimes_{P(1)}(S(U_2)\boxtimes_{P(1)}S(U_3)))
\end{equation*}
and
\begin{equation*}
 (U_1\boxtimes U_2)\boxtimes U_3=T((S(U_1)\boxtimes_{P(1)} 
S(U_2))\boxtimes_{P(1)}S(U_3)).
\end{equation*}
Thus we can take the associativity isomorphism for $U_1$, $U_2$, and $U_3$ to be
\begin{equation*}
 \mathcal{A}_{U_1,U_2,U_3}=T(\mathcal{A}_{S(U_1),S(U_2),S(U_3)}).
\end{equation*}
That these associativity isomorphisms give a natural isomorphism from 
$\boxtimes\circ(1_{\mathbf{C}^{fin}(A(V))}\times\boxtimes)$ to 
$\boxtimes\circ(\boxtimes\times 1_{\mathbf{C}^{fin}(A(V))})$ and satisfy the 
pentagon axiom follows easily from these properties for the associativity 
isomorphisms in $V-\mathbf{mod}$.

Now, as in the proof of Proposition \ref{tensprodpropo}, let $\tau_W$ for a 
$V$-module $W$ denote the unique isomorphism from $S(T(W))$ to $W$ that equals 
the identity on the top level $T(W)$. Notice that for $V$-modules $W_1$ and 
$W_2$,
\begin{equation*}
 T(W_1)\boxtimes T(W_2)=A(S(T(W_1)))\otimes_{A(V)} T(W_2)
\end{equation*}
while 
\begin{equation*}
 T(W_1\boxtimes_{P(1)}W_2)=A(W_1)\otimes_{A(V)} T(W_2).
\end{equation*}
Then we can define
\begin{equation*}
 M_{W_1,W_2}: T(W_1)\boxtimes T(W_2)\rightarrow T(W_1\boxtimes_{P(1)} W_2)
\end{equation*}
to be $M_{W_1,W_2}=A(\tau_{W_1})\otimes 1_{T(W_2)}$. Thus we obtain a natural 
isomorphism $M$ from $\boxtimes\circ(T\times T)$ to $T\circ\boxtimes_{P(1)}$.

Next, we take the unit object of $\mathbf{C}^{fin}(A(V))$ to be $T(V)$. The 
natural isomorphism $M$ then allows us to define unit isomorphisms in 
$\mathbf{C}^{fin}(A(V))$. In particular, for a finite-dimensional $A(V)$-module 
$U$, we define the left unit isomorphism $l_U$ to be the composition
\begin{equation}\label{leftunit}
 T(V)\boxtimes U=T(V)\boxtimes T(S(U))\xrightarrow{M_{V,S(U)}} 
T(V\boxtimes_{P(1)} S(U))\xrightarrow{T(l_{S(U)})} T(S(U))=U
\end{equation}
and the right unit isomorphism $r_U$ to be the composition
\begin{equation}\label{rightunit}
 U\boxtimes T(V)=T(S(U))\boxtimes T(V)\xrightarrow{M_{S(U),V}} 
T(S(U)\boxtimes_{P(1)} V)\xrightarrow{T(r_{S(U)})} T(S(U))=U.
\end{equation}

We recall the notion of equivalence of tensor categories from, for example, 
\cite{BK} or \cite{Ka}. Then it is straightforward to prove:
\begin{theo}\label{tensequiv}
 Under the assumptions of this subsection, the category $\mathbf{C}^{fin}(A(V))$ 
equipped with the tensor product $\boxtimes$, unit object $T(V)$, and 
associativity and unit isomorphisms as described above is a tensor category. 
Moreover, $(T, M, 1_{T(V)})$ defines a tensor equivalence from $V-\mathbf{mod}$ 
to $\mathbf{C}^{fin}(A(V))$.
\end{theo}

\subsection{Application to $\imzero$}\label{sub:apptoghat}

We now apply the results of this section to the vertex operator algebra 
$\imzero$ when $\ell\in\mathbb{N}$. It was shown in Theorem 3.1.3 of \cite{FZ} that every $\N$-gradable weak $\imzero$-module is a direct sum of irreducible $\imzero$-modules; in particular, $\imzero-\mathbf{mod}$ is semisimple. First, from 
\cite{FZ},
\begin{equation}\label{zalg}
 A(\imzero)\cong U(\g)/\langle x_\theta^{\ell+1}\rangle
\end{equation}
as an associative algebra, where $x_\theta$ is a root vector for the longest 
root $\theta$ of $\g$; the isomorphism is determined by
\begin{equation*}
 g+\langle x_\theta^{\ell+1}\rangle\mapsto g(-1)\mathbf{1}+O(\imzero)
\end{equation*}
for $g\in\g$. Thus if we set $\mathbf{D}(\g,\ell)$ to be the 
category of finite-dimensional $A(\imzero)$-modules, it is clear that 
$\mathbf{D}(\g,\ell)$ is simply the full subcategory of finite-dimensional 
$\g$-modules on which $x_\theta^{\ell+1}$ acts trivially. By Theorem 
\ref{catequiv}, the functor
\begin{equation*}
 T: \imzero-\mathbf{mod}\rightarrow\mathbf{D}(\g,\ell)
\end{equation*}
is an equivalence of categories. Note that if $U$ is an object of 
$\mathbf{D}(\g,\ell)$, then $T(\imu)=U$ and $S(U)\cong\imu$.

Next from \cite{FZ}, for an irreducible $\imzero$-module 
$L_{\ghat}(\ell,\lambda)$, the corresponding $A(\imzero)$-bimodule 
is
\begin{equation}\label{bimod}
 A(L_{\ghat}(\ell,\lambda))\cong(L_\lambda\otimes U(\g))/\langle 
v_\lambda\otimes 
x_\theta^{\ell-\langle\lambda,\theta\rangle+1}\rangle,
\end{equation}
where $v_\lambda$ is a highest weight vector of $L_\lambda$ and 
$\langle\cdot\rangle$ indicates the sub-bimodule generated by an element; the 
isomorphism is determined by
\begin{equation}\label{bimodiso}
 u\otimes g_1\cdots g_n+\langle v_\lambda\otimes 
x_\theta^{\ell-\langle\lambda,\theta\rangle+1}\rangle\mapsto g_n(-1)\cdots 
g_1(-1)u+O(\imlambda)
\end{equation}
for $u\in L_\lambda$ and $g_1,\ldots, g_n\in\g$. The 
$A(\imzero)\cong U(\g)/\langle x_\theta^{\ell+1}\rangle$-bimodule structure on 
$A(L_{\ghat}(\ell,\lambda))$ is induced by the following $U(\g)$-bimodule 
structure on 
$L_\lambda\otimes U(\g)$:
\begin{equation}\label{leftaction}
 x\cdot(v\otimes y)=(x\cdot v)\otimes y+v\otimes xy
\end{equation}
for $x\in\g$, $y\in U(\g)$, $v\in L_\lambda$, and
\begin{equation}\label{rightaction}
 (v\otimes y)\cdot x=v\otimes yx.
\end{equation}

We can now identify the tensor product in $\mathbf{D}(\g,\ell)$ with a quotient 
of the usual tensor product. For objects $U_1$ and $U_2$ of 
$\mathbf{D}(\g,\ell)$, let $W_{U_1,U_2}^{(\ell)}$ denote the $\g$-submodule of 
$U_1\otimes U_2$ generated by vectors of the form $v_\lambda\otimes 
x_\theta^{\ell-\langle\lambda,\theta\rangle+1}\cdot w$ where $v_\lambda$ is any 
highest weight vector (of weight $\lambda$) in $U_1$ and $w$ is any vector in 
$U_2$.
\begin{propo}\label{identboxtimes}
For objects $U_1$ and $U_2$ in $\mathbf{D}(\g,\ell)$, there is an 
$A(\imzero)\cong U(\g)/\langle x_\theta^{\ell+1}\rangle$ isomorphism
 \begin{align}\label{identboxtimesiso}
  \Phi_{U_1,U_2}: (U_1\otimes U_2)/W^{(\ell)}_{U_1,U_2} & \rightarrow 
A(S(U_1))\otimes_{A(\imzero)} U_2\nonumber\\
 u_{(1)}\otimes u_{(2)}+W^{(\ell)}_{U_1,U_2} & \mapsto 
(u_{(1)}+O(S(U_1)))\otimes u_{(2)}
\end{align}
\end{propo}
\begin{proof}
Since finite-dimensional $\g$-modules are completely reducible, we may take 
$U_1=L_{\lambda_1}$ where $\langle\lambda_1,\theta\rangle\leq\ell$. Then using 
the identifications \eqref{bimod}, \eqref{bimodiso}, \eqref{leftaction}, and 
\eqref{rightaction}, we have a $\g$-module homomorphism
\begin{align*}
 \widetilde{\Phi}_{U_1,U_2}: U_1\otimes U_2 & \rightarrow (U_1\otimes 
U(\g))/\langle v_\lambda\otimes x_\theta^{\ell-\langle\lambda_1,\theta\rangle+1} 
\rangle\otimes U_2\nonumber\\
 u_{(1)}\otimes u_{(2)} &\mapsto (u_{(1)}\otimes 1+\langle v_\lambda\otimes 
x_\theta^{\ell-\langle\lambda_1,\theta\rangle+1} \rangle)\otimes u_{(2)}
\end{align*}
which contains $W^{(\ell)}_{U_1,U_2}$ in its kernel, so we have the desired 
homomorphism $\Phi_{U_1,U_2}$. It is straightforward to show that 
$\Phi_{U_1,U_2}$ has an inverse, so it is an isomorphism. See Lemma 4.1 in 
\cite{M} for more details; see also Theorem 3.2.3 in \cite{FZ} and the related Theorem 6.5 in \cite{FF}.
\end{proof}

Due to the preceding proposition, we can take
\begin{equation}\label{gboxtimes}
 U_1\boxtimes U_2=(U_1\otimes U_2)/W^{(\ell)}_{U_1,U_2}
\end{equation}
for $\g$-modules $U_1$ and $U_2$ in $\mathbf{D}(\g,\ell)$, recalling the 
$\boxtimes$ notation of the previous subsection. Note that if $f_1: 
U_1\rightarrow\widetilde{U}_1$ and $f_2: U_2\rightarrow\widetilde{U}_2$ are 
$\g$-module homomorphisms in $\mathbf{D}(\g,\ell)$, $f_1\otimes f_2$ takes 
generators of $W^{(\ell)}_{U_1,U_2}$ to generators of 
$W^{(\ell)}_{\widetilde{U}_1,\widetilde{U}_2}$, so $f_1\otimes f_2$ induces a 
$\g$-module homomorphism
\begin{equation*}
 f_1\boxtimes f_2: U_1\boxtimes 
U_2\rightarrow\widetilde{U}_1\boxtimes\widetilde{U}_2.
\end{equation*}
Then \eqref{tensmorphA} and \eqref{identboxtimesiso} show that the isomorphisms 
$\Phi_{U_1,U_2}$ define a natural isomorphism.
\begin{rema}
 Our realization \eqref{gboxtimes} of $U_1\boxtimes U_2$ is not especially useful for calculating fusion rules, that is, the multiplicities of irreducible $\g$-modules in $U_1\boxtimes U_2$, since the submodule $W^{(\ell)}_{U_1, U_2}$ is usually difficult to calculate explicitly (but note that Theorem 6.2 in \cite{FF} is an interesting formula for fusion rules likewise derived using results in \cite{FZ}). The main reason we use \eqref{gboxtimes} is that it allows us to realize the triple tensor products $U_1\boxtimes(U_2\boxtimes U_3)$ and $(U_1\boxtimes U_2)\boxtimes U_3$, for objects $U_1$, $U_2$, and $U_3$ of $\mathbf{D}(\g,\ell)$, as quotients of $U_1\otimes U_2\otimes U_3$. In the following sections, we will use this to realize the associativity isomorphism $\mathcal{A}_{U_1,U_2,U_3}$ as the isomorphism induced from a certain automorphism of $U_1\otimes U_2\otimes U_3$.
\end{rema}

Now for $\imzero$-modules $W_1$ and $W_2$, we recall from Proposition 
\ref{tensprodpropo} the intertwining operator $\mathcal{Y}_\boxtimes$ of type 
$\binom{S(A(W_1)\otimes_{A(\imzero)} T(W_2))}{W_1\,\,\,\,\,W_2}$ and obtain:
\begin{propo}
 If $W_1$ and $W_2$ are $\imzero$-modules and $z\in\C^\times$, 
$S(A(W_1)\otimes_{A(\imzero)} T(W_2))$ equipped with the $P(z)$-intertwining map 
$I_{\mathcal{Y}_\boxtimes,0}$ is a $P(z)$-tensor product of $W_1$ and $W_2$.
\end{propo}

Thus in the tensor category of $\imzero$-modules given by \cite{HL1}-\cite{HL4}, 
\cite{H1}, we may take the tensor product $\boxtimes_{P(1)}$ of 
$\imzero$-modules $W_1$ and $W_2$ to be
\begin{equation*}
 W_1\boxtimes_{P(1)} W_2=S(A(W_1)\otimes_{A(\imzero)} T(W_2)))\cong 
L_{\ghat}(\ell, T(W_1)\boxtimes T(W_2))
\end{equation*}
equipped with the $P(1)$-intertwining map
\begin{equation*}
 I_{\mathcal{Y}_\boxtimes,0}=\mathcal{Y}_\boxtimes(\cdot,1)\cdot.
\end{equation*}
Moreover, by Theorem \ref{tensequiv} and the natural isomorphism of Proposition 
\ref{identboxtimes}, the equivalence of categories $T$ is in fact an equivalence 
of tensor categories, where $\mathbf{D}(\g,\ell)$ is equipped with the tensor 
product $\boxtimes$ of \eqref{gboxtimes}. 

The unit object of $\mathbf{D}(\g,\ell)$ is $T(\imzero)=\C\mathbf{1}$, the 
trivial one-dimensional $\g$-module. It is easy to see from \eqref{leftunit}, 
\eqref{rightunit}, \eqref{identboxtimesiso}, and the definition of the unit 
isomorphisms in $\imzero-\mathbf{mod}$ (see for example Section 12.2 in 
\cite{HLZ8}), that the unit isomorphisms in $\mathbf{D}(\g,\ell)$ are the 
obvious ones
\begin{align*}
 l_U: \C\mathbf{1}\boxtimes U=\C\mathbf{1}\otimes U & \rightarrow U\nonumber\\
 \mathbf{1}\otimes u & \mapsto u
\end{align*}
and
\begin{align*}
 r_U: U\boxtimes\C\mathbf{1}=U\otimes\C\mathbf{1} & \rightarrow U\nonumber\\
 u\otimes\mathbf{1} & \mapsto u
\end{align*}
for $\g$-modules $U$ in $\mathbf{D}(\g,\ell)$. (Note that both 
$W^{(\ell)}_{\C\mathbf{1},U}$ and $W^{(\ell)}_{U,\C\mathbf{1}}$ are zero.)
The remainder of this paper is devoted to a description of the associativity 
isomorphisms in $\mathbf{D}(\g,\ell)$.

\section{The Knizhnik-Zamolodchikov equations}

The goal of this section and the next is to identify the associativity 
isomorphisms in $\mathbf{D}(\g,\ell)$ with certain isomorphisms of $\g$-modules 
obtained from solutions to Knizhnik-Zamolodchikov (KZ) equations (\cite{KZ}; see 
also \cite{HL4}).

\subsection{Formal KZ equations}\label{KZeqns}

We start by deriving a system of (formal) differential equations satisfied by an 
iterate of intertwining operators. Such equations were first derived in 
\cite{KZ}. We will also need a similar system of equations for a product of 
intertwining operators; we will not derive them here, since a vertex algebraic 
derivation may be found in \cite{HL4}.

First we consider the action of $L(-1)$ on a $\imzero$-module $W$. If $w\in W$ 
satisfies $g(n)w=0$ for any $g\in\g$ and $n>0$, then \eqref{L(-1)} implies
\begin{equation}\label{L-1}
 L(-1)w=\dfrac{1}{\ell+h^\vee}\sum_{i=1}^{\mathrm{dim}\,\g} 
\gamma_i(-1)\gamma_i(0)w.
\end{equation}
We shall use the $L(-1)$-derivative property \eqref{intwopderiv} for 
intertwining operators and \eqref{L-1} to derive the KZ equations for an iterate 
of intertwining operators.

We will need the commutator formula \eqref{commform} for intertwining operators. 
Suppose $W_1$, $W_2$, and $W_3$ are $\imzero$-modules and $\mathcal{Y}\in 
\mathcal{V}^{W_3}_{W_1\,W_2}$. If we take $v=g(-1)\mathbf{1}$ for $g\in\g$ in 
\eqref{commform}, we obtain
\begin{align}\label{intwopcomm}
 [g(x_1), \mathcal{Y}(w_{(1)},x_2)]
& =\mathrm{Res}_{x_0}\, x_2^{-1}\delta \left(\frac{x_1-x_0}{x_2}\right) 
\mathcal{Y}(g(x_0)w_{(1)},x_2)\nonumber\\
& =\sum_{i\geq 0}\dfrac{(-1)^i}{i!}\left(\dfrac{\partial}{\partial 
x_1}\right)^i\left(x_2^{-1}\delta \left(\frac{x_1}{x_2}\right)\right) 
\mathcal{Y}(g(i)w_{(1)},x_2)
\end{align}
for $w_{(1)}\in W_1$. If we further extract the coefficient of $x_1^{-n-1}$ in 
\eqref{intwopcomm}, we obtain
 \begin{align}\label{intwopcomm2}
  [g(n), \mathcal{Y}(w_{(1)},x_2)]=\sum_{i\geq 
0}  \binom{n}{i} x_2^{n-i} \mathcal{Y}(g(i)w_{(1)},x_2)
 \end{align}
 for $w_{(1)}\in W_1$. When $w_{(1)}\in W_1$ satisfies $g(i)w=0$ for $i>0$, 
\eqref{intwopcomm2} simplifies to
 \begin{equation}\label{intwopcomm3}
  [g(n),\mathcal{Y}(w_{(1)},x)]=x^n\mathcal{Y}(g(0)w_{(1)},x)
 \end{equation}
for any $g\in\g$ and $n\in\Z$. We shall also need the iterate formula for 
intertwining operators. When we take $v=g(-1)\mathbf{1}$ for $g\in\g$ in 
\eqref{itform}, we obtain
\begin{align}\label{intwopit}
 \mathcal{Y}(g(n)w_{(1)}, x_2) & =\mathrm{Res}_{x_1}\,\left( (x_1-x_2)^n 
g(x_1)\mathcal{Y}(w_{(1)}, x_2)-(-x_2+x_1)^n \mathcal{Y}(w_{(1)},x_2) 
g(x_1)\right)\nonumber\\
 & =\sum_{i\geq 0} (-1)^i\binom{n}{i}(x_2^i\, 
g(n-i)\mathcal{Y}(w_{(1)},x_2)-(-1)^n x_2^{n-i} \mathcal{Y}(w_{(1)},x_2)g(i))
\end{align}
for $w_{(1)}\in W_1$.

Now suppose $W_1$, $W_2$, $W_3$, $W_4$ and $M_2$ are $\imzero$-modules, 
$\mathcal{Y}^1$ is an intertwining operator of type $\binom{W_4}{M_2\,W_3}$, and 
$\mathcal{Y}^2$ is an intertwining operator of type $\binom{M_2}{W_1\,W_2}$. We 
use \eqref{intwopit} to obtain:
\begin{lemma}\label{aon1and2}
 For $u_{(1)}\in T(W_1)$, $u_{(2)}\in T(W_2)$, $u_{(3)}\in T(W_3)$, $u_{(4)}'\in 
T(W_4')$, and $g\in\g$,
 \begin{align}\label{aon1}
  \langle u_{(4)}',\mathcal{Y}^1(\mathcal{Y}^2(g(-1)u_{(1)}, x_0)u_{(2)}, & 
x_2)u_{(3)}\rangle = x_0^{-1}\langle 
u_{(4)}',\mathcal{Y}^1(\mathcal{Y}^2(u_{(1)}, 
x_0)g(0)u_{(2)},x_2)u_{(3)}\rangle\nonumber\\
  & +(x_2+x_0)^{-1}\langle u_{(4)}',\mathcal{Y}^1(\mathcal{Y}^2(u_{(1)}, 
x_0)u_{(2)},x_2)g(0)u_{(3)}\rangle
 \end{align}
and
\begin{align}\label{aon2}
 \langle u_{(4)}',\mathcal{Y}^1(\mathcal{Y}^2(u_{(1)}, 
x_0)g(-1)u_{(2)},x_2)u_{(3)}\rangle & =-x_0^{-1}\langle 
u_{(4)}',\mathcal{Y}^1(\mathcal{Y}^2(g(0)u_{(1)}, 
x_0)u_{(2)},x_2)u_{(3)}\rangle\nonumber\\
&\;\;\;\;+x_2^{-1} \langle u_{(4)}',\mathcal{Y}^1(\mathcal{Y}^2(u_{(1)}, 
x_0)u_{(2)},x_2)g(0)u_{(3)}\rangle.
\end{align}
\end{lemma}
\begin{proof}
 To prove \eqref{aon1}, we first note that \eqref{intwopit} and the fact that 
$g(i)u_{(2)}=0$ for $i>0$ implies that
 \begin{equation}\label{firstterm}  
\mathcal{Y}^2(g(-1)u_{(1)},x_0)u_{(2)}=x_0^{-1}\mathcal{Y}^2(u_{(1)},x_0)g(0)u_{
(2)}+\sum_{i\geq 0} x_0^i\,g(-i-1)\mathcal{Y}^2(u_{(1)},x_0)u_{(2)}.
 \end{equation}
Then applying \eqref{intwopit} again as well as the fact that $g(j)u_{(3)}=0$ 
for $j>0$,
\begin{align}\label{secondterm}
 \mathcal{Y}^1(g(-i-1)\mathcal{Y}^2( & u_{(1)},x_0) u_{(2)},x_2)u_{(3)} =(-1)^i 
x_2^{-i-1}\mathcal{Y}^1(\mathcal{Y}^2(u_{(1)},x_0)u_{(2)},x_2)g(0)u_{(3)}
\nonumber\\
 & +\sum_{j\geq 0} (-1)^j\binom{-i-1}{j} 
x_2^j\,g(-i-j-1)\mathcal{Y}^1(\mathcal{Y}^2(u_{(1)},x_0)u_{(2)},x_2)u_{(3)}.
\end{align}
Putting \eqref{firstterm} and \eqref{secondterm} together,
\begin{align}\label{together}
 \mathcal{Y}^1(\mathcal{Y}^2( & g(-1)u_{(1)}, x_0)u_{(2)},x_2)u_{(3)} = 
x_0^{-1}\mathcal{Y}^1(\mathcal{Y}^2(u_{(1)}, 
x_0)g(0)u_{(2)},x_2)u_{(3)}\nonumber\\
 &+\sum_{i\geq 0} (-1)^i x_2^{-1-i} x_0^i \mathcal{Y}^1(\mathcal{Y}^2(u_{(1)}, 
x_0)u_{(2)},x_2)g(0)u_{(3)}\nonumber\\
 & +\sum_{i\geq 0}\sum_{j\geq 0} (-1)^j\binom{-i-1}{j} 
x_2^j\,g(-i-j-1)\mathcal{Y}^1(\mathcal{Y}^2(u_{(1)},x_0)u_{(2)},x_2)u_{(3)}.
\end{align}
Since the adjoint operator of $g(-i-j-1)$ is $-g(i+j+1)$, the third term on the 
right side of \eqref{together} disappears when \eqref{together} is paired with 
$u_{(4)}'\in T(W_4')$. Since also
\begin{equation*}
 \sum_{i\geq 0} (-1)^i x_2^{-i-1} x_0^i=(x_2+x_0)^{-1},
\end{equation*}
\eqref{aon1} follows.

To prove \eqref{aon2}, we first use \eqref{intwopcomm3} to obtain
\begin{equation}\label{firstterm'} 
\mathcal{Y}^2(u_{(1)},x_0)g(-1)u_{(2)}=g(-1)\mathcal{Y}^2(u_{(1)},x_0)u_{(2)}
-x_0^{-1}\mathcal{Y}^2(g(0)u_{(1)},x_0)u_{(2)}.
\end{equation}
Then by \eqref{intwopit} and the fact that $g(i)u_{(3)}=0$ for $i>0$,
\begin{align}\label{secondterm'}
 \mathcal{Y}^1(g(-1)\mathcal{Y}^2( u_{(1)},x_0) & 
u_{(2)},x_2)u_{(3)}=x_2^{-1}\mathcal{Y}^1(\mathcal{Y}^2(u_{(1)},x_0)u_{(2)},
x_2)g(0)u_{(3)}\nonumber\\
 & +\sum_{i\geq 0} 
x_2^i\,g(-i-1)\mathcal{Y}^1(\mathcal{Y}^2(u_{(1)},x_0)u_{(2)},x_2)u_{(3)}.
\end{align}
Putting \eqref{firstterm'} and \eqref{secondterm'} together,
\begin{align}\label{together'}
 \mathcal{Y}^1( & \mathcal{Y}^2( u_{(1)},x_0) g(-1)u_{(2)},x_2)u_{(3)} 
=-x_0^{-1}\mathcal{Y}^1(\mathcal{Y}^2( g(0)u_{(1)},x_0)  
u_{(2)},x_2)u_{(3)}\nonumber\\
 & + 
x_2^{-1}\mathcal{Y}^1(\mathcal{Y}^2(u_{(1)},x_0)u_{(2)},x_2)g(0)u_{(3)}+\sum_{
i\geq 0} 
x_2^i\,g(-i-1)\mathcal{Y}^1(\mathcal{Y}^2(u_{(1)},x_0)u_{(2)},x_2)u_{(3)}.
\end{align}
Since the adjoint of the operator $g(-i-1)$ is $-g(i+1)$, the third term on the 
right side of \eqref{together'} disappears when \eqref{together'} is paired with 
$u_{(4)}'\in T(W_4')$, and \eqref{aon2} follows.
\end{proof}

We now define an operator $\Omega_{1,2}$ on $(T(W_1)\otimes T(W_2)\otimes 
T(W_3))^*$ by
\begin{equation}\label{omega12}
 (\Omega_{1,2} F)(u_{(1)}\otimes u_{(2)}\otimes 
u_{(3)})=\sum_{i=1}^{\mathrm{dim}\,\g} 
F(\gamma_i(0)u_{(1)}\otimes\gamma_i(0)u_{(2)}\otimes u_{(3)})
\end{equation}
for any $F\in (T(W_1)\otimes T(W_2)\otimes T(W_3))^*$, recalling 
that $\lbrace \gamma_i\rbrace$ is an orthonormal basis for $\g$. We analogously 
define operators $\Omega_{1,3}$ and $\Omega_{2,3}$ on $(T(W_1)\otimes 
T(W_2)\otimes T(W_3))^*$ in the obvious way. The operators $\Omega_{1,2}$, 
$\Omega_{1,3}$, and $\Omega_{2,3}$ also have natural extensions to operators on 
$(T(W_1)\otimes T(W_2)\otimes T(W_3))^*\lbrace x_0,x_2\rbrace$. We can now 
derive the KZ equations:
\begin{theo}
 For any $u_{(4)}'\in T(W_4')$, define
 \begin{equation*}
\varphi_{\mathcal{Y}^1,\mathcal{Y}^2}(u_{(4)}')=\langle 
u_{(4)}',\mathcal{Y}^1(\mathcal{Y}^2(\cdot,x_0)\cdot,x_2)\cdot\rangle\in 
(T(W_1)\otimes T(W_2)\otimes T(W_3))^*\lbrace x_0,x_2\rbrace.
\end{equation*}
Then $\varphi_{\mathcal{Y}^1,\mathcal{Y}^2}(u_{(4)}')$ satisfies the system of 
formal differential equations
\begin{align}
 (\ell+h^\vee)\dfrac{\partial\varphi}{\partial x_0} & 
=\left(\dfrac{\Omega_{1,2}}{x_0}+\dfrac{\Omega_{1,3}}{x_2+x_0}
\right)\varphi\label{KZit0}\\
 (\ell+h^\vee)\dfrac{\partial\varphi}{\partial x_2} & 
=\left(\dfrac{\Omega_{2,3}}{x_2}+\dfrac{\Omega_{1,3}}{x_2+x_0}
\right)\varphi.\label{KZit2}
\end{align}
Similarly, suppose $M_1$ is an $\imzero$-module, $\mathcal{Y}_1$ is an 
intertwining operator of type $\binom{W_4}{W_1\,M_1}$ and $\mathcal{Y}_2$ is an 
intertwining operator of type $\binom{M_1}{W_1\,W_2}$; for any $u_{(4)}'\in 
T(W_4')$, set
\begin{align*}
 \psi_{\mathcal{Y}_1,\mathcal{Y}_2}(u_{(4)}')=\langle 
u_{(4)}',\mathcal{Y}_1(\cdot,x_1)\mathcal{Y}_2(\cdot,x_2)\cdot\rangle\in 
(T(W_1)\otimes T(W_2)\otimes T(W_3))^*\lbrace x_1,x_2\rbrace.
\end{align*}
Then $\psi_{\mathcal{Y}_1,\mathcal{Y}_2}(u_{(4)}')$ satisfies the system of 
formal differential equations
\begin{align}
 (\ell+h^\vee)\dfrac{\partial\psi}{\partial x_1} & 
=\left(\dfrac{\Omega_{1,2}}{x_1-x_2}+\dfrac{\Omega_{1,3}}{x_1}\right)\psi\label{KZprod1}\\
 (\ell+h^\vee)\dfrac{\partial\psi}{\partial x_2} & 
=\left(\dfrac{\Omega_{2,3}}{x_2}-\dfrac{\Omega_{1,2}}{x_1-x_2}\right)\psi.\label
{KZprod2}
\end{align}
\end{theo}
\begin{proof}
 To derive \eqref{KZit0}, we use the $L(-1)$-derivative property 
\eqref{intwopderiv}, \eqref{L-1}, and \eqref{aon1} to obtain
 \begin{align}\label{KZit0calc}
  (\ell+h^\vee) & \dfrac{\partial}{\partial x_0}\langle 
u_{(4)}',\mathcal{Y}^1(\mathcal{Y}^2(u_{(1)},x_0)u_{(2)},x_2)u_{(3)}\rangle 
\nonumber\\
  & =(\ell+h^\vee)\langle 
u_{(4)}',\mathcal{Y}^1(\mathcal{Y}^2(L(-1)u_{(1)},x_0)u_{(2)},x_2)u_{(3)}
\rangle\nonumber\\
  & =\sum_{i=1}^{\mathrm{dim}\,\g}\langle 
u_{(4)}',\mathcal{Y}^1(\mathcal{Y}^2(\gamma_i(-1)\gamma_i(0)u_{(1)},x_0)u_{(2)},
x_2)u_{(3)}\rangle\nonumber\\
  & =\sum_{i=1}^{\mathrm{dim}\,\g} x_0^{-1}\langle 
u_{(4)}',\mathcal{Y}^1(\mathcal{Y}^2(\gamma_i(0)u_{(1)},x_0)\gamma_i(0)u_{(2)},
x_2)u_{(3)}\rangle\nonumber\\
  & \;\;\;\;+\sum_{i=1}^{\mathrm{dim}\,\g} (x_2+x_0)^{-1}\langle 
u_{(4)}',\mathcal{Y}^1(\mathcal{Y}^2(\gamma_i(0)u_{(1)},x_0)u_{(2)},
x_2)\gamma_i(0)u_{(3)}\rangle
 \end{align}
for any $u_{(1)}\in T(W_1)$, $u_{(2)}\in T(W_2)$, $u_{(3)}\in T(W_3)$, and 
$u_{(4)}'\in T(W_4')$, and \eqref{KZit0} follows.

To derive \eqref{KZit2}, we will need the $L(-1)$-commutator formula
\begin{equation*}
 [L(-1),\mathcal{Y}(w,x)]=\mathcal{Y}(L(-1)w,x)
\end{equation*}
which follows from the Jacobi identity \eqref{intwopjac} with $v=\omega$ and 
holds for any intertwining operator $\mathcal{Y}$. Using this, the 
$L(-1)$-derivative property \eqref{intwopderiv}, and \eqref{L-1},
\begin{align}\label{KZit2calc}
 (\ell+h^\vee) & \dfrac{\partial}{\partial x_2}\langle 
u_{(4)}',\mathcal{Y}^1(\mathcal{Y}^2(u_{(1)},x_0)u_{(2)},x_2)u_{(3)}\rangle 
\nonumber\\
  & =(\ell+h^\vee)\langle 
u_{(4)}',\mathcal{Y}^1(L(-1)\mathcal{Y}^2(u_{(1)},x_0)u_{(2)},x_2)u_{(3)}
\rangle\nonumber\\
  & = (\ell+h^\vee)\langle 
u_{(4)}',\mathcal{Y}^1(\mathcal{Y}^2(u_{(1)},x_0)L(-1)u_{(2)},x_2)u_{(3)}
\rangle\nonumber\\
  & \;\;\;\;+(\ell+h^\vee)\langle 
u_{(4)}',\mathcal{Y}^1(\mathcal{Y}^2(L(-1)u_{(1)},x_0)u_{(2)},x_2)u_{(3)}
\rangle\nonumber\\
  & = \sum_{i=1}^{\mathrm{dim}\,\g}\langle 
u_{(4)}',\mathcal{Y}^1(\mathcal{Y}^2(u_{(1)},x_0)\gamma_i(-1)\gamma_i(0)u_{(2)},
x_2)u_{(3)}\rangle\nonumber\\
  & \;\;\;\;+(\ell+h^\vee)\dfrac{\partial}{\partial x_0}\langle 
u_{(4)}',\mathcal{Y}^1(\mathcal{Y}^2(u_{(1)},x_0)u_{(2)},x_2)u_{(3)}\rangle
\end{align}
for any $u_{(1)}\in T(W_1)$, $u_{(2)}\in T(W_2)$, $u_{(3)}\in T(W_3)$, and 
$u_{(4)}'\in T(W_4')$. Using \eqref{aon2}, the first term on the right side of 
\eqref{KZit2calc} becomes
\begin{align*}
 & x_2^{-1}\sum_{i=1}^{\mathrm{dim}\,\g}\langle 
u_{(4)}',\mathcal{Y}^1(\mathcal{Y}^2(u_{(1)},x_0)\gamma_i(0)u_{(2)},
x_2)\gamma_i(0)u_{(3)}\rangle\nonumber\\
 & -x_0^{-1}\sum_{i=1}^{\mathrm{dim}\,\g}\langle 
u_{(4)}',\mathcal{Y}^1(\mathcal{Y}^2(\gamma_i(0)u_{(1)},x_0)\gamma_i(0)u_{(2)},
x_2)u_{(3)}\rangle.
\end{align*}
Combining this with \eqref{KZit0calc} and \eqref{KZit2calc} yields 
\eqref{KZit2}.
 
 The derivation of \eqref{KZprod1} and \eqref{KZprod2}, which requires an 
analogue of Lemma \ref{aon1and2}, is analogous, and we omit it here. See for 
instance \cite{HL4} for details.
\end{proof}
\begin{rema}
 The equation \eqref{KZit0} is heuristically the same as \eqref{KZprod1} under 
the identification $x_0=x_1-x_2$. But one cannot substitute $x_0=x_1-x_2$ in 
\eqref{KZit0} because $(x_2+(x_1-x_2))^{-1}$ is not a well-defined formal 
series.
\end{rema}

Although we cannot substitute $x_1-x_2$ for $x_0$ in 
$\varphi_{\mathcal{Y}^1,\mathcal{Y}^2}=\langle\cdot,\mathcal{Y}^1(\mathcal{Y}
^2(\cdot,x_0)\cdot,x_2)\cdot\rangle$ or in \eqref{KZit0}-\eqref{KZit2}, we can 
substitute $x_2=x_1-x_0$. Then we have
\begin{corol}\label{KZit01}
 For any $u_{(4)}'\in T(W_4')$, define
 \begin{equation*}
\widetilde{\varphi}_{\mathcal{Y}^1,\mathcal{Y}^2}(u_{(4)}')=\varphi_{\mathcal{Y}
^1,\mathcal{Y}^2}(u_{(4)}')\vert_{x_2=x_1-x_0}=\langle 
u_{(4)}',\mathcal{Y}^1(\mathcal{Y}^2(\cdot,x_0)\cdot,x_1-x_0)\cdot\rangle.
\end{equation*}
Then $\widetilde{\varphi}_{\mathcal{Y}^1,\mathcal{Y}^2}(u_{(4)}')$ satisfies the 
system of formal differential equations
\begin{align}
 (\ell+h^\vee)\dfrac{\partial\varphi}{\partial x_0} & 
=\left(\dfrac{\Omega_{1,2}}{x_0}-\dfrac{\Omega_{2,3}}{x_1-x_0}
\right)\varphi\label{KZit0'}\\
 (\ell+h^\vee)\dfrac{\partial\varphi}{\partial x_1} & 
=\left(\dfrac{\Omega_{2,3}}{x_1-x_0}+\dfrac{\Omega_{1,3}}{x_1}
\right)\varphi.\label{KZit1}
\end{align}
\end{corol}
\begin{proof}
 This follows from \eqref{KZit0} and \eqref{KZit2} since
 \begin{equation*}
  \dfrac{\partial}{\partial 
x_0}\widetilde{\varphi}_{\mathcal{Y}^1,\mathcal{Y}^2}(u_{(4)}
')=\left.\left(\dfrac{\partial}{\partial 
x_0}\varphi_{\mathcal{Y}^1,\mathcal{Y}^2}(u_{(4)}')\right)\right\vert_{
x_2=x_1-x_0}-\left.\left(\dfrac{\partial}{\partial 
x_2}\varphi_{\mathcal{Y}^1,\mathcal{Y}^2}(u_{(4)}')\right)\right\vert_{
x_2=x_1-x_0}
 \end{equation*}
and
\begin{equation*}
 \dfrac{\partial}{\partial 
x_1}\widetilde{\varphi}_{\mathcal{Y}^1,\mathcal{Y}^2}(u_{(4)}
')=\left.\left(\dfrac{\partial}{\partial 
x_2}\varphi_{\mathcal{Y}^1,\mathcal{Y}^2}(u_{(4)}')\right)\right\vert_{
x_2=x_1-x_0}
\end{equation*}
for any $u_{(4)}'\in T(W_4')$.
\end{proof}

\subsection{Reduction to one variable}

Suppose $W_1$, $W_2$, and $W_3$ are irreducible $\imzero$-modules. Then for each 
$i=1,2,3$, there is some $h_i\in\mathbb{Q}$ which is the lowest conformal weight 
of $W_i$, so that the conformal weights of $W_i$ are contained in 
$h_i+\mathbb{N}$. In fact, we have $h_i=h_{\lambda_i,\ell}$ (recall 
\eqref{hlambdal}) for some dominant integral weight $\lambda_i$ of $\g$.

Now suppose $\mathcal{Y}$ is an intertwining operator of type 
$\binom{W_3}{W_1\,W_2}$. By Remark 5.4.4 in \cite{FHL}, we can write
\begin{equation*}
 \mathcal{Y}(w_{(1)},x)w_{(2)}=\sum_{n\in\Z} o^\mathcal{Y}_n(w_{(1)}\otimes 
w_{(2)}) x^{h-n-1}
\end{equation*}
where $h=h_3-h_1-h_2$. For each $n\in\Z$, we have
\begin{equation}\label{oyops}
 o^\mathcal{Y}_n: (W_1)_{(h_1+k)}\otimes (W_2)_{(h_2+m)}\rightarrow 
(W_3)_{h_3+k+m-n-1}
\end{equation}
for $k,m\in\mathbb{N}$. In particular, $o^\mathcal{Y}_{-1}$ maps $T(W_1)\otimes 
T(W_2)$ to $T(W_3)$. 

Since $\imzero$-modules are completely reducible, an intertwining operator among 
$\imzero$-modules is a sum of intertwining operators among irreducible 
$\imzero$-modules. Thus for any intertwining operator of type 
$\binom{W_3}{W_1\,W_2}$ where $W_1$, $W_2$, and $W_3$ are not necessarily 
irreducible, we can define operators $o^\mathcal{Y}_n$ for $n\in\Z$. In 
particular, $o^\mathcal{Y}_{-1}$ maps $T(W_1)\otimes T(W_2)$ into $T(W_3)$.

Now suppose that $W_1$, $W_2$, $W_3$, $W_4$, $M_1$, and $M_2$ are 
$\imzero$-modules, $\mathcal{Y}_1\in\mathcal{V}^{W_4}_{W_1 M_1}$, 
$\mathcal{Y}_2\in\mathcal{V}^{M_1}_{W_2 W_3}$, 
$\mathcal{Y}^1\in\mathcal{V}^{W_4}_{M_2 W_3}$, and 
$\mathcal{Y}^2\in\mathcal{V}^{M_2}_{W_1 W_2}$. Recalling the operators 
$\Omega_{1,2}$, $\Omega_{1,3}$, and $\Omega_{2,3}$ on $(T(W_1)\otimes 
T(W_2)\otimes T(W_3))^*$ from the previous subsection, we define
\begin{equation*}
 H=\Omega_{1,2}+\Omega_{1,3}+\Omega_{2,3}
\end{equation*}
and $\widetilde{H}=(\ell+h)^{-1} H$. We also recall the functionals 
$\psi_{\mathcal{Y}_1,\mathcal{Y}_2}(u_{(4)}')$ and 
$\widetilde{\varphi}_{\mathcal{Y}^1,\mathcal{Y}^2}(u_{(4)}')$ for $u_{(4)}'\in 
T(W_4')$ from the previous subsection.
\begin{propo}\label{proditshapepropo}
 For any $u_{(4)}'\in T(W_4')$, there exist series 
 \begin{equation*}
  F(x), G(x)\in (T(W_1)\otimes T(W_2)\otimes T(W_3))^*\lbrace x\rbrace
 \end{equation*}
such that
\begin{equation*}
 \psi_{\mathcal{Y}_1,\mathcal{Y}_2}(u_{(4)}')=x_1^{\widetilde{H}} 
F\left(\dfrac{x_2}{x_1}\right) 
\hspace{2em}\text{and}\hspace{2em}\widetilde{\varphi}_{\mathcal{Y}^1,\mathcal{Y}
^2}(u_{(4)}')= x_1^{\widetilde{H}} G\left(\dfrac{x_0}{x_1}\right) .
\end{equation*}
\end{propo}
\begin{proof}
 Since $\imzero$-modules are completely reducible, we may assume for simplicity 
that $W_1$, $W_2$, $W_3$, $W_4$, $M_1$, and $M_2$ are all irreducible, with 
lowest conformal weights $h_1,\ldots, h_6$, respectively. We prove the assertion 
for $\widetilde{\varphi}_{\mathcal{Y}^1,\mathcal{Y}^2}$, the proof of the 
assertion for $\psi_{\mathcal{Y}_1,\mathcal{Y}_2}$ being similar. For 
$u_{(1)}\in T(W_1)$, $u_{(2)}\in T(W_2)$, and $u_{(3)}\in T(W_3)$, we have
\begin{align*}
 \langle & \widetilde{\varphi}_{\mathcal{Y}^1,\mathcal{Y}^2}(u_{(4)}'), 
u_{(1)}\otimes u_{(2)}\otimes u_{(3)}\rangle =\langle 
u_{(4)}',\mathcal{Y}^1(\mathcal{Y}^2(u_{(1)},x_0)u_{(2)},x_1-x_0)u_{(3)}
\rangle\\
 & =\sum_{m,n\in\Z}\langle u_{(4)}', 
o^{\mathcal{Y}^1}_m(o^{\mathcal{Y}^2}_n(u_{(1)}\otimes u_{(2)})\otimes u_{(3)}) 
x_0^{h_6-h_1-h_2-n-1} (x_1-x_0)^{h_4-h_6-h_3-m-1}\\
 & =\sum_{n\geq 0} 
o^{\mathcal{Y}^1}_{n-1}(o^{\mathcal{Y}^2}_{-n-1}(u_{(1)}\otimes u_{(2)})\otimes 
u_{(3)}) x_0^{h_6-h_1-h_2+n} (x_1-x_0)^{h_4-h_6-h_3-n}\\
 & =\sum_{n\geq 0} 
o^{\mathcal{Y}^1}_{n-1}(o^{\mathcal{Y}^2}_{-n-1}(u_{(1)}\otimes u_{(2)})\otimes 
u_{(3)}) 
\left(\dfrac{x_0}{x_1}\right)^{h_6-h_1-h_2+n}\left(1-\dfrac{x_0}{x_1}\right)^{
h_4-h_6-h_3-n} x_1^{h_4-h_1-h_2-h_3}.
\end{align*}
using \eqref{oyops}. Thus, for each $n\in\mathbb{N}$, we can define maps
\begin{equation*}
 g_n: T(W_1)\otimes T(W_2)\otimes T(W_3)\rightarrow W_4
\end{equation*}
by
\begin{equation*}
 g_n(u_{(1)}\otimes u_{(2)}\otimes u_{(3)})= 
o^{\mathcal{Y}^1}_{n-1}(o^{\mathcal{Y}^2}_{-n-1}(u_{(1)}\otimes u_{(2)})\otimes 
u_{(3)})
\end{equation*}
for $u_{(1)}\in T(W_1)$, $u_{(2)}\in T(W_2)$, and $u_{(3)}\in T(W_3)$. These 
maps are $\g$-homomorphisms by the $n=0$ case of \eqref{intwopcomm3}. Now we can 
take the series $G(x)$ to be
\begin{equation*}
 G(x)=\sum_{n\geq 0} g_n^*(u_{(4)}') x^{h_6-h_1-h_2+n} (1-x)^{h_4-h_6-h_3-n},
\end{equation*}
where $*$ denotes the adjoint $\g$-module homomorphism.

To complete the proof, we must show that for any $\g$-homomorphism
\begin{equation*}
 f: T(W_1)\otimes T(W_2)\otimes T(W_3)\rightarrow T(W_4),
\end{equation*}
we have
\begin{equation*}
 \langle u_{(4)}', f(u_{(1)}\otimes u_{(2)}\otimes u_{(3)})\rangle 
x^{h_4-h_1-h_2-h_3} =\langle u_{(4)}', f(x^{\widetilde{H}^*}(u_{(1)}\otimes 
u_{(2)}\otimes u_{(3)}))\rangle.
\end{equation*}
Recall from \eqref{L(0)} that if $W$ is an $\imzero$-module, then
\begin{equation*}
 L(0)\vert_{T(W)}=\dfrac{1}{2(\ell+h^\vee)} C_{T(W)}
\end{equation*}
where $C_{T(W)}$ denotes the action on $T(W)$ of the Casimir operator associated 
to the invariant bilinear form $\langle\cdot,\cdot\rangle$ on $\g$. Thus since 
$L(0)$ is self-adjoint and $f$ is a $\g$-homomorphism, we have
\begin{align*}
 \langle u_{(4)}', & f(u_{(1)}\otimes u_{(2)}\otimes u_{(3)})\rangle 
x^{h_4-h_1-h_2-h_3} \\
 & =\langle x^{L(0)} u_{(4)}', f( x^{-L(0)}u_{(1)}\otimes 
x^{-L(0)}u_{(2)}\otimes x^{-L(0)}u_{(3)}\rangle\\
 & = \langle u_{(4)}', x^{L(0)} f( x^{-L(0)}u_{(1)}\otimes 
x^{-L(0)}u_{(2)}\otimes x^{-L(0)}u_{(3)}\rangle\\
 & =\left\langle u_{(4)}', 
x^{\frac{C_{T(W_4)}}{2(\ell+h^\vee)}}f\left(x^{-\frac{C_{T(W_1)}}{2(\ell+h^\vee)
}} u_{(1)}\otimes x^{-\frac{C_{T(W_2)}}{2(\ell+h^\vee)}} u_{(2)}\otimes 
x^{-\frac{C_{T(W_3)}}{2(\ell+h^\vee)}} u_{(3)}\right)\right\rangle \\
 & =\left\langle u_{(4)}', f\left( x^{\frac{C_{T(W_1)\otimes T(W_2)\otimes 
T(W_3)}}{2(\ell+h^\vee)}}\left(x^{-\frac{C_{T(W_1)}}{2(\ell+h^\vee)}} 
u_{(1)}\otimes x^{-\frac{C_{T(W_2)}}{2(\ell+h^\vee)}} u_{(2)}\otimes 
x^{-\frac{C_{T(W_3)}}{2(\ell+h^\vee)}} u_{(3)}\right)\right)\right\rangle. 
\end{align*}
But it is clear from the action of $\g$ on tensor products of $\g$-modules that
\begin{align*}
 C_{T(W_1)\otimes T(W_2)\otimes T(W_3)} & =C_{T(W_1)}\otimes 1_{T(W_2)}\otimes 
1_{T(W_3)}+1_{T(W_1)}\otimes C_{T(W_2)}\otimes 1_{T(W_3)}\nonumber\\
 &\;\;\;\; +1_{T(W_1)}\otimes 1_{T(W_2)}\otimes C_{T(W_3)}
+2 (\Omega_{1,2}+\Omega_{1,3}+\Omega_{2,3})^*,
\end{align*}
hence by the definition of $\widetilde{H}$,
\begin{equation*}
 x^{\widetilde{H}^*} (u_{(1)}\otimes u_{(2)}\otimes u_{(3)})= 
x^{\frac{C_{T(W_1)\otimes T(W_2)\otimes 
T(W_3)}}{2(\ell+h^\vee)}}\left(x^{-\frac{C_{T(W_1)}}{2(\ell+h^\vee)}} 
u_{(1)}\otimes x^{-\frac{C_{T(W_2)}}{2(\ell+h^\vee)}} u_{(2)}\otimes 
x^{-\frac{C_{T(W_3)}}{2(\ell+h^\vee)}} u_{(3)}\right),
\end{equation*}
for any $u_{(1)}\in T(W_1)$, $u_{(2)}\in T(W_2)$, and $u_{(3)}\in T(W_3)$, 
completing the proof.
\end{proof}

The significance of the preceding proposition is illustrated by the following:
\begin{propo}
 A series of the form $x_1^{\widetilde{H}} F(\frac{x_2}{x_1})$ where $F(x)\in 
(T(W_1)\otimes T(W_2)\otimes T(W_3))^*\lbrace x\rbrace $ solves 
\eqref{KZprod1}-\eqref{KZprod2} if and only if $F(x)$ satisfies the formal 
differential equation
 \begin{equation}\label{KZonevarprod}
  (\ell+h^\vee)\dfrac{d}{dx} F(x) 
=\left(\dfrac{\Omega_{2,3}}{x}-\dfrac{\Omega_{1,2}}{1-x}\right) F(x).
 \end{equation}
Similarly, a series of the form $x_1^{\widetilde{H}} G(\frac{x_0}{x_1})$ where 
$G(x)\in (T(W_1)\otimes T(W_2)\otimes T(W_3))^*\lbrace x\rbrace $ solves 
\eqref{KZit0'}-\eqref{KZit1} if and only if $G(x)$ solves
\begin{equation}\label{KZonevarit}
 (\ell+h^\vee)\dfrac{d}{dx}G(x) 
=\left(\dfrac{\Omega_{1,2}}{x}-\dfrac{\Omega_{2,3}}{1-x}\right) G(x).
\end{equation}
\end{propo}
\begin{proof}
 We must first verify that $\widetilde{H}$ commutes with $\Omega_{1,2}$, 
$\Omega_{1,3}$, and $\Omega_{2,3}$. It suffices to verify that
 \begin{equation*}  
[\Omega_{1,2},\Omega_{1,3}+\Omega_{2,3}]=[\Omega_{1,3},\Omega_{1,2}+\Omega_{2,3}
]=[\Omega_{2,3},\Omega_{1,2}+\Omega_{1,3}]=0.
 \end{equation*}
In fact, for any $f\in (T(W_1)\otimes T(W_2)\otimes T(W_3))^*$ and $u_{(1)}\in 
T(W_1)$, $u_{(2)}\in T(W_2)$, $u_{(3)}\in T(W_3)$, we have
\begin{align*}
 \langle [\Omega_{1,2}, & \Omega_{1,3}+\Omega_{2,3}]f, u_{(1)}\otimes 
u_{(2)}\otimes u_{(3)}\rangle =\langle f, 
[\Omega_{1,3}^*+\Omega_{2,3}^*,\Omega_{1,2}^*](u_{(1)}\otimes u_{(2)}\otimes 
u_{(3)})\rangle\\
 & =\sum_{i,j=1}^{\mathrm{dim}\,\g}\langle f, (\gamma_i \gamma_j u_{(1)}\otimes 
\gamma_j u_{(2)}+\gamma_j u_{(1)}\otimes \gamma_i \gamma_j u_{(2)})\otimes 
\gamma_i u_{(3)}\rangle\\
 &\;\;\;\; -\sum_{i,j=1}^{\mathrm{dim}\,g} \langle f, (\gamma_j \gamma_i 
u_{(1)}\otimes \gamma_j u_{(2)}+\gamma_j u_{(1)}\otimes \gamma_j \gamma_i 
u_{(2)})\otimes \gamma_i u_{(3)}\rangle\\
 & =\sum_{i,j=1}^{\mathrm{dim}\,\g} \langle f, [\gamma_i, \gamma_j] 
u_{(1)}\otimes \gamma_j u_{(2)}+\gamma_j u_{(1)}\otimes [\gamma_i, \gamma_j] 
u_{(2)})\otimes \gamma_i u_{(3)}\\
 & =\sum_{i,j,k=1}^{\mathrm{dim}\,\g} \langle f, c^k_{i,j}(\gamma_k 
u_{(1)}\otimes \gamma_j u_{(2)}\otimes \gamma_i u_{(3)})+c^k_{i,j}(\gamma_j 
u_{(1)}\otimes \gamma_k u_{(2)}\otimes \gamma_i u_{(3)})\rangle\\
 & =\sum_{i,j,k=1}^{\mathrm{dim}\,\g} \langle f, (c^k_{i,j}+c^j_{i,k})(\gamma_k 
u_{(1)}\otimes \gamma_j u_{(2)}\otimes \gamma_i u_{(3)})\rangle,
\end{align*}
where $[\gamma_i,\gamma_j]=\sum_{k=1}^{\mathrm{dim}\,\g} c^k_{i,j} \gamma_k$ for 
any $i$, $j$. Now, for any $i$, $j$, and $k$ we have $c^k_{i,j}=-c^k_{j,i}$ by 
skew symmetry and we have $c^k_{i,j}=c^i_{j,k}$ by the invariance of the form 
$\langle\cdot,\cdot\rangle$ on $\g$. This implies that $c^k_{i,j}=-c^j_{i,k}$ 
for any $i$, $j$, and $k$, which proves that 
$[\Omega_{1,2},\Omega_{1,3}+\Omega_{2,3}]=0$. The proofs that 
$[\Omega_{1,3},\Omega_{1,2}+\Omega_{2,3}]=0$ and 
$[\Omega_{2,3},\Omega_{1,2}+\Omega_{1,3}]=0$ are similar.

Now suppose $x_1^{\widetilde{H}} F(\frac{x_2}{x_1})$ satisfies 
\eqref{KZprod1}-\eqref{KZprod2}. Using \eqref{KZprod1} and the fact that 
$\widetilde{H}$ commutes with $\Omega_{1,2}$ and $\Omega_{1,3}$, we have
\begin{align*}
 (\ell+h^\vee)\dfrac{\partial}{\partial x_1} \left(x_1^{\widetilde{H}} 
F\left(\dfrac{x_2}{x_1}\right)\right) & = H x^{\widetilde{H}-1} 
F\left(\dfrac{x_2}{x_1}\right)-(\ell+h^\vee) x_1^{\widetilde{H}-2} x_2 
F'\left(\dfrac{x_2}{x_1}\right)\\
 & = 
x_1^{\widetilde{H}-1}\left(\Omega_{1,3}+\dfrac{\Omega_{1,2}}{1-x_2/x_1}\right) 
F\left(\dfrac{x_2}{x_1}\right),
\end{align*}
so that
\begin{align*}
 (\ell+h^\vee) F'\left(\dfrac{x_2}{x_1}\right) & 
=\dfrac{x_1}{x_2}\left(H-\Omega_{1,3}-\dfrac{\Omega_{1,2}}{1-x_2/x_1}\right) 
F\left(\dfrac{x_2}{x_1}\right)\\
 & =\left(\dfrac{\Omega_{2,3}}{x_2/x_1}-\dfrac{\Omega_{1,2}}{1-x_2/x_1}\right) 
F\left(\dfrac{x_2}{x_1}\right).
\end{align*}
Thus $F(x)$ satisfies \eqref{KZonevarprod}. Conversely, if $F(x)$ satisfies 
\eqref{KZonevarprod}, reversing the steps above shows that $x_1^{\widetilde{H}} 
F(\frac{x_2}{x_1})$ satisfies \eqref{KZprod1}. We need to show that 
$x_1^{\widetilde{H}} F(\frac{x_2}{x_1})$ satisfies \eqref{KZprod2}. Indeed,
\begin{align*}
 (\ell+h^\vee) & \dfrac{\partial}{\partial x_2}\left(x_1^{\widetilde{H}} 
F\left(\dfrac{x_2}{x_1}\right)\right)  
=(\ell+h^\vee)x_1^{\widetilde{H}-1}F'\left(\dfrac{x_2}{x_1}\right)\\
 & 
=x_1^{\widetilde{H}-1}\left(\dfrac{\Omega_{2,3}}{x_2/x_1}-\dfrac{\Omega_{1,2}}{
1-x_2/x_1}\right)F\left(\dfrac{x_2}{x_1}\right)=\left(\dfrac{\Omega_{2,3}}{x_2}
-\dfrac{\Omega_{1,2}}{x_1-x_2}\right) x_1^{\widetilde{H}} 
F\left(\dfrac{x_2}{x_1}\right),
\end{align*}
since $\widetilde{H}$ commutes with $\Omega_{2,3}$ and $\Omega_{1,2}$.  

The proof of the second assertion of the proposition is entirely analogous.
\end{proof}

As an immediate consequence of the preceding two propositions, we have
\begin{corol}\label{proditshape}
 For any $u_{(4)}'\in T(W_4')$, 
$x_1^{-\widetilde{H}}\psi_{\mathcal{Y}_1,\mathcal{Y}_2}(u_{(4)}')$ is a series 
$F(\frac{x_2}{x_1})$ where $F(x)$ satisfies \eqref{KZonevarprod}. Similarly, 
$x_1^{-\widetilde{H}}\widetilde{\varphi}_{\mathcal{Y}^1,\mathcal{Y}^2}(u_{(4)}
')$ is a series $G(\frac{x_0}{x_1})$ where $G(x)$ satisfies \eqref{KZonevarit}.
\end{corol}

\subsection{Drinfeld associator isomorphisms}\label{Drinassoc}

In this subsection we study solutions to \eqref{KZonevarprod} and 
\eqref{KZonevarit} in an abstract setting. We fix a finite-dimensional 
$\g$-module $W$ and two  diagonalizable $\g$-module endomorphisms $A$ and $B$ of 
$W$. Our goal is to obtain a $\g$-module automorphism of $W$ from 
solutions to the following one-variable version of the KZ equations:
\begin{equation}\label{KZonevar}
 \dfrac{d\varphi}{dx} = \left(\dfrac{A}{x}-\dfrac{B}{1-x}\right)\varphi
\end{equation}
where $\varphi\in W\lbrace x\rbrace [\mathrm{log}\,x]$. 

Suppose $\lbrace\lambda\rbrace$ is the minimal set of eigenvalues of $A$ such that all 
eigenvalues of $A$ are contained in $\cup_\lambda  (\lambda+\mathbb{N})$. Thus for 
each $\lambda$, $\lbrace\lambda+N^\lambda_j\rbrace_{j=0}^{J_\lambda}$ is the set 
of eigenvalues of $A$ in $\lambda+\Z$, where each $N^\lambda_j\in\mathbb{N}$ and
\begin{equation*}
 0=N^\lambda_0<N^\lambda_1<\ldots<N^\lambda_{J_\lambda}.
\end{equation*}
For any eigenvalue $\mu$ of $A$, we use $\pi^A_{\mu}$ to denote projection onto 
the $\mu$-eigenspace of $A$.
\begin{propo}\label{KZpropA}
 For any $w\in W$, there is a unique solution to \eqref{KZonevar} of the form
 \begin{equation}\label{wformsoln}
  \varphi^A_w(x)=\sum_\lambda\sum_{j=0}^{J_\lambda}\sum_{i\geq N_j^\lambda} 
w^{(\lambda)}_{i,j} x^{\lambda+i} (\mathrm{log}\,x)^j\in W\lbrace x\rbrace 
[\mathrm{log}\,x]
 \end{equation}
such that for each $\lambda$, $w^{(\lambda)}_{0,0}=\pi^A_\lambda(w)$ and for 
each $j>0$, 
$\pi^A_{\lambda+N^\lambda_j}(w^{(\lambda)}_{N^\lambda_j,0})=\pi^A_{
\lambda+N^\lambda_j}(w)$.
\end{propo}
\begin{proof}
 Since no two $\lambda$'s are congruent mod $\Z$, any series as in 
\eqref{wformsoln} solves \eqref{KZonevar} if and only if for each $\lambda$, the 
double series over $j$ and $i$ in \eqref{wformsoln} solves \eqref{KZonevar}. 
Thus we may fix a $\lambda$, and suppose that the $A$-eigenvalues in 
$\lambda+\Z$ consist of $\lambda, \lambda+N_1,\ldots,\lambda+N_J$. Then we need 
to show that any solution to \eqref{KZonevar} of the form
 \begin{equation}\label{onelambda}
  \sum_{j=0}^J\sum_{i\geq N_j} w_{i,j} x^{\lambda+i} (\mathrm{log}\,x)^j
 \end{equation}
exists and is uniquely determined by the vectors $w_{0,0}$ and 
$\pi^A_{\lambda+N_j}(w_{N_j,0})$ for $j>0$; moreover, we need to show that 
$w_{0,0}$ is an $A$-eigenvector with eigenvalue $\lambda$.

We observe that the series \eqref{onelambda} solves \eqref{KZonevar} if and only 
if
\begin{align}
 \sum_{j=0}^J\sum_{i\geq N_j} & \left( (\lambda+i-A)w_{i,j} 
x^{\lambda+i}(\mathrm{log}\,x)^j+jw_{i,j} x^{\lambda+i} 
(\mathrm{log}\,x)^{j-1}\right)\nonumber\\
 & =-\dfrac{x}{1-x}\sum_{j=0}^J\sum_{i\geq N_j} Bw_{i,j} 
x^{\lambda+i}(\mathrm{log}\,x)^j.
\end{align}
When we identify powers of $x$ and $\mathrm{log}\,x$ in preceding equation, we 
see that we must show there is a unique solution to the equations
\begin{equation}\label{KZcomps}
 (\lambda+i-A)w_{i,j}+(j+1)w_{i,j+1}=-\sum_{k=N_j}^{i-1} Bw_{k,j}
\end{equation}
for $0\leq j\leq J$ and $i\geq N_j$. It is enough to prove that for any $j'\leq 
J$, the equations \eqref{KZcomps} for $j\geq j'$ have a solution which is 
uniquely determined by the vectors $w_{N_{j'},j'}$ and 
$\pi^A_{\lambda+N_j}(w_{N_j,j'})$, and that moreover $w_{N_{j'},j'}$ is an 
eigenvector of $A$ with eigenvalue $\lambda+N_{j'}$. We can prove this assertion 
by downward induction on $j'$, starting with the base case $j'=J$.

Equations \eqref{KZcomps} for $j=J$ yield
\begin{equation*}
 Aw_{N_J,J}=(\lambda+N_J)w_{N_J,J}
\end{equation*}
and 
\begin{equation*}
 w_{i,J}=(A-(\lambda+i)1_W)^{-1}\sum_{k=N_J}^{i-1} Bw_{k,J}
\end{equation*}
for $i>N_J$. This shows that equations \eqref{KZcomps} for $j=J$ have a solution 
uniquely determined by $w_{N_J,J}$, which must be in the 
$(\lambda+N_J)$-eigenspace for $A$. This proves the base case of the induction.

Now suppose that for some $j'<J$, equations \eqref{KZcomps} for $j>j'$ have a 
solution uniquely determined by the vectors $w_{N_{j'+1},j'+1}$ and 
$\pi^A_{\lambda+N_j}(w_{N_j,j'+1})$ for $j>j'+1$, where $w_{N_{j'+1},j'+1}$ must 
be in the $(\lambda+N_{j'+1})$-eigenspace of $A$. We must show that the 
coefficients $w_{i,j'}$ for $i\geq N_{j'}$ as well as the vectors 
$w_{N_{j'+1},j'+1}$ and $\pi^A_{\lambda+N_j}(w_{N_j,j'+1})$ for $j>j'+1$ are 
uniquely determined by equations \eqref{KZcomps} for $j=j'$ together with the 
vectors $w_{N_{j'},j'}$ and $\pi^A_{\lambda+N_j}(w_{N_j,j'})$ for $j>j'$. We 
must also show that $w_{N_{j'},j'}$ must be in the $(\lambda+N_{j'})$-eigenspace 
for $A$.

In fact, when $i\neq N_j$ for $j\geq j'$, equations \eqref{KZcomps} for $j=j'$ 
yield
\begin{equation}\label{noteigen}
 w_{i,j'}=(A-(\lambda+i)1_W)^{-1}\left((j'+1)w_{i,j'+1}+\sum_{k=N_{j'}}^{i-1} B 
w_{k,j'}\right).
\end{equation}
On the other hand, when $i=N_{j'}$, we have
\begin{equation}\label{loweigen}
 Aw_{N_{j'},j'}=(\lambda+N_{j'})w_{N_{j'},j'},
\end{equation}
and when $i=N_j$ for $j>j'$, projecting \eqref{KZcomps} to the eigenspaces of 
$A$ yields
\begin{equation}\label{mueigen} 
\pi^A_\mu(w_{N_j,j'})=\dfrac{1}{\mu-\lambda-N_j}\left((j'+1)\pi^A_{\mu}(w_{N_j,
j'+1})+\sum_{k=N_{j'}}^{N_j-1}\pi^A_\mu(Bw_{k,j'})\right)
\end{equation}
when $\mu\neq\lambda+N_j$ and 
\begin{equation}\label{notmueigen}
\pi^A_{\lambda+N_j}(w_{N_j,j'+1})=-\dfrac{1}{j'+1}\sum_{k=N_{j'}}^{N_j-1}\pi^A_{
\lambda+N_j}(Bw_{k,j'}).
\end{equation}
Equations \eqref{noteigen}, \eqref{loweigen}, \eqref{mueigen}, and 
\eqref{notmueigen}, together with the induction hypothesis, imply that equations 
\eqref{KZcomps} for $j\geq j'$ have a unique solution determined by the vectors 
$w_{N_{j'},j'}$ and $\pi^A_{\lambda+N_j}(w_{N_j,j'})$ for $j>j'$ whenever 
$w_{N_{j'},j'}$ is in the $(\lambda+N_{j'})$-eigenspace of $A$. This completes 
the proof.
\end{proof}

The space of series $W\lbrace x\rbrace [\mathrm{log}\,x]$ has an obvious 
$\g$-module structure, and because $A$ and $B$ are $\g$-module endomorphisms, 
the subspace $S^{(x)}_{KZ}$ of $W\lbrace x\rbrace [\mathrm{log}\,x]$ consisting 
of solutions to \eqref{KZonevar} is also a $\g$-module. Then the linearity of 
\eqref{KZonevar} and the fact that $A$ and $B$ are $\g$-module endomorphisms 
implies that the correspondence
\begin{equation*}
 \varphi_A^{(x)}: w\mapsto\varphi^A_w(x)
\end{equation*}
defines a $\g$-module homomorphism from $W$ to $S^{(x)}_{KZ}$. Our goal, 
however, is a $\g$-homomorphism from $W$ to a space of $W$-valued analytic 
functions rather than a space of formal series.

We give $W$ the weak topology whereby a sequence $\lbrace w_n\rbrace\subseteq W$ 
converges to an element $w\in W$ if and only if for any $w'\in W^*$, $\langle 
w',w_n\rangle\rightarrow\langle w',w\rangle$ in the usual topology on $\C$. Note 
that the action of $\g$ on $W$ is continuous because if $\lbrace w_n\rbrace$ is 
a sequence converging to $w\in W$, then for any $g\in\g$ and $w'\in W^*$,
\begin{equation*}
 \langle w', g\cdot w_n\rangle=-\langle g\cdot w',w_n\rangle\rightarrow-\langle 
g\cdot w',w\rangle=\langle w',g\cdot w\rangle,
\end{equation*}
so that $\lbrace g\cdot w_n\rbrace$ converges to $g\cdot w$.

The topology on $W$ allows us to speak of analytic functions from $\C$ to $W$. 
We use $S_{KZ}$ to denote the space of analytic $W$-valued solutions of 
\begin{equation}\label{KZanaly}
 \dfrac{d}{dz}\varphi(z)=\left(\dfrac{A}{z}-\dfrac{B}{1-z}\right)\varphi(z)
\end{equation}
on $\C\setminus\left((-\infty,0]\cup [1,\infty)\right)$. Since $A$ and $B$ are 
$\g$-module homomorphisms, $S_{KZ}$ is a $\g$-module with the action of $\g$ 
defined by
\begin{equation*}
 (x\cdot\varphi)(z)=x\cdot\varphi(z)
\end{equation*}
for $x\in\g$, $\varphi\in S_{KZ}$, and $z\in\C\setminus\left((-\infty,0]\cup 
[1,\infty)\right)$.

Since \eqref{KZanaly} is a linear differential equation with regular singular 
points at $0$, $1$, and $\infty$, any formal solution
\begin{equation*}
 \varphi(x)=\sum_{j=0}^J\sum_{n\in\C} w_{j,n} x^n (\mathrm{log}\,x)^j\in 
S_{KZ}^{(x)}
\end{equation*}
induces a solution $\varphi(z)\in S_{KZ}$ which is the analytic extension of the 
(convergent) series
\begin{equation*}
 \sum_{j=0}^J\sum_{n\in\C} w_{j,n} z^n (\mathrm{log}\,z)^j
\end{equation*}
to $\C\setminus\left((-\infty,0]\cup[1,\infty)\right)$ (see for example Appendix 
B of \cite{Kn} for an overview of the theory of linear differential equations 
with regular singular points). Here
the complex numbers $z^n$ and $\mathrm{log}\,z$ are determined using some fixed 
branch of logarithm with a branch cut along the negative real axis. Because the 
action of $\g$ on $W$ is continuous, the correspondence 
$\varphi(x)\mapsto\varphi(z)$ defines a $\g$-module homomorphism from 
$S^{(x)}_{KZ}$ to $S_{KZ}$. Then we get a $\g$-homomorphism
\begin{equation*}
 \varphi_A: w\mapsto\varphi^A_w(z)
\end{equation*}
from $W$ to $S_{KZ}$.
\begin{propo}
 The $\g$-homomorphism $\varphi_A$ is an  isomorphism.
\end{propo}
\begin{proof}
 The theory of linear differential equations implies that the dimension of 
$S_{KZ}$ equals the dimension of $W$ so it suffices to show that $\varphi_A$ is 
injective. Thus suppose
 \begin{equation*}
  \varphi^A_w(z)=  \sum_\lambda\sum_{j=0}^{J_\lambda}\sum_{i\geq N_j^\lambda} 
w^{(\lambda)}_{i,j} z^{\lambda+i} (\mathrm{log}\,z)^j=0
 \end{equation*}
for some $w\in W$. Then the fact that 
\begin{equation*}
 \left(\cup_\lambda (\lambda+\mathbb{N})\right)\times\lbrace 
0,1,\ldots,\mathrm{max}_\lambda\,J_\lambda\rbrace
\end{equation*}
 is a unique expansion set (see Definition 7.5 and Proposition 7.8 in 
\cite{HLZ5}) implies that $w^{(\lambda)}_{i,j}=0$ for all $\lambda, i, j$. In 
particular, for each $\lambda$, $w^{(\lambda)}_{0,0}=\pi^A_\lambda(w)=0$ and for 
each $j>0$, 
$\pi^A_{\lambda+N^\lambda_j}(w^{(\lambda)}_{N^\lambda_j,0})=\pi^A_{
\lambda+N^\lambda_j}(w)=0$. Since $\pi^A_\mu(w)=0$ for all eigenvalues $\mu$ of 
$A$, $w=0$, proving injectivity.
\end{proof}

Now we consider the operator $B$ and its eigenvalues. Suppose 
$\lbrace\mu\rbrace$ is the minimal set of eigenvalues of $B$ such that all 
eigenvalues of $B$ are contained in $\cup_\mu (\mu+\mathbb{N})$. Thus for each 
$\mu$, $\lbrace\mu+M^\mu_k\rbrace_{k=0}^{K_\mu}$ is the set of eigenvalues of 
$B$ in $\mu+\Z$, where each $M^\mu_k\in\mathbb{N}$ and
\begin{equation*}
 0=M^\mu_0<M^\mu_1<\ldots<M^\mu_{K_\mu}.
\end{equation*}
For any eigenvalue $\lambda$ of $B$, we use $\pi^B_{\lambda}$ to denote 
projection onto the $\lambda$-eigenspace of $B$. The proof of Proposition 
\ref{KZpropA} yields
\begin{propo}\label{KZpropB}
 For any $w\in W$, there is a unique solution to
 \begin{equation*}
  \dfrac{d}{dy}\varphi(y)=\left(\dfrac{B}{y}-\dfrac{A}{1-y}\right)\varphi(y)
 \end{equation*}
of the form
 \begin{equation*}
  \varphi^B_w(y)=\sum_\mu\sum_{k=0}^{K_\mu}\sum_{i\geq M_k^\mu} w^{(\mu)}_{i,k} 
y^{\mu+i} (\mathrm{log}\,y)^k\in W\lbrace y\rbrace [\mathrm{log}\,y]
 \end{equation*}
such that for each $\mu$, $w^{(\mu)}_{0,0}=\pi^B_\mu(w)$ and for each $k>0$, 
$\pi^B_{\mu+M^\lambda_k}(w^{(\mu)}_{M^\mu_k,0})=\pi^B_{\mu+M^\mu_k}(w)$.
\end{propo}
Then the same argument as for the operator $A$ shows that there is a $\g$-module 
isomorphism
\begin{equation*}
 \widetilde{\varphi}_B: w\mapsto\varphi^B_w(z)
\end{equation*}
from $W$ to the space $\widetilde{S}_{KZ}$ of $W$-valued analytic solutions of
\begin{equation}\label{KZanalyB}
 \dfrac{d}{dz}\varphi(z)=\left(\dfrac{B}{z}-\dfrac{A}{1-z}\right)\varphi(z)
\end{equation}
on $\C\setminus\left((-\infty,0]\cup [1,\infty)\right)$. Observe that replacing 
$z$ with $1-z$ in \eqref{KZanalyB} gives \eqref{KZanaly}. This means that 
$\varphi(z)\in\widetilde{S}_{KZ}$ if and only if $\varphi(1-z)\in S_{KZ}$, and 
the correspondence $\varphi(z)\mapsto\varphi(1-z)$ is a $\g$-module isomorphism 
from $\widetilde{S}_{KZ}$ to $S_{KZ}$.

As a consequence of the preceding considerations, we have a $\g$-module 
isomorphism
\begin{equation*}
 \varphi_B: w\mapsto\varphi^B_w(1-z)
\end{equation*}
from $W$ to $S_{KZ}$. This allows us to define the $\g$-module automorphism
\begin{equation*}
 \Phi_{KZ}=\varphi_B^{-1}\circ\varphi_A
\end{equation*}
of $W$, which we call the \textit{Drinfeld associator} for the $\g$-module $W$ 
equipped with the diagonalizable $\g$-endomorphisms $A$ and $B$. Observe that 
$\Phi_{KZ}$ is defined by the relation
\begin{equation}\label{DAdefprop}
 \varphi_B\circ\Phi_{KZ}=\varphi_A.
\end{equation}

\section{The associativity isomorphisms in $\mathbf{D}(\g,\ell)$}

In this section we use the results and constructions of the previous section to 
describe the associativity isomorphisms in the category $\mathbf{D}(\g,\ell)$.

\subsection{Associativity of intertwining operators}\label{sub:mainassoctheo}

We recall the convergence and associativity of intertwining operators in 
$\imzero-\mathbf{mod}$, which was proved in Theorem 3.8 in \cite{HL4} using the 
KZ equations:
\begin{theo}\label{intwopassoc} Suppose $W_1$, $W_2$, $W_3$, and $W_4$ are 
$\imzero$-modules.
 \begin{enumerate}
  \item For any $\imzero$-module $M_1$ and intertwining operators 
$\mathcal{Y}_1$ of type $\binom{W_4}{W_1\,M_1}$ and $\mathcal{Y}_2$ of type 
$\binom{M_1}{W_2\,W_3}$, the double series
  \begin{equation*}
   \langle 
w_{(4)}',\mathcal{Y}_1(w_{(1)},z_1)\mathcal{Y}_2(w_{(2)},z_2)w_{(3)}\rangle
  \end{equation*}
converges absolutely for any $w_{(1)}\in W_1$, $w_{(2)}\in W_2$, $w_{(3)}\in 
W_3$, $w_{(4)}'\in W_4'$ and $\vert z_1\vert>\vert z_2\vert >0$.

\item For any $\imzero$-module $M_2$ and intertwining operators $\mathcal{Y}^1$ 
of type $\binom{W_4}{M_2\,W_3}$ and $\mathcal{Y}^2$ of type 
$\binom{M_2}{W_1\,W_2}$, the double series
  \begin{equation*}
   \langle 
w_{(4)}',\mathcal{Y}^1(\mathcal{Y}^2(w_{(1)},z_0)w_{(2)},z_2)w_{(3)}\rangle
  \end{equation*}
converges absolutely for any $w_{(1)}\in W_1$, $w_{(2)}\in W_2$, $w_{(3)}\in 
W_3$, $w_{(4)}'\in W_4'$ and $\vert z_2\vert>\vert z_0\vert >0$.

\item For any $\imzero$-module $M_1$ and intertwining operators $\mathcal{Y}_1$ 
of type $\binom{W_4}{W_1\,M_1}$ and $\mathcal{Y}_2$ of type 
$\binom{M_1}{W_2\,W_3}$, there is an $\imzero$-module $M_2$ and intertwining 
operators $\mathcal{Y}^1$ of type $\binom{W_4}{M_2\,W_3}$ and $\mathcal{Y}^2$ of 
type $\binom{M_2}{W_1\,W_2}$ such that
\begin{equation*}
 \langle 
w_{(4)}',\mathcal{Y}_1(w_{(1)},z_1)\mathcal{Y}_2(w_{(2)},z_2)w_{(3)}\rangle 
=\langle 
w_{(4)}',\mathcal{Y}^1(\mathcal{Y}^2(w_{(1)},z_1-z_2)w_{(2)},z_2)w_{(3)}\rangle
\end{equation*}
for any $w_{(1)}\in W_1$, $w_{(2)}\in W_2$, $w_{(3)}\in W_3$, $w_{(4)}'\in W_4'$ 
and $\vert z_1\vert>\vert z_2\vert>\vert z_1-z_2\vert>0$.

\item For any $\imzero$-module $M_2$ and intertwining operators $\mathcal{Y}^1$ 
of type $\binom{W_4}{M_2\,W_3}$ and $\mathcal{Y}^2$ of type 
$\binom{M_2}{W_1\,W_2}$, there is an $\imzero$-module $M_1$ and intertwining 
operators $\mathcal{Y}_1$ of type $\binom{W_4}{W_1\,M_1}$ and $\mathcal{Y}_2$ of 
type $\binom{M_1}{W_2\,W_3}$ such that
\begin{equation*}
 \langle 
w_{(4)}',\mathcal{Y}^1(\mathcal{Y}^2(w_{(1)},z_0)w_{(2)},z_2)w_{(3)}
\rangle=\langle 
w_{(4)}',\mathcal{Y}_1(w_{(1)},z_0+z_2)\mathcal{Y}_2(w_{(2)},z_2)w_{(3)}\rangle
\end{equation*}
for any $w_{(1)}\in W_1$, $w_{(2)}\in W_2$, $w_{(3)}\in W_3$, $w_{(4)}'\in W_4'$ 
and $\vert z_0+z_2\vert>\vert z_2\vert>\vert z_0\vert>0$.
 \end{enumerate}
\end{theo}

\begin{rema}\label{assocunique}
 It is easy to see from the universal property of $P(z_2)$-tensor products that 
the module $M_1$ in part 4 of Theorem \ref{intwopassoc} can be taken to be 
$W_2\boxtimes_{P(z_2)} W_3$ and that $\mathcal{Y}_2$ can be taken to be the 
corresponding tensor product intertwining operator. Then the intertwining 
operator $\mathcal{Y}_1$ is uniquely determined by these choices (see Corollary 
9.30 in \cite{HLZ6}). An analogous remark holds for the module $M_2$ and 
intertwining operators $\mathcal{Y}^1$ and $\mathcal{Y}^2$ in part 3 of the 
theorem. 
\end{rema}

We now suppose $\mathcal{Y}_1\in\mathcal{V}^{W_4}_{W_1\,M_1}$, 
$\mathcal{Y}_2\in\mathcal{V}^{M_1}_{W_2\,W_3}$, 
$\mathcal{Y}^1\in\mathcal{V}^{W_4}_{M_2\,W_3}$, and 
$\mathcal{Y}^2\in\mathcal{V}^{M_2}_{W_1\,W_2}$ satisfy the associativity 
relations of Theorem \ref{intwopassoc}. For $u_{(4)}'\in T(W_4')$, we recall 
from Subsection \ref{KZeqns} the series
\begin{equation*}
 \psi_{\mathcal{Y}_1,\mathcal{Y}_2}(u_{(4)}')\in (T(W_1)\otimes T(W_2)\otimes 
T(W_3))^*\lbrace x_1, x_2\rbrace
\end{equation*}
and
\begin{equation*}
 \widetilde{\varphi}_{\mathcal{Y}^1,\mathcal{Y}^2}(u_{(4)}')\in (T(W_1)\otimes 
T(W_2)\otimes T(W_3))^*\lbrace x_0, x_1\rbrace.
\end{equation*}
By Corollary \ref{proditshape}, we know
\begin{equation*}
 \psi_{\mathcal{Y}_1,\mathcal{Y}_2}(u_{(4)}')=x_1^{\widetilde{H}} 
F\left(\dfrac{x_2}{x_1}\right)\hspace{2em}\mathrm{and}\hspace{2em}\widetilde{
\varphi}_{\mathcal{Y}^1,\mathcal{Y}^2}(u_{(4)}')=x_1^{\widetilde{H}} 
G\left(\dfrac{x_0}{x_1}\right)
\end{equation*}
where $F(x)$ satisfies \eqref{KZonevarprod} and $G(x)$ satisfies 
\eqref{KZonevarit}.

We now take the $\g$-module $W$ of Subsection \ref{Drinassoc} to be 
$(T(W_1)\otimes T(W_2)\otimes T(W_3))^*$ and we take the $\g$-module 
endomorphisms in \eqref{KZonevar} to be 
$A=\widetilde{\Omega}_{1,2}=(\ell+h^\vee)^{-1}\Omega_{1,2}$ and 
$B=\widetilde{\Omega}_{2,3}=(\ell+h^\vee)^{-1}\Omega_{2,3}$. It is easy to see that 
$\widetilde{\Omega}_{1,2}$ and $\widetilde{\Omega}_{2,3}$ are diagonalizable. 
For example,
\begin{equation}\label{omega12tilde}
 \widetilde{\Omega}_{1,2}^*=\dfrac{1}{2(\ell+h^\vee)}(C_{T(W_1)\otimes 
T(W_2)}-C_{T(W_1)}\otimes 1_{T(W_2)}-1_{T(W_1)}\otimes C_{T(W_2)})\otimes 
1_{T(W_3)},
\end{equation}
where as earlier $C_U$ for a $\g$-module $U$ is the Casimir operator on $U$. 
Since $C_{T(W_1)\otimes T(W_2)}$, $C_{T(W_1)}\otimes 1_{T(W_2)}$, and 
$1_{T(W_1)}\otimes C_{T(W_2)}$ are all diagonalizable and commute, they are 
simultaneously diagonalizable. Hence $\widetilde{\Omega}_{1,2}^*$ and 
$\widetilde{\Omega}_{1,2}$ are diagonalizable. Similarly, 
$\widetilde{\Omega}_{2,3}$ is diagonalizable. Thus we have the $\g$-module 
automorphism 
$\Phi_{KZ}=\varphi_{\widetilde{\Omega}_{2,3}}^{-1}\circ\varphi_{\widetilde{
\Omega}_{1,2}}$ of $(T(W_1)\otimes T(W_2)\otimes T(W_3))^*$.

For any $u_{(4)}'\in T(W_4')$, we define linear functionals
\begin{equation*}
 F_{Pr}(u_{(4)}'), F_{It}(u_{(4)}')\in (T(W_1)\otimes T(W_2)\otimes T(W_3))^*
\end{equation*}
by
\begin{equation*}
 \langle F_{Pr}(u_{(4)}'),u_{(1)}\otimes u_{(2)}\otimes u_{(3)}\rangle =\langle 
u_{(4)}', o^{\mathcal{Y}_1}_{-1}(u_{(1)}\otimes 
o^{\mathcal{Y}_2}_{-1}(u_{(2)}\otimes u_{(3)}))\rangle
\end{equation*}
and
\begin{equation*}
 \langle F_{It}(u_{(4)}'),u_{(1)}\otimes u_{(2)}\otimes u_{(3)}\rangle =\langle 
u_{(4)}', o^{\mathcal{Y}^1}_{-1}(o^{\mathcal{Y}^2}_{-1}(u_{(1)}\otimes 
u_{(2)})\otimes u_{(3)})\rangle
\end{equation*}
for $u_{(1)}\in T(W_1)$, $u_{(2)}\in T(W_2)$, and $u_{(3)}\in T(W_3)$. The main 
result of this subsection is:
\begin{theo}\label{mainassoctheo}
 For any $u_{(4)}'\in T(W_4')$, $\Phi_{KZ}(F_{It}(u_{(4)}'))=F_{Pr}(u_{(4)}')$.
\end{theo}
\begin{proof}
 The first step in the proof is to show that for any $u_{(4)}'\in T(W_4')$,
\begin{equation}\label{FcomesfromFpr}
 F(x)=\varphi^{\widetilde{\Omega}_{2,3}}_{F_{Pr}(u_{(4)}')}(x),
\end{equation}
using the notation of Proposition \ref{KZpropB}, and
\begin{equation}\label{GcomesfromFit}
 G(x)=\varphi^{\widetilde{\Omega}_{1,2}}_{F_{It}(u_{(4)}')}(x),
\end{equation}
using the notation of Proposition \ref{KZpropA}. Since $\imzero$-modules are 
completely reducible, we may without loss of generality assume that 
the $\imzero$-modules $W_1$, $W_2$, $W_3$, $W_4$, $M_1$, and $M_2$ are irreducible 
with lowest conformal weights $h_1,\ldots, h_6$, respectively. We prove 
\eqref{GcomesfromFit}, the proof of \eqref{FcomesfromFpr} being similar and 
indeed slightly simpler.

Since $G(x)$ satisfies \eqref{KZonevarit}, certainly 
$G(x)=\varphi^{\widetilde{\Omega}_{1,2}}_{w_{It}}(x)$ for some $w_{It}\in 
(T(W_1)\otimes T(W_2)\otimes T(W_3))^*$. We recall from the proof of Proposition 
\ref{proditshapepropo} that
\begin{align}\label{gx}
 G(x) & =\sum_{n\geq 0} g^*_n(u_{(4)}') x^{h_6-h_1-h_2} 
(1-x)^{h_4-h_6-h_3-n}\nonumber\\
 & =\sum_{n\geq 0}\sum_{m\geq 0} (-1)^m\binom{h_4-h_6-h_3-n}{m} g_n^*(u_{(4)}') 
x^{h_6-h_1-h_2+m+n}\nonumber\\
 & =\sum_{n\geq 0}\left(\sum_{m=0}^n (-1)^m\binom{h_4-h_6-h_3-n+m}{m} 
g_{n-m}^*(u_{(4)}')\right) x^{h_6-h_1-h_2+n}
\end{align}
where $g_n^*$ is the adjoint of
\begin{equation*} 
g_n=o^{\mathcal{Y}^1}_{-n-1}(o^{\mathcal{Y}^2}_{n-1}
(\cdot\otimes\cdot)\otimes\cdot).
\end{equation*}

Suppose we use $\lbrace h_6-h_1-h_2+N_j\rbrace_{j=0}^J$, where $N_0=0$ and each 
$N_j\in\mathbb{N}$, to denote the eigenvalues of $\widetilde{\Omega}_{1,2}$ in 
$h_6-h_1-h_2+\mathbb{N}$. Then we see from Proposition \ref{KZpropA} and 
\eqref{gx} that
\begin{equation*}
 w_{It}=\sum_{j=0}^J\sum_{m=0}^{N_j} 
(-1)^m\binom{h_4-h_6-h_3-N_j+m}{m}\pi^{\widetilde{\Omega}_{1,2}}_{
h_6-h_1-h_2+N_j}(g^*_{N_j-m}(u_{(4)}')).
\end{equation*}
Now, $F_{It}(u_{(4)}')=g^*_0(u_{(4)}')$ is an eigenvector of 
$\widetilde{\Omega}_{1,2}$ with eigenvalue $h_6-h_1-h_2$. Indeed, for 
$u_{(1)}\in T(W_1)$, $u_{(2)}\in T(W_2)$ and $u_{(3)}\in T(W_3)$, using 
\eqref{omega12tilde}, \eqref{L(0)}, and the fact that Casimir operators commute 
with the $\g$-module homomorphism $o^{\mathcal{Y}^2}_{-1}$, we have
\begin{align*}
 \langle\widetilde{\Omega}_{1,2}(F_{It}( & u_{(4)}')),  u_{(1)}\otimes 
u_{(2)}\otimes u_{(3)}\rangle\nonumber\\
 &=\dfrac{1}{2(\ell+h^\vee)}\langle 
u_{(4)}',o^{\mathcal{Y}^1}_{-1}(o^{\mathcal{Y}^2}_{-1}(C_{T(W_1)\otimes 
T(W_2)}(u_{(1)}\otimes u_{(2)}))\otimes u_{(3)})\rangle\nonumber\\
 &\;\;\;\;\;-\dfrac{1}{2(\ell+h^\vee)}\langle 
u_{(4)}',o^{\mathcal{Y}^1}_{-1}(o^{\mathcal{Y}^2}_{-1}(C_{T(W_1)}(u_{(1)}
)\otimes u_{(2)}+u_{(1)}\otimes C_{T(W_2)}(u_{(2)}))\otimes 
u_{(3)})\rangle\nonumber\\
 & =\dfrac{1}{2(\ell+h^\vee)}\langle u_{(4)}',o^{\mathcal{Y}^1}_{-1}(C_{T(M_2)} 
o^{\mathcal{Y}^2}_{-1}(u_{(1)}\otimes u_{(2)})\otimes u_{(3)})\rangle\nonumber\\
 &\;\;\;\;\;-(h_1+h_2)\langle F_{It}(u_{(4)}'),u_1\otimes u_2\otimes 
u_3\rangle\nonumber\\
 & =(h_6-h_1-h_2)\langle F_{It}(u_{(4)}'),u_1\otimes u_2\otimes u_3\rangle.
\end{align*}
Thus we have
\begin{equation*}
 w_{It}-F_{It}(u_{(4)}')=\sum_{j=1}^J\sum_{m=0}^{N_j} 
(-1)^m\binom{h_4-h_6-h_3-N_j+m}{m}\pi^{\widetilde{\Omega}_{1,2}}_{
h_6-h_1-h_2+N_j}(g^*_{N_j-m}(u_{(4)}')).
\end{equation*}
Consequently, it is enough to show that 
$\pi^{\widetilde{\Omega}_{1,2}}_{h_6-h_1-h_2+N_j}(g^*_{N_j-m}(u_{(4)}'))=0$ for 
$j\geq 1$ and $0\leq m\leq N_j$.

Let us use $\pi^{C_{T(W_1)\otimes T(W_2)}}_{h}$ for an eigenvalue $h$ of 
$C_{T(W_1)\otimes T(W_2)}$ to denote projection from $T(W_1)\otimes T(W_2)$ to 
the $h$-eigenspace of $C_{T(W_1)\otimes T(W_2)}$. Then
\begin{align*}
 \langle\pi^{\widetilde{\Omega}_{1,2}}_{h_6-h_1-h_2+N_j} & 
(g^*_{N_j-m}(u_{(4)}')), u_{(1)}\otimes u_{(2)}\otimes u_{(3)}\rangle\nonumber\\
 & =\langle 
u_{(4)}',o^{\mathcal{Y}^1}_{-N_j+m-1}(o^{\mathcal{Y}^2}_{N_j-m-1}(\pi^{C_{
T(W_1)\otimes T(W_2)}}_{2(\ell+h^\vee)(h_6+N_j)}(u_{(1)}\otimes u_{(2)}))\otimes 
u_{(3)})\rangle
\end{align*}
Note that the image of $\pi^{C_{T(W_1)\otimes 
T(W_2)}}_{2(\ell+h^\vee)(h_6+N_j)}$ is the sum of all $\g$-submodules of 
$T(W_1)\otimes T(W_2)$ which are isomorphic to some $L_{\lambda'}$ where 
$\lambda'$ is a dominant integral weight of $\g$ that satisfies 
$h_{\lambda',\ell}=h_6+N_j$. Note also that the $n=0$ case of 
\eqref{intwopcomm3} implies that $o^{\mathcal{Y}^2}_{N_j-m-1}$ is a 
$\g$-homomorphism, which maps $T(W_1)\otimes T(W_2)$ into $(M_2)_{(h_6+N_j-m)}$. 
Thus to show that 
$\pi^{\widetilde{\Omega}_{1,2}}_{h_6-h_1-h_2+N_j}(g^*_{N_j-m}(u_{(4)}'))=0$ for 
$j\geq 1$ and $0\leq m\leq N_j$, it is enough to show that $(M_2)_{(h_6+N_j-m)}$ 
cannot contain a $\g$-submodule with highest weight $\lambda'$ such that 
$h_{\lambda',\ell}=h_6+N_j$.

Suppose that $T(M_2)$ is an irreducible $\g$-module with highest weight 
$\lambda$, so that $\langle\lambda,\theta\rangle\leq\ell$. Then 
$h_6=h_{\lambda,\ell}$; also, recalling from Subsection \ref{sub:gtilde} that a 
$\ghat$-module becomes a $\widetilde{\g}$-module on which the derivation 
$\mathbf{d}$ acts as $-L(0)$, we see that $M_2$ as a $\widetilde{\g}$-module is 
isomorphic to $L(\Lambda)$ where
\begin{equation*}
 \Lambda=\lambda+\ell\mathbf{k}'-h_{\lambda,\ell}\mathbf{d}'.
\end{equation*}
We need to show that for a dominant integral weight $\lambda'$ of $\g$ such that 
$h_{\lambda',\ell}=h_{\lambda,\ell}+N_j$ for $j\geq 1$, and for $0\leq m\leq 
N_j$, $L(\Lambda)$ cannot have
\begin{equation*} 
\Lambda'=\lambda'+\ell\mathbf{k}'-(h_{\lambda,\ell}+N_j-m)\mathbf{d}
'=\lambda'+\ell\mathbf{k}'-(h_{\lambda',\ell}-m)\mathbf{d}'
\end{equation*}
as a weight. But this follows immediately from Theorem \ref{weightstheo}, 
completing the proof that $w_{It}=F_{It}(u_{(4)}')$. 

Now that we know $F(x)=\varphi^{\widetilde{\Omega}_{2,3}}_{F_{Pr}(u_{(4)}')}(x)$ 
and $G(x)=\varphi^{\widetilde{\Omega}_{1,2}}_{F_{It}(u_{(4)}')}(x)$, we can 
prove that $\Phi_{KZ}(F_{It}(u_{(4)}'))=F_{Pr}(u_{(4)}')$ for $u_{(4)}'\in 
T(W_4)^*$. For this, it is enough to show that
\begin{equation*} 
\varphi^{\widetilde{\Omega}_{2,3}}_{\Phi_{KZ}(F_{It}(u_{(4)}'))}(z)=\varphi^{
\widetilde{\Omega}_{2,3}}_{F_{Pr}(u_{(4)}')}(z)
\end{equation*}
for all $z$ contained in some non-empty open set $U$ of 
$\C\setminus((-\infty,0]\cup [1,\infty))$. In fact, we take $U$ to be the set of 
$z\in\C$ such that $\vert z\vert<1$ and $\mathrm{Re}\,z>\frac{1}{2}$. 
Equivalently, $U$ is the set of $z\in\C$ which satisfy
\begin{equation*}
 1>\vert z\vert >\vert 1-z\vert >0.
\end{equation*}
We recall from Subsection \ref{Drinassoc} the isomorphisms
\begin{equation*}
 \varphi_{\widetilde{\Omega}_{1,2}}: w\mapsto 
\varphi^{\widetilde{\Omega}_{1,2}}_w(z)
\end{equation*}
and
\begin{equation*}
 \varphi_{\widetilde{\Omega}_{2,3}}: 
w\mapsto\varphi^{\widetilde{\Omega}_{2,3}}_w(1-z)
\end{equation*}
from $(T(W_1)\otimes T(W_2)\otimes T(W_3))^*$ to $S_{KZ}$.

Now, using \eqref{DAdefprop}, Corollary \ref{proditshape}, \eqref{intwopassoc}, 
\eqref{FcomesfromFpr}, and \eqref{GcomesfromFit}, we see that for any $z\in U$,
\begin{align*}
 \langle 
1^{\widetilde{H}}\varphi^{\widetilde{\Omega}_{2,3}}_{\Phi_{KZ}(F_{It}(u_{(4)}'))
} & (z), u_{(1)}\otimes u_{(2)}\otimes u_{(3)}\rangle\nonumber\\
 & =\langle 
1^{\widetilde{H}}((\varphi_{\widetilde{\Omega}_{2,3}}\circ\Phi_{KZ})(F_{It}(u_{
(4)}')))(1-z), u_{(1)}\otimes u_{(2)}\otimes u_{(3)}\rangle\nonumber\\
 & =\langle 
1^{\widetilde{H}}(\varphi_{\widetilde{\Omega}_{1,2}}(F_{It}(u_{(4)}')))(1-z), 
u_{(1)}\otimes u_{(2)}\otimes u_{(3)}\rangle\nonumber\\
 & =\langle 
u_{(4)}',\mathcal{Y}^1(\mathcal{Y}^2(u_{(1)},1-z)u_{(2)},z)u_{(3)}
\rangle\nonumber\\
 & =\langle u_{(4)}', 
\mathcal{Y}_1(u_{(1)},1)\mathcal{Y}_2(u_{(2)},z)u_{(3)}\rangle\nonumber\\
 & =\langle 
1^{\widetilde{H}}\varphi^{\widetilde{\Omega}_{2,3}}_{F_{Pr}(u_{(4)}')}(z),u_{(1)
}\otimes u_{(2)}\otimes u_{(3)}\rangle
\end{align*}
for any $u_{(1)}\in T(W_1)$, $u_{(2)}\in T(W_2)$, and $u_{(3)}\in T(W_3)$. Since 
$1^{\widetilde{H}}$ is invertible, this completes the proof of the theorem.
\end{proof}

\subsection{The associativity isomorphisms in $\imzero-\mathbf{mod}$ and their 
restrictions to top levels}

We now recall from \cite{H1} (see also \cite{HLZ6}) the construction of the 
associativity isomorphisms in the tensor category $\imzero-\mathbf{mod}$, using 
associativity of intertwining operators. We recall from Subsection 
\ref{sub:apptoghat} that for $\imzero$-modules $W_1$ and $W_2$ and 
$z\in\C^\times$, we have chosen
\begin{equation*}
 (W_1\boxtimes_{P(z)} W_2,\boxtimes_{P(z)})=(W_1\boxtimes_{P(1)} 
W_2=S(A(W_1)\otimes_{A(\imzero)} T(W_2)), \mathcal{Y}^{1,2}_{\boxtimes}(\cdot, 
e^{\mathrm{log}\,z})\cdot),
\end{equation*}
where $\mathcal{Y}^{1,2}_\boxtimes$ is the intertwining operator of type 
$\binom{S(A(W_1)\otimes_{A(\imzero)} T(W_2))}{W_1\,\,\,\,\, W_2}$ induced from 
the identity map on $A(W_1)\otimes_{A(\imzero)} T(W_2)$. In particular, we have 
chosen the tensor products $W_1\boxtimes_{P(z)} W_2$ to be the same as modules 
for all $z\in\C^\times$. We also recall that we have chosen $\boxtimes_{P(1)}$ 
as the tensor product bifunctor for $\imzero-\mathbf{mod}$ as a tensor category.

Now suppose $W_1$, $W_2$, and $W_3$ are three $\imzero$-modules. We have the 
tensor product intertwining operators $\mathcal{Y}^{2,3}_\boxtimes$ of type 
$\binom{W_1\boxtimes_{P(1)} W_2}{W_1\,\,\,W_2}$, $\mathcal{Y}^{1,2\otimes 
3}_\boxtimes$ of type 
$\binom{W_1\boxtimes_{P(1)}(W_2\boxtimes_{P(1)}W_3)}{W_1\,\,\,W_2\boxtimes_{P(1)
}W_3}$, $\mathcal{Y}^{1,2}_\boxtimes$ of type 
$\binom{W_1\boxtimes_{P(1)}W_2}{W_1\,\,\, W_2}$, and $\mathcal{Y}^{1\otimes 2, 
3}_\boxtimes$ of type 
$\binom{(W_1\boxtimes_{P(1)}W_2)\boxtimes_{P(1)}W_3}{W_1\boxtimes_{P(1)}W_2\,\,\,W_3}$. Then part 4 of Theorem \ref{intwopassoc} and Remark \ref{assocunique} 
imply that there is a unique intertwining operator $\mathcal{Y}$ of type 
$\binom{(W_1\boxtimes_{P(1)}W_2)\boxtimes_{P(1)} 
W_3}{W_1\,\,\,W_2\boxtimes_{P(1)} W_3}$ such that
\begin{equation}\label{tensintwopassoc}
 \langle w_{(4)}',\mathcal{Y}^{1\otimes 
2,3}_\boxtimes(\mathcal{Y}^{1,2}_\boxtimes(w_{(1)},z_1-z_2)w_{(2)}, 
z_2)w_{(3)}\rangle =\langle 
w_{(4)}',\mathcal{Y}(w_{(1)},z_1)\mathcal{Y}^{2,3}_\boxtimes(w_{(2)},z_2)w_{(3)}
\rangle
\end{equation}
for any $w_{(1)}\in W_1$, $w_{(2)}\in W_2$, $w_{(3)}\in W_3$, $w_{(4)}'\in W_4$, 
and $z_1,z_2\in\C^\times$ satisfying
\begin{equation}\label{z1z2cond}
 \vert z_1\vert>\vert z_2\vert>\vert z_1-z_2\vert >0;
\end{equation}
 as usual, we substitute formal variables with complex numbers using the branch 
$\mathrm{log}\,z$ of the logarithm.

Now the universal property of the $P(z_1)$-tensor product implies that there is 
a unique $\imzero$-homomorphism
\begin{equation*}
 \mathcal{A}_{z_1,z_2}: W_1\boxtimes_{P(z_1)}(W_2\boxtimes_{P(1)} 
W_3)=W_1\boxtimes_{P(1)}(W_2\boxtimes_{P(1)}W_3)\rightarrow 
(W_1\boxtimes_{P(1)}W_2)\boxtimes_{P(1)}W_3
\end{equation*}
such that
\begin{equation}\label{Adef}
 \mathcal{A}_{z_1,z_2}\circ\mathcal{Y}^{1,2\otimes 3}_\boxtimes(\cdot, z_1)\cdot 
=\mathcal{Y}(\cdot,z_1)\cdot.
\end{equation}
In particular, \eqref{tensintwopassoc} implies
\begin{equation}\label{charassoc}
 \langle w_{(4)}', (w_{(1)}\boxtimes_{P(z_1-z_2)} w_{(2)})\boxtimes_{P(z_2)} 
w_{(3)}\rangle=\langle w_{(4)}', \mathcal{A}_{z_1,z_2}(w_{(1)}\boxtimes_{P(z_1)} 
(w_{(2)}\boxtimes_{P(z_2)} w_{(3)}))\rangle
\end{equation}
for any $w_{(1)}\in W_1$, $w_{(2)}\in W_2$, $w_{(3)}\in W_3$, and $w_{(4)}'\in 
W_4$. It is easy to see that projections to the conformal weight spaces of 
elements $\mathcal{Y}^{1,2\otimes 3}_\boxtimes(w_{(1)},z_1)w_{(2,3)}$ for 
$w_{(1)}\in W_1$, $w_{(2,3)}\in W_2\boxtimes_{P(1)} W_3$ span 
$W_1\boxtimes_{P(z_1)}(W_2\boxtimes_{P(1)} 
W_3)=W_1\boxtimes_{P(1)}(W_2\boxtimes_{P(1)}W_3)$ (see for instance Lemma 4.9 of 
\cite{H1} or Proposition 4.23 in \cite{HLZ3}). This means that 
$W_1\boxtimes_{P(1)}(W_2\boxtimes_{P(1)}W_3)$ is also spanned by coefficients of 
powers of $x$ in series $\mathcal{Y}^{1,2\otimes 
3}_\boxtimes(w_{(1)},x)w_{(2,3)}$ for $w_{(1)}\in W_1$, $w_{(2,3)}\in 
W_2\boxtimes_{P(1)} W_3$. Thus the homomorphism $\mathcal{A}_{z_1,z_2}$ is 
completely determined by the condition
\begin{equation*}
 \mathcal{A}_{z_1,z_2}\circ\mathcal{Y}^{1,2\otimes 
3}_\boxtimes(\cdot,x)\cdot=\mathcal{Y}(\cdot,x)\cdot,
\end{equation*}
which follows from Proposition \ref{opmapiso} and \eqref{Adef}. Consequently, 
since $\mathcal{Y}^{1,2\otimes 3}_\boxtimes$ and $\mathcal{Y}$ are independent 
of $z_1$ and $z_2$, so is $\mathcal{A}_{z_1,z_2}$. In fact, 
$\mathcal{A}_{z_1,z_2}$ is the associativity isomorphism 
$\mathcal{A}_{W_1,W_2,W_3}$.
\begin{rema}
 The invertibility of $\mathcal{A}$ follows from part 3 Theorem 
\ref{intwopassoc}, which implies that there is a homomorphism
 \begin{equation*}
  \mathcal{A}^{-1}_{W_1,W_2,W_3}: 
(W_1\boxtimes_{P(1)}W_2)\boxtimes_{P(1)}W_3\rightarrow 
W_1\boxtimes_{P(1)}(W_2\boxtimes_{P(1)}W_3)
 \end{equation*}
which satisfies
\begin{equation*}
 \langle w_{(4)}', \mathcal{A}^{-1}_{W_1,W_2,W_3}((w_{(1)}\boxtimes_{P(z_1-z_2)} 
w_{(2)})\boxtimes_{P(z_2)} w_{(3)})\rangle=\langle w_{(4)}', 
w_{(1)}\boxtimes_{P(z_1)} (w_{(2)}\boxtimes_{P(z_2)} w_{(3)})\rangle
\end{equation*}
for any $z_1,z_2\in\C^\times$ which satisfy \eqref{z1z2cond}. This condition 
together with \eqref{charassoc} shows that $\mathcal{A}^{-1}_{W_1,W_2,W_3}$ is 
in fact the inverse of $\mathcal{A}_{W_1,W_2,W_3}$.
\end{rema}

Now suppose $U_1$, $U_2$, and $U_3$ are objects of $\mathbf{D}(\g,\ell)$. We 
will use Theorem \ref{mainassoctheo} and \eqref{charassoc} to describe the 
associativity isomorphism $\mathcal{A}_{U_1,U_2,U_3}$. First we introduce some 
notation: we use $W^{(\ell)}_{U_1, (U_2, U_3)}$ to denote the kernel of the 
composition of the natural projections
\begin{equation*}
 U_1\otimes U_2\otimes U_3\rightarrow U_1\otimes (U_2\boxtimes U_3)\rightarrow 
U_1\boxtimes (U_2\boxtimes U_3),
\end{equation*}
and similarly we use $W^{(\ell)}_{(U_1, U_2), U_3}$ to denote the kernel of the 
composition of the natural projections
\begin{equation*}
 U_1\otimes U_2\otimes U_3\rightarrow (U_1\boxtimes U_2)\otimes U_3\rightarrow 
(U_1\boxtimes U_2)\boxtimes U_3. 
\end{equation*}
Thus using the natural isomorphism of Proposition \ref{identboxtimes}, we have 
natural isomorphisms
\begin{equation*}
\Psi_{U_1,(U_2,U_3)}: (U_1\otimes U_2\otimes 
U_3)/W^{(\ell)}_{U_1,(U_2,U_3)}\rightarrow 
T(S(U_1)\boxtimes_{P(1)}(S(U_2)\boxtimes_{P(1)} S(U_3)))
\end{equation*}
and
\begin{equation*}
 \Psi_{(U_1,U_2),U_3}: (U_1\otimes U_2\otimes 
U_3)/W^{(\ell)}_{(U_1,U_2),U_3}\rightarrow 
T((S(U_1)\boxtimes_{P(1)}S(U_2))\boxtimes_{P(1)} S(U_3)).
\end{equation*}
Then from the discussion in Subsection \ref{sub:tensequiv}, we can identify 
$\mathcal{A}_{U_1,U_2,U_3}$ with the isomorphism
\begin{align*}
 \Psi_{(U_1,U_2),U_3}^{-1}\circ 
T(\mathcal{A}_{S(U_1),S(U_2),S(U_3)})\circ\Psi_{U_1,(U_2,U_3)}: (U_1\otimes & 
U_2\otimes U_3)/W^{(\ell)}_{U_1,(U_2,U_3)}\nonumber\\
 & \rightarrow (U_1\otimes U_2\otimes U_3)/W^{(\ell)}_{(U_1,U_2),U_3}
\end{align*}

Using the notation of Subsection \ref{sub:mainassoctheo} and taking $W_i=S(U_i)$ 
for $i=1,2,3$, $\mathcal{Y}^1=\mathcal{Y}^{1\otimes 2,3}_\boxtimes$, 
$\mathcal{Y}^2=\mathcal{Y}^{1,2}_\boxtimes$, 
$\mathcal{Y}_1=\mathcal{A}_{S(U_1),S(U_2),S(U_3)}\circ\mathcal{Y}^{1,2\otimes 
3}_\boxtimes$, and $\mathcal{Y}_2=\mathcal{Y}^{2,3}_\boxtimes$, we see from the 
definitions (in particular \eqref{identboxtimesiso}) that
\begin{equation*}
 \Psi_{(U_1,U_2),U_3}(u_{(1)}\otimes u_{(2)}\otimes 
u_{(3)}+W^{(\ell)}_{(U_1,U_2),U_3})=o^{\mathcal{Y}^{1\otimes 
2,3}_\boxtimes}_{-1}(o^{\mathcal{Y}^{1,2}_\boxtimes}_{-1}(u_{(1)}\otimes 
u_{(2)})\otimes u_{(3)})
\end{equation*}
and
\begin{equation*}
 \Psi_{U_1,(U_2,U_3)}(u_{(1)}\otimes u_{(2)}\otimes 
u_{(3)}+W^{(\ell)}_{U_1,(U_2,U_3)})=o^{\mathcal{Y}^{1,2\otimes 
3}_\boxtimes}_{-1}(u_{(1)}\otimes 
o^{\mathcal{Y}^{2,3}_\boxtimes}_{-1}(u_{(2)}\otimes u_{(3)})),
\end{equation*}
for any $u_{(1)}\in U_1$, $u_{(2)}\in U_2$, $u_{(3)}\in U_3$, so that
\begin{equation*}
 \langle F_{It}(u_{(4)}'), u_{(1)}\otimes u_{(2)}\otimes u_{(3)}\rangle=\langle 
u_{(4)}', \Psi_{(U_1,U_2),U_3}(u_{(1)}\otimes u_{(2)}\otimes 
u_{(3)}+W^{(\ell)}_{(U_1, U_2), U_3})\rangle
\end{equation*}
and
\begin{align*}
 \langle F_{Pr}(u_{(4)}'), & u_{(1)}\otimes u_{(2)}\otimes 
u_{(3)}\rangle\nonumber\\
 &=\langle u_{(4)}', 
(T(\mathcal{A}_{S(U_1),S(U_2),S(U_3)})\circ\Psi_{U_1,(U_2,U_3)})(u_{(1)}\otimes 
u_{(2)}\otimes u_{(3)}+W^{(\ell)}_{U_1, (U_2, U_3)})\rangle
\end{align*}
 for any $u_{(4)}'\in T((W_1\boxtimes_{P(1)}W_2)\boxtimes_{P(1)}W_3)^*$. Then we 
immediately obtain from Theorem \ref{mainassoctheo} the following description of 
the associativity isomorphisms in $\mathbf{D}(\g,\ell)$:
\begin{theo}\label{associsotoplevel}
 Under the identifications of $U_1\boxtimes (U_2\boxtimes U_3)$ with 
$(U_1\otimes U_2\otimes U_3)/W^{(\ell)}_{U_1,(U_2,U_3)}$ and $(U_1\boxtimes 
U_2)\boxtimes U_3$ with $(U_1\otimes U_2\otimes 
U_3)/W^{(\ell)}_{(U_1,U_2),U_3}$,
 \begin{equation*}
  \mathcal{A}_{U_1,U_2,U_3}(u_{(1)}\otimes u_{(2)}\otimes 
u_{(3)}+W^{(\ell)}_{U_1, (U_2, U_3)})=\Phi^*_{KZ}(u_{(1)}\otimes u_{(2)}\otimes 
u_{(3)})+W^{(\ell)}_{(U_1, U_2), U_3},
 \end{equation*}
for any $u_{(1)}\in U_1$, $u_{(2)}\in U_2$, and $u_{(3)}\in U_3$, where 
$\Phi_{KZ}^*$ is the automorphism of $U_1\otimes U_2\otimes U_3$ adjoint to 
$\Phi_{KZ}$.
\end{theo}
\begin{rema}
 Note that one non-trivial implication of Theorem \ref{associsotoplevel} is 
that 
 \begin{equation*}
 \Phi_{KZ}^*(W^{(\ell)}_{U_1, (U_2, U_3)})=W^{(\ell)}_{(U_1, U_2), U_3}.
\end{equation*}
It seems to be difficult to prove this directly, without using the associativity 
isomorphisms in $\imzero-\mathbf{mod}$. Also, the fact that $W^{(\ell)}_{U_1, 
(U_2, U_3)}$ and $W^{(\ell)}_{(U_1, U_2), U_3}$ are generally distinct subspaces 
of $U_1\otimes U_2\otimes U_3$ shows why the associativity isomorphisms in 
$\mathbf{D}(\g,\ell)$ must be induced by ordinarily non-trivial isomorphisms 
such as $\Phi_{KZ}^*$.
 \end{rema}

 \paragraph{Acknowledgments}
I would like to thank Yi-Zhi Huang for many helpful discussions and suggestions 
regarding this work, and I would like to thank the referees for suggesting corrections and enhancements. This version of the paper has been accepted for publication in \textit{Communications in Mathematical Physics}; the final publication is available at Springer via http://dx.doi.org/10.1007/s00220-016-2683-y.

\end{document}